\documentclass[a4paper,11pt,english]{smfart}

\usepackage{amssymb}
\usepackage{url}
\usepackage[T1]{fontenc}
\usepackage[cal=boondoxo]{mathalfa}
\usepackage{fancybox}
\usepackage{slashbox,booktabs,amsmath}
\usepackage{makecell}

\usepackage[leftbars]{changebar}
\usepackage[textwidth=1in]{todonotes}
\usepackage[vcentermath,enableskew]{youngtab}

\usepackage{bbm}

\usepackage{palatino}
\usepackage{rotating}
\usepackage{graphicx}

\usepackage{enumerate}
\usepackage{enumitem}

\usepackage{color}
\definecolor{shadecolor}{gray}{0.90}

\def\bfit{\bfseries\itshape}

\def\proofname{D\'emonstration}

\input xypic
\xyoption{all}
\xyoption{arc}

\makeindex

\def\indexname{Index des notations}

\newtheorem{theo}{Theorem}[section]
\def\thetheo{\thesection.\arabic{theo}}
\newtheorem{prop}[theo]{Proposition}

\newtheorem{conj}[theo]{Conjecture}

\renewcommand\theequation{\thetheo}
\def\equat{\refstepcounter{theo}\begin{equation}}
\def\endequat{\end{equation}}

\renewcommand\thetable{\thetheo}

\renewcommand\thesection{\arabic{section}}
\renewcommand\thesubsection{\thesection.\Alph{subsection}}

\def\AG{{\mathfrak A}}    
    \def\BM{{\mathbb{B}}}
\def\CG{{\mathfrak C}}    \def\CM{{\mathbb{C}}}

\def\FG{{\mathfrak F}}    \def\FM{{\mathbb{F}}}

  \def\lG{{\mathfrak l}}

    \def\PM{{\mathbb{P}}}
    \def\QM{{\mathbb{Q}}}
    \def\RM{{\mathbb{R}}}
\def\SG{{\mathfrak S}}  \def\sG{{\mathfrak s}}

  \def\zG{{\mathfrak z}}  \def\ZM{{\mathbb{Z}}}

  \def\ab{{\mathbf a}}  \def\AC{{\mathcal{A}}}
\def\Bb{{\mathbf B}}    
\def\Cb{{\mathbf C}}  \def\cb{{\mathbf c}}  \def\CC{{\mathcal{C}}}
\def\Db{{\mathbf D}}    
  \def\eb{{\mathbf e}}  
    \def\FC{{\mathcal{F}}}
\def\Gb{{\mathbf G}}    \def\GC{{\mathcal{G}}}
\def\Hb{{\mathbf H}}  \def\hb{{\mathbf h}}  \def\HC{{\mathcal{H}}}

\def\Lb{{\mathbf L}}    \def\LC{{\mathcal{L}}}

    \def\OC{{\mathcal{O}}}
\def\Pb{{\mathbf P}}    \def\PC{{\mathcal{P}}}
  \def\qb{{\mathbf q}}  
    \def\RC{{\mathcal{R}}}
\def\Sb{{\mathbf S}}  \def\sb{{\mathbf s}}  \def\SC{{\mathcal{S}}}
\def\Tb{{\mathbf T}}    \def\TC{{\mathcal{T}}}
\def\Ub{{\mathbf U}}  \def\ub{{\mathbf u}}  
    
    \def\WC{{\mathcal{W}}}
\def\Xb{{\mathbf X}}    \def\XC{{\mathcal{X}}}
\def\Yb{{\mathbf Y}}    \def\YC{{\mathcal{Y}}}
\def\Zb{{\mathbf Z}}    \def\ZC{{\mathcal{Z}}}

    \def\BCB{{\boldsymbol{\mathcal{B}}}}
\def\Crm{{\mathrm{C}}}    
\def\Drm{{\mathrm{D}}}

\def\Grm{{\mathrm{G}}}    
\def\Hrm{{\mathrm{H}}}

\def\Lrm{{\mathrm{L}}}    
    
\def\Nrm{{\mathrm{N}}}    
    \def\OCB{{\boldsymbol{\mathcal{O}}}}

\def\Srm{{\mathrm{S}}}    
\def\Trm{{\mathrm{T}}}

\def\Wrm{{\mathrm{W}}}    
    \def\XCB{{\boldsymbol{\mathcal{X}}}}
    
\def\Zrm{{\mathrm{Z}}}

\def\Mti{{\tilde{M}}}

          \def\jba{{\bar{j}}}

          \def\vba{{\bar{v}}}

\def\a{\alpha}
\def\b{\beta}
\def\g{\gamma}
\def\G{\Gamma}
\def\d{\delta}

\def\e{\varepsilon}
\def\ph{\varphi}
\def\l{\lambda}
\def\L{\Lambda}

\def\o{\omega}
\def\O{\Omega}
\def\r{\rho}
\def\s{\sigma}

\def\t{\tau}

\def\z{\zeta}

          \def\chit{{\tilde{\chi}}}

\def\mub{{\boldsymbol{\mu}}}

\def\Omeb{{\boldsymbol{\Omega}}}

\def\rhob{{\boldsymbol{\rho}}}

\DeclareMathOperator{\cus}{{\mathrm{cus}}}
\DeclareMathOperator{\Cus}{{\mathrm{Cus}}}
\DeclareMathOperator{\diag}{{\mathrm{diag}}}

\DeclareMathOperator{\End}{{\mathrm{End}}}

\DeclareMathOperator{\Id}{{\mathrm{Id}}}

\DeclareMathOperator{\im}{{\mathrm{Im}}}

\DeclareMathOperator{\Ind}{{\mathrm{Ind}}}

\DeclareMathOperator{\Irr}{{\mathrm{Irr}}}
\DeclareMathOperator{\Ker}{{\mathrm{Ker}}}

\DeclareMathOperator{\Part}{{\mathrm{Part}}}

\DeclareMathOperator{\Res}{{\mathrm{Res}}}

\DeclareMathOperator{\Spec}{{\mathrm{Spec}}}

\DeclareMathOperator{\uni}{{\mathrm{uni}}}
\DeclareMathOperator{\val}{{\mathrm{val}}}

\def\to{\rightarrow}
\def\longto{\longrightarrow}
\def\injto{\hookrightarrow}

\def\fonction#1#2#3#4#5{\begin{array}{rccc}
{#1} : & {#2} & \longto & {#3}  \\
& {#4} & \longmapsto & {#5} 
\end{array}}

\def\fonctio#1#2#3#4{\begin{array}{ccc}
{#1} & \longto & {#2} \\
{#3} & \longmapsto & {#4} 
\end{array}}

\def\vide{\varnothing}
\def\DS{\displaystyle}
\def\SS{\scriptstyle}
\def\SSS{\scriptscriptstyle}

\def\finl{~$\blacksquare$}

\def\lexp#1#2{\kern\scriptspace\vphantom{#2}^{#1}\kern-\scriptspace#2}
\def\le{\hspace{0.1em}\mathop{\leqslant}\nolimits\hspace{0.1em}}
\def\ge{\hspace{0.1em}\mathop{\geqslant}\nolimits\hspace{0.1em}}

\mathchardef\inferieur="321E
\mathchardef\superieur="321F

\def\eqna{\begin{eqnarray*}}
\def\endeqna{\end{eqnarray*}}

\def\maxi{{\mathrm{max}}}

\def\itemth#1{\item[${\mathrm{(#1)}}$]}

\def\gfp{{\FM_{\! p}}}
\def\gfq{{\FM_{\! q}}}

\def\qlb{{\overline{\QM}_{\! \ell}}}

\catcode`\@=11
\long\def\@car#1#2\@nil{#1}
\long\def\@first#1#2{#1}
\long\def\@second#1#2{#2}
\long\def\ifempty#1{\expandafter\ifx\@car#1@\@nil @\@empty
  \expandafter\@first\else\expandafter\@second\fi}
\catcode`\@=12

\def\GL{{\mathrm{GL}}}

\renewcommand{\Ref}{{\mathrm{Ref}}}

\def\boitegrise#1#2{\begin{centerline}{\fcolorbox{black}{shadecolor}{~
    \begin{minipage}[t]{#2}{\vphantom{~}#1\vphantom{$A_{\DS{A_A}}$}}
            \end{minipage}~}}\end{centerline}\medskip}

\def\ve{{\SSS{\vee}}}

\def\surto{\twoheadrightarrow}
\def\module{{\text{-}}{\mathrm{mod}}}

\theoremstyle{remark}
\newtheorem{rema}[theo]{Remark}

\newtheorem{commentary}[theo]{Commentary}
\newtheorem{exemple}[theo]{Example}

\theoremstyle{plain}

\def\Frac{{\mathrm{Frac}}}

\def\BIL{LR}
\def\GAUCHE{L}
\def\CAR{CAR}
\def\FAM{FAM}

\def\reg{{\mathrm{reg}}}

\def\euler{{\eb\ub}}

\def\calo{{\Crm\Mrm}}
\def\eval{{\mathrm{ev}}}

\def\xyinj{\ar@{^{(}->}}
\def\xysur{\ar@{->>}}

\def\isomorphisme#1{{\boldsymbol{[}}\hskip0.5mm #1\hskip0.5mm {\boldsymbol{]}}}

\bigskip

\def\petitespace{\vphantom{$\DS{\frac{\DS{A^A}}{\DS{A_A}}}$}}

\def\kl{{\mathrm{KL}}}

\def\cyclo{{\mathrm{cyc}}}

\makeatletter
\def\hlinewd#1{%
\noalign{\ifnum0=`}\fi\hrule \@height #1 %
\futurelet\reserved@a\@xhline}
\makeatother

\newlength\epaisLigne

\usepackage{multirow}

\newcommand{\longiso}{\stackrel{\sim}{\longrightarrow}}

\def\schur{{\mathrm{sch}}}

\makeatletter
\def\hlinewd#1{%
\noalign{\ifnum0=`}\fi\hrule \@height #1 %
\futurelet\reserved@a\@xhline}
\makeatother

\usepackage{multirow}

\def\petitespace{\vphantom{$\DS{\frac{\DS{A^A}}{\DS{A_A}}}$}}

\usepackage{lscape}

\newcommand{\longsurto}{\relbar\joinrel\twoheadrightarrow}
\newcommand{\longinjto}{\lhook\joinrel\longrightarrow}

\def\red{{\mathrm{red}}}
\def\nor{{\mathrm{nor}}}
\def\la{\langle}
\def\ra{\rangle}
\def\gr{\operatorname{gr}\nolimits}
\def\Rees{\operatorname{Rees}\nolimits}
\def\codim{\operatorname{codim}\nolimits}
\def\cod{{\mathrm{cod}}}
\def\pt{{\mathbf{pt}}}

\def\gr{{\mathrm{gr}}}

\def\lus{{\mathrm{Lus}}}
\def\kl{{\mathrm{KL}}}
\def\uni{{\mathrm{un}}}

\def\unip{{\mathrm{Unip}}}
\def\unipc{{\mathcal{U\!n\!i\!p}}}
\def\refl{{\mathrm{ref}}}

\def\Fam{{\mathrm{Fam}}}
\def\calo{{\mathrm{CM}}}
\def\fix{{\mathrm{fix}}}
\def\hecke{{\mathrm{Hec}}}
\def\cyclo{{\mathrm{cyc}}}

\def\CUSC{{\mathcal{C\!u\!s}}}

\def\maxi{{\mathrm{max}}}
\def\irrc{{\mathcal{I\!r\!r}}}

\def\groups{{\GC\!\mathcal{r\!o\!u\!p\!s}}}
\def\ksp{{k_{\mathrm{sp}}}}
\def\kspo{{k_{\mathrm{sp}}^\circ}}
\def\Famuni{\mathcal{F\!a\!m}_\uni}
\def\cus{{\mathrm{cus}}}
\def\degb{{\mathbf{deg}}}

\def\symp{{\mathcal{S\!y\!m\!p}}}
\def\Nrmov{{\overline{\mathrm{N}}}}
\def\leaf{\SC}
\def\CUSC{{\mathcal{C\!u\!s}}}

\makeatletter
\let\original@addcontentsline\addcontentsline
\newcommand{\dummy@addcontentsline}[3]{}
\newcommand{\DeactivateToc}{\let\addcontentsline\dummy@addcontentsline}
\newcommand{\ActivateToc}{\let\addcontentsline\original@addcontentsline}
\makeatother

\begin{document}


\title{Calogero-Moser spaces vs unipotent representations}

\author{{\sc C\'edric Bonnaf\'e}}
\address{IMAG, Universit\'e de Montpellier, CNRS, Montpellier, France} 

\makeatletter
\email{cedric.bonnafe@umontpellier.fr}
\makeatother

\date{\today}

\thanks{The author is partly supported by the ANR: 
Projects No ANR-16-CE40-0010-01 (GeRepMod) and ANR-18-CE40-0024-02 (CATORE).}

\pagestyle{myheadings}

\markboth{\sc C. Bonnaf\'e}{\sc Calogero-Moser spaces 
vs unipotent representations}

\maketitle

\vskip-1cm 
\centerline{\it To George Lusztig, with admiration}

\vskip1cm

\begin{abstract}
Lusztig's classification of unipotent representations of finite reductive 
groups depends only on the associated Weyl group $W$ (and the automorphism that the Frobenius automorphism 
induces on $W$). All the structural questions (families, 
Harish-Chandra series, partition into blocks...) have an answer in a combinatorics that 
can be entirely built directly from $W$. 

Over the years, we have noticed that the same combinatorics seems 
to be encoded in the Poisson geometry of a Calogero-Moser space associated 
with $W$ (families should correspond to ${\mathbb{C}}^\times$-fixed 
points, Harish-Chandra series should correspond to symplectic leaves, blocks 
should correspond to symplectic leaves in the fixed point subvariety 
under the action of a root of unity). 

The aim of this survey is to gather all these observations, state precise 
conjectures and provide general facts and examples supporting these conjectures.
\end{abstract}

For this introduction, let us focus on the case where $\Gb$ is a {\it split} 
reductive group over a finite field with $q$ elements $\gfq$ and let $G=\Gb(\gfq)$ 
be the finite group consisting of $\gfq$-rational points. Let $W$ denote 
the Weyl group of $\Gb$ and let $\ZC$ denote the Calogero-Moser space 
associated with $W$ {\it at equal parameters} (recall that it is an affine Poisson 
variety endowed with a $\CM^\times$-action~\cite[\S{4}]{EG}). Let $\unip(G)$ denote the 
set of irreducible unipotent characters of $G$ (as defined by Lusztig). 
A consequence of a conjecture of Gordon-Martino (2007,~\cite{gordon martino}) 
is that the fixed point set $\ZC^{\CM^\times}$ should be in bijection with 
the set of {\it Lusztig families} of $\unip(G)$. This first link 
was the starting point of our interest in 
the geometry of Calogero-Moser spaces. 

In 2008, Gordon, following works of Haiman, obtained in type $A$ a parametrization 
of the irreducible components of the fixed point subvariety $\ZC^{\mub_d}$ 
by $d$-cores of partitions. This fits perfectly with the 
partition of irreducible unipotent representations of $\Gb\Lb_n(\gfq)$ into 
$d$-Harish-Chandra series (defined by Brou\'e-Malle-Michel~\cite{BMM}). 
In 2011, Bellamy and Losev (independently) obtained a parametrization 
{\it \`a la Harish-Chandra} of symplectic leaves of $\ZC$. 
In 2013, Rouquier and the author~\cite{pacific} constructed 
partitions of $W$ into left, right and two-sided Calogero-Moser cells 
and conjectured they coincide with Kazhdan-Lusztig cells. 

From then, the author has worked (with many different authors) on representations 
of Cherednik algebras at $t=0$ and the geometry of Calogero-Moser spaces 
(see~\cite{pacific, calogero, bonnafe diedral, bonnafe maksimau, bonnafe shan}), 
for, as main motivation, understanding these strange analogies between the geometry 
of Calogero-Moser spaces and 
the representation theory of finite reductive groups. Over the years, 
the author has enriched these coincidences with several examples but has never  
exposed them in a paper. This is the aim of this survey to present them, 
state precise conjectures, and provide a list of examples that support 
these conjectures. A main reason for waiting for such a long time is that we 
needed to establish some theoretical background on Calogero-Moser space to state 
precise conjectures: 
this is done in~\cite{auto}, where we generalize some results of 
Bellamy~\cite{bellamy cuspidal} and Losev~\cite{losev} on symplectic 
leaves. We also needed some general results (cohomology, fixed points, 
regular automorphisms) in accordance with these
conjectures~\cite{bonnafe shan, bonnafe maksimau, regular}. 

Let us explain one of the strangest (and most convincing) coincidences. 
Let $\ell$ be a prime number not dividing $q$ and assume for simplicity 
that $\ell$ does not divide $|W|$. We denote by 
$d$ the order of $q$ modulo $\ell$. Then each $\ell$-block $B$ 
of $\unip(G)$ should correspond to a {\it symplectic leaf} $\leaf_B$ 
of the fixed point subvariety $\ZC^{\mub_d}$ of $\ZC$ in such a way that:
\begin{itemize}
\item[$\bullet$] On one hand, the {\it $d$-Harish-Chandra theory} 
of Brou\'e-Malle-Michel~\cite{BMM} associates to $B$ a complex reflection 
group $\WC_B$ whose irreducible characters are in bijection with $B$. 
Moreover, Brou\'e-Malle-Michel also associate to $B$ a Deligne-Lusztig 
variety $\XC_B$ and a parameter $k_B$ and conjecture that the endomorphism 
algebra of the $\ell$-adic cohomology of $\XC_B$ is isomorphic to a Hecke 
algebra of $\WC_B$ with parameter $k_B$. This association is motivated 
by Brou\'e's abelian defect conjecture, and its geometric version 
for finite reductive groups~\cite[\S{6}]{broue iso} (see also~\cite{BMM, broue malle}). 

\item[$\bullet$] On the other hand, an analogue of a {\it $d$-Harish-Chandra 
theory for symplectic leaves} of $\ZC^{\mub_d}$ developed by the author~\cite{auto} 
(extending earlier works of Bellamy~\cite{bellamy cuspidal} and Losev~\cite{losev} 
which deal with the case where $d=1$) associates to $\leaf_B$ a finite linear group 
$\WC_B'$ and a parameter $k_B'$. We conjecture~\cite[Conj.~B]{auto} that 
the normalization $\overline{\leaf}_B^\nor$ of the closure of the symplectic leaf $\leaf_B$ 
is the Calogero-Moser space for the pair $(\WC_B',k_B')$.

\item[$\bullet$] The main intriguing observation is that, in the cases 
where computations can be done, $\WC_B$ is a subgroup of $\WC_B'$ 
(in fact, in most cases, $\WC_B=\WC_B'$) and the parameter $k_B$ is the restriction 
of the parameter $k_B'$. Our main conjecture is that this holds in general.
\end{itemize}
So, important features of the $\ell$-modular representation 
theory of $G$ seem to be encoded in the (Poisson) geometry 
of the affine variety $\ZC^{\mub_d}$ (where $d$ and $\ell$ are linked 
by the fact that $q$ is a primitive $d$-th root of unity modulo $\ell$). 
Moreover, this correspondence seems to carry more properties, as explained in Section~\ref{sec:main}. 
To support our conjectures, we have the following examples available:
\begin{itemize}
\item[$\bullet$] We are able to prove most of them if $W$ is of type $A$ 
(see Section~\ref{sec:type a}).

\item[$\bullet$] They hold if $W$ is of type $B_2$ or $G_2$ and $d$ is the Coxeter number 
(see Section~\ref{sec:2}).

\item[$\bullet$] They hold if $W$ is of classical type and $d=1$ (classical Harish-Chandra 
theory); see Section~\ref{sec:classique}.

\item[$\bullet$] In the {\it regular} case (see~\S\ref{sub:regular} for the definition), 
we have a general result on Calogero-Moser spaces (see Theorem~\ref{theo:chi-tau}) 
which fits with observations made on the unipotent representations side (see Example~\ref{ex:regular}).

\item[$\bullet$] Our conjectures are compatible with Ennola duality.
\end{itemize}


\bigskip

The text is organized as follows. An introductory part presents the set-up 
and the notation involved all along the text. We summarize in the first part 
some general questions on the geometry of Calogero-Moser spaces 
(cohomology, geometry of symplectic leaves...), already 
contained in~\cite{calogero, bonnafe shan, auto, regular}. 
The second part is a crash-course on unipotent representations of finite reductive 
groups (we hope it is understandable for non-specialists). The third part 
contains an explanation of the notion of {\it genericity} and also a detailed 
exposition of the different coincidences (stated as conjectures) we expect 
between the Poisson geometry of $\ZC$ and the representation theory of $G$: this 
is the heart of this survey. The fourth part contains several very explicit 
examples confirming 
the conjectures. The last (short) part is an invitation to the {\it Spetses} theory 
of Brou\'e-Malle-Michel~\cite{spetses 1, spetses 2}, which have connections 
with the theme of this paper.

\vskip-3cm

\tableofcontents

\noindent{\bf Commentary.} 
Recently, Riche-Williamson~\cite{riche} provided a geometric proof of 
the {\it linkage principle}~\cite{verma, humphreys, jantzen, andersen}: 
in their construction, blocks of the category 
of rational representations of $\Gb(\overline{\FM}_{\! q})$ are 
in bijection with the irreducible components of ${\mathcal{G\!r}}^{\mub_p}$, where 
$p$ is the prime number dividing $q$ and ${\mathcal{G\!r}}$ is the (complex) 
affine Grassmannian of the (complex) Langlands dual group to $\Gb$. 
So our observation has the same flavor as Riche-Williamson result: the blocks 
of some category of representations are controlled by the geometry of fixed 
points under the action of a group of roots of unity on some variety. 
Of course, the main difference is that Riche-Williamson proved a true theorem, 
built on the geometric Satake 
equivalence~\cite{lusztig sing, ginzburg, BD, mirko} 
between representations of $\Gb(\overline{\FM}_{\! q})$ and some category of perverse 
sheaves on ${\mathcal{G\!r}}$. Our observations are conjectural, and are only concerned 
with numerical/combinatorial coincidences. 
We lack of a {\it geometric Calogero-Moser equivalence}\footnote{Recent 
works of Dudas-Rouquier relate the category of coherent sheaves on the Hilbert scheme of 
points in the plane (which is diffeomorphic to the Calogero-Moser 
space associated with the symmetric group)
and representations of finite general linear or unitary groups. This work 
is still unpublished, but the interested reader 
might look at numerous videos of some of their 
talks:\\
\tiny{\tt https://www.birs.ca/events/2017/5-day-workshops/17w5003/videos/watch/201710181031-Rouquier.html}\\
{\tt https://www.msri.org/workshops/820/schedules/23934}\\
{\tt https://www.youtube.com/watch?v=CMBVSJC6EX0}}...

\bigskip

\noindent{\bf Acknowledgements.} I wish to thank warmly the {\it Spetses} team 
(Michel Brou\'e, Olivier Dudas, Gunter Malle, 
Jean Michel and Rapha\"el Rouquier), from which 
I learnt most of what I know on representation theory of finite reductive groups, 
and for the hours and hours of passionate discussions we had together.

\part*{Set-up}

\section{General notation}

\medskip

Throughout this paper, we will abbreviate $\otimes_\CM$ as 
$\otimes$. 

If $\XC$ is a quasi-projective scheme of finite 
type over an algebraically closed field, we denote by $\XC_\red$ its reduced 
subscheme. By an algebraic variety, we mean a quasi-projective 
reduced scheme of finite type over an algebraically closed field. 
If $\XC$ is an algebraic variety, we denote by $\XC^\nor$ 
its normalization. If $\XC$ is affine we denote by $\CM[\XC]$ 
its coordinate ring.

If $\XC$ is a complex algebraic variety, we denote by $\Hrm^j(\XC)$ its $j$-th 
singular cohomology group with coefficients in $\CM$. If 
$\XC$ carries a regular action of a torus $\Tb$, 
we denote by $\Hrm_\Tb^j(\XC)$ its $j$-th $\Tb$-equivariant cohomology group 
(still with coefficients in $\CM$). Note that 
$\Hrm^{2\bullet}(\XC)=\bigoplus_{j \ge 0} \Hrm^{2j}(\XC)$ is a graded $\CM$-algebra and 
$\Hrm^{2\bullet}_\Tb(\XC)=\bigoplus_{j \ge 0} \Hrm^{2j}_\Tb(\XC)$ 
is a graded $\Hrm^{2\bullet}_\Tb(\pt)$-algebra, where $\pt=\Spec(\CM)$. 
We identify $\Hrm^{2\bullet}_{\CM^\times}(\pt)$ with $\CM[\hbar]$ in the usual way 
(note that $\Hrm_\Tb^{2j+1}(\pt)=0$ for all $j$). 
If $\YC$ is another complex variety endowed with a regular $\Tb$-action 
and if $\ph : \YC \to \XC$ is a $\Tb$-equivariant morphism of varieties, 
we denote by $\ph^* : \Hrm_\Tb^\bullet(\XC) \longto \Hrm_\Tb^\bullet(\YC)$ 
the induced morphism in equivariant cohomology.

\bigskip

\section{Finite linear group, reflections} 

\medskip

\boitegrise{{\bf Notation.} {\it We fix in this paper a finite dimensional 
$\CM$-vector space $V$ and a finite subgroup $W$ of $\GL_\CM(V)$.}}{0.75\textwidth}

\medskip

\subsection{Reflections, hyperplanes} We set $\e : W \to \CM^\times$, $w \mapsto \det(w)$ and 
$$\Ref(W)=\{s \in W~|~\dim_\CM V^s=n-1\}.$$
Note that, for the moment, we do not assume that $W=\langle \Ref(W) \rangle$. 
We identify $\CM[V]$ (resp. $\CM[V^*]$) with the symmetric 
algebra $\Srm(V^*)$ (resp. $\Srm(V)$).

We denote by $\AC$ the set of {\it reflecting hyperplanes} of $W$, namely
$$\AC=\{V^s~|~s \in \Ref(W)\}.$$
If $H \in \AC$, we denote by $W_H$ its inertia group, i.e. the group consisting of 
elements $w \in W$ such that $w(v)=v$ for all $v \in H$. 
We denote by $\a_H$ an element of $V^*$ such that 
$H=\Ker(\a_H)$ and by $\a_H^\vee$ an element of $V$ such that 
$V=H \oplus \CM \a_H^\vee$ and the line $\CM\a_H^\vee$ is $W_H$-stable.
We set $e_H=|W_H|$. Note that $W_H$ is cyclic of order $e_H$ and that 
$\Irr(W_H)=\{\Res_{W_H}^W \e^j~|~0 \le j \le e-1\}$. We denote by $\e_{H,j}$ 
the (central) primitive idempotent of $\CM W_H$ associated with the character 
$\Res_{W_H}^W \e^{-j}$, namely
$$\e_{H,j}=\frac{1}{e_H}\sum_{w \in W_H} \e(w)^j w \in \CM W_H.$$
If $\O$ is a $W$-orbit of reflecting hyperplanes, we write $e_\O$ for the 
common value of all the $e_H$, where $H \in \O$. 
We denote by $\aleph$ the set of pairs $(\O,j)$ where $\O \in \AC/W$ and 
$0 \le j \le e_\O-1$. 
The vector space of families of complex numbers 
indexed by $\aleph$ will be denoted by $\CM^\aleph$, elements 
of $\CM^\aleph$ will be called {\it parameters}. 
If $k=(k_{\O,j})_{(\O,j) \in \aleph} \in \CM^\aleph$, we 
define $k_{H,j}$ for all $H \in \O$ and $j \in \ZM$ by 
$k_{H,j}=k_{\O,j_0}$ where $\O$ is the $W$-orbit of $H$ 
and $j_0$ is the unique element of $\{0,1,\dots,e_H-1\}$ such that 
$j \equiv j_0 \mod e_H$.

\medskip

\subsection{Filtration}\label{sub:filtration}
Let $\cod : W \to \ZM_{\geqslant 0}$ be defined by 
$$\cod(w)=\codim_\CM(V^w)$$
(note that $\Ref(W)=\cod^{-1}(1)$) 
and we define a filtration $\FC_\bullet(\CM W)$ of the group algebra of $W$ as follows: let
$$\FC_j(\CM W)=\bigoplus_{\cod(w) \le j} \CM w.$$
Then 
$$\CM\Id_V=\FC_0(\CM W)\subset \FC_1(\CM W)\subset \cdots\subset\FC_n(\CM W)=\CM W=\FC_{n+1}(\CM W)= \cdots$$ 
is a filtration of $\CM W$. For any subalgebra $A$ of $\CM W$, we set $\FC_j(A)=A\cap\FC_j(\CM W)$, 
so that 
$$\CM\Id_V=\CM\FC_0(A)\subset \FC_1(A)\subset \cdots\subset\FC_n(A)=A=\FC_{n+1}(A)=\cdots$$ 
is also a filtration of $A$. Write 
\begin{align*}
&\Rees_\FC^\bullet(A)=\bigoplus_{j\geqslant 0}\hbar^j\FC_j(A)\subset \CM[\hbar] \otimes A
\qquad(\text{the {\it Rees algebra}}),\\
&\gr_\FC^\bullet(A)=\bigoplus_{j\geqslant 0}\FC_j(A)/\FC_{j-1}(A).
\end{align*}
Recall that $\gr_\FC^\bullet(A) \simeq 
\Rees_\FC^\bullet(A)/\hbar \Rees_\FC^\bullet(A)$.

\bigskip

\section{Rational Cherednik algebra at ${\boldsymbol{t=0}}$}

\medskip

\boitegrise{{\bf Notation.} {\it Throughout this paper, we fix a parameter 
$k \in \CM^\aleph$.}}{0.65\textwidth}

\bigskip

\subsection{Definition} 
We define the {\it rational Cherednik algebra} $\Hb_k$ to be the quotient 
of the algebra $\Trm(V\oplus V^*)\rtimes W$ (the semi-direct product of the tensor algebra 
$\Trm(V \oplus V^*)$ with the group $W$) 
by the relations 
\equat\label{eq:rels}
\begin{cases}
[x,x']=[y,y']=0,\\
[y,x]=\DS{\sum_{H\in\mathcal{A}} \sum_{j=0}^{e_H-1}
e_H(k_{H,j}-k_{H,j+1}) 
\frac{\langle y,\a_H \rangle \cdot \langle \a_H^\ve,x\rangle}{\langle \a_H^\ve,\a_H\rangle} \e_{H,j}},
\end{cases}
\endequat
for all $x$, $x'\in V^*$, $y$, $y'\in V$. 
Here $\la\ ,\ \ra: V\times V^*\to\CM$ is the standard pairing. 
The first commutation relations imply that 
we have morphisms of algebras $\CM[V] \to \Hb_k$ and $\CM[V^*] \to \Hb_k$. 
Recall~\cite[Theo.~1.3]{EG} 
that we have an isomorphism of $\CM$-vector spaces 
\equat\label{eq:pbw}
\CM[V] \otimes \CM W \otimes \CM[V^*] \longiso \Hb_k
\endequat
induced by multiplication (this is the so-called {\it PBW-decomposition}). 

\medskip

\begin{rema}\label{rem:parametres particuliers}
Let $(l_\O)_{\O \in \AC/W}$ be a family of complex numbers and let 
$k' \in \CM^\aleph$ be defined by $k_{\O,j}'=k_{\O,j} + l_\O$. Then 
$\Hb_k=\Hb_{k'}$. This means that there is no restriction to generality 
if we consider for instance only 
parameters $k$ such that $k_{\O,0}=0$ for all $\O$, 
or only parameters $k$ such that $k_{\O,0}+k_{\O,1}+\cdots+k_{\O,e_\O-1}=0$ 
for all $\O$ (as in~\cite{calogero}).\finl
\end{rema}

\medskip

\subsection{Calogero-Moser space}\label{sub:cm}
We denote by $\Zb_k$ the center of the algebra $\Hb_k$: it is well-known~\cite[Theo~3.3~and~Lem.~3.5]{EG} 
that $\Zb_k$ is an integral domain, which is integrally closed. Moreover, it contains 
$\CM[V]^W$ and $\CM[V^*]^W$ as subalgebras~\cite[Prop.~3.6]{gordon} 
(so it contains $\Pb=\CM[V]^W \otimes \CM[V^*]^W$), 
and it is a free $\Pb$-module of rank $|W|$. We denote by $\ZC_k$ the 
affine algebraic variety whose ring of regular functions $\CM[\ZC_k]$ is $\Zb_k$: 
this is the {\it Calogero-Moser space} associated with the datum $(V,W,k)$. 
It is irreducible and integrally closed. 

We set $\PC=V/W \times V^*/W$, so that $\CM[\PC]=\Pb$ and the inclusion 
$\Pb \injto \Zb_k$ induces a finite and flat morphism of varieties 
$$\Upsilon_k : \ZC_k \longto \PC.$$

Using the PBW-decomposition, we define a $\CM$-linear map 
$\Omeb^{\Hb_k} : \Hb_k \longto \CM W$
by 
$$\Omeb^{\Hb_k}(f w g)=f(0)g(0)w$$
for all $f \in \CM[V]$, $g \in \CM[V^*]$ and $w \in \CM W$. This map is $W$-equivariant 
with respect to the action on both sides by conjugation, so it induces a well-defined $\CM$-linear map 
$$\Omeb^k : \Zb_k \longto \Zrm(\CM W).$$
Recall from~\cite[Cor.~4.2.11]{calogero} that $\Omeb^k$ is a morphism of algebras, and that 
\equat\label{eq:lissite-omega}
\text{\it $\ZC_k$ is smooth if and only if $\Omeb^k$ is surjective.}
\endequat
The ``only if'' part is essentially due to Gordon~\cite[Cor.~5.8]{gordon} 
(but the reader must see take~\cite[Prop.~9.6.6~and~$($16.1.2$)$]{calogero} 
into account for translating Gordon's result in terms of $\Omeb^k$) while 
the ``if'' part follows from the work of Bellamy, Schedler and Thiel~\cite[Cor.~1.4]{BST}.

\bigskip

\subsection{Other parameters} 
Let $\CC$ denote the space of maps $\Ref(W) \to \CM$ 
which are constant on conjugacy classes of reflections. 
The element
$$\sum_{(\O,j) \in \aleph} \sum_{H \in \O} (k_{H,j}-k_{H,j+1}) e_H \e_{H,j}$$
of $\Zrm(\CM W)$ is supported only by reflections, so there exists 
a unique map $c_k \in \CC$ such that 
$$\sum_{(\O,j) \in \aleph} \sum_{H \in \O} (k_{H,j}-k_{H,j+1}) e_H \e_{H,j}
= \sum_{s \in \Ref(W)} (\e(s)-1) c_k(s) s.$$
Then the map $\CM^\aleph \to \CC$, $k \mapsto c_k$, is linear and surjective. 
With this notation, we have 
\equat\label{eq:cs}
[y,x] = \sum_{s \in \Ref(W)} (\e(s)-1)\hskip1mm c_k(s) 
\hskip1mm
\frac{\langle y,\a_s \rangle \cdot \langle \a_s^\ve,x\rangle}{\langle 
\a_s^\ve,\a_s\rangle}
\hskip1mm s,
\endequat
for all $y \in V$ and $x \in V^*$. Here, $\a_s=\a_{V^s}$ and $\a_s^\ve=\a_{V^s}^\ve$. 

\bigskip

\subsection{Extra-structures on the Calogero-Moser space}
The Calogero-Moser space $\ZC_k$ is endowed with a $\CM^\times$-action, 
a Poisson bracket and an Euler element 
which are described below.

\bigskip

\subsubsection{Grading, $\CM^\times$-action}
The algebra $\Trm(V\oplus V^*)\rtimes W$ can be $\ZM$-graded in such a way that the generators have the following degrees
$$
\begin{cases}
\deg(y)=-1 & \text{if $y \in V$,}\\
\deg(x)=1 & \text{if $x \in V^*$,}\\
\deg(w)=0 & \text{if $w \in W$.}
\end{cases}
$$
This descends to a $\ZM$-grading on $\Hb_k$ because the defining relations~(\ref{eq:rels}) 
are homogeneous. Since the center of a graded algebra is always graded, the subalgebra $\Zb_k$ 
is also $\ZM$-graded.  So the Calogero-Moser space $\ZC_k$ 
inherits a regular $\CM^\times$-action. Note also that
by definition $\Pb=\CM[V]^W \otimes \CM[V^*]^W$ is clearly a graded 
subalgebra of $\Zb_k$. 

\bigskip

\subsubsection{Poisson structure}
Let $t \in \CM$. One can define a deformation $\Hb_{t,k}$ of $\Hb_k$ as follows: 
$\Hb_{t,k}$ is the quotient 
of the algebra $\Trm(V\oplus V^*)\rtimes W$ 
by the relations 
\equat\label{eq:rels-1}
\begin{cases}
[x,x']=[y,y']=0,\\
[y,x]=t \la y,x \ra + \DS{\sum_{H\in\mathcal{A}} \sum_{j=0}^{e_H-1}
e_H(k_{H,j}-k_{H,j+1}) 
\frac{\langle y,\a_H \rangle \cdot \langle \a_H^\ve,x\rangle}{\langle \a_H^\ve,\a_H\rangle} \e_{H,j}},
\end{cases}
\endequat
for all $x$, $x'\in V^*$, $y$, $y'\in V$. It is well-known~\cite[Theo~1.3]{EG} 
that the PBW decomposition (as in~\eqref{eq:pbw}) still holds so 
that the family $(\Hb_{t,k})_{t \in \CM}$ is a flat deformation of $\Hb_k=\Hb_{0,k}$. 
This allows to define a Poisson bracket $\{\ ,\ \}$ on $\Zb_k$ as follows: 
if $z_1$, $z_2 \in \Zb_k$, 
we denote by $z_1^{t}$, $z_2^t$ the corresponding element of $\Hb_{t,k}$ through the 
PBW decomposition and we define 
$$\{z_1,z_2\} = \lim_{t \to 0} \frac{[z_1^t,z_2^t]}{t}.$$
Finally, note that 
\equat\label{eq:poisson equivariant}
\text{\it The Poisson bracket is $\CM^\times$-equivariant.}
\endequat

\bigskip

\subsubsection{Euler element}
Let $(y_1,\dots,y_m)$ be a basis of $V$ and let $(x_1,\dots,x_m)$ 
denote its dual basis. As in~\cite[\S{3.3}]{calogero}, we set
$$\euler=~\sum_{j=1}^m x_j y_j + \sum_{s \in \Ref(W)} \e(s) c_k(s) s 
=~\sum_{j=1}^m x_j y_j +
\sum_{H\in\AC}\sum_{j=0}^{e_H-1}e_H\ k_{H,j}\varepsilon_{H,j}.$$
Recall that $\euler$ does not depend on the choice of the basis of $V$. 
Also
\equat\label{eq:euler centre}
\euler \in \Zb_k,\qquad \Frac(\Zb_k)=\Frac(\Pb)[\euler]
\endequat
and
\equat\label{eq:euler poisson}
\{\euler,z\}=d z 
\endequat
if $z \in \Zb_k$ is homogeneous of degree $d$ 
(see for instance~\cite[Prop.~3.3.3]{calogero}).

\bigskip

%

\boitegrise{{\bf Notation.} {\it If $?$ is one of the above objects 
defined in this section ($\Hb_k$, $\ZC_k$, $\aleph$, $\AC$, $\HC_k$\dots), 
we will sometimes denote it by $?(W)$ or $?(V,W)$ if we need 
to emphasize the context.}}{0.75\textwidth}

\bigskip

\section{Reflection groups} 

\medskip

Recall that, for the moment, we did not assume that $W=\langle \Ref(W) \rangle$ (this will 
be assumed only after this section). 
Let $W_\refl=\la \Ref(W) \ra$ 
be the maximal subgroup of $W$ generated by reflections. Then the set $\AC$ depends 
only on $W_\refl$ and the finite group $W/W_\refl$ acts on $\aleph(W_\refl)$ 
and $\aleph=\aleph(W_\refl)/(W/W_\refl)$. In other words, giving an element 
$k \in \aleph$ is equivalent to giving an element $k \in \aleph(W_\refl)$ which 
is $W/W_\refl$-invariant. In this case, 
the relations $(\HC_{\! k})$ only depend on $W_\refl$. If we denote by 
$\Hb_k(W_\refl)$ the Cherednik algebra defined with $W_\refl$ instead of $W$, 
then $\Hb_k(W_\refl)$ is naturally a subalgebra of $\Hb_k$ and, as a $\CM W$-module, 
$\Hb_k=\CM W \otimes_{\CM W_\refl} \Hb_k(W_\refl)$. Note also that the finite 
group $W/W_\refl$ acts on $\Zb_k(W_\refl)$ and on $\ZC_k(W_\refl)$, and that
\equat\label{eq:cm-quotient}
\Zb_k = \Zb_k(W_\refl)^{W/W_\refl}\qquad\text{and}\qquad 
\ZC_k = \ZC_k(W_\refl)/(W/W_\refl).
\endequat
We deduce from this the following facts:

\bigskip

\begin{prop}\label{prop:ref-pasref}
Let $q : \ZC_k(W_\refl) \longto \ZC_k$ denote the quotient map. 
Then:
\begin{itemize}
\itemth{a} We have $\ZC_k^{\CM^\times}=q(\ZC_k(W_\refl)^{\CM^\times})$ 
and $q^{-1}(\ZC_k(W_\refl)^{\CM^\times}) = \ZC_k^{\CM^\times}$.

\itemth{b} The morphism $q$ induces isomorphisms 
$$q_* : \Hrm^\bullet(\ZC_k) \longiso \Hrm^\bullet(\ZC_k(W_\refl))^{W/W_\refl}$$
$$q_* : \Hrm_{\CM^\times}^\bullet(\ZC_k) \longiso 
\Hrm_{\CM^\times}^\bullet(\ZC_k(W_\refl))^{W/W_\refl}.
\leqno{\text{\it and}}$$
\end{itemize}
\end{prop}

\bigskip

\begin{proof}
(a) follows since $q$ is a finite morphism and since an action 
of $\CM^\times$ on a finite set is necessarily trivial. (b) is a 
classical property of cohomology~\cite[Theo.~III.2.4]{bredon}.
\end{proof}

\bigskip

Continuing this reduction, we denote by $W(k)$ the subgroup 
of $W_\refl$ generated by the reflections $s \in \Ref(W)$ such 
that $c_k(s) \neq 0$. It is a normal subgroup of $W$ and $W_\refl$. 
Also, the formula~(\ref{eq:cs}) shows 
that, as a $\CM W$-module, 
$\Hb_k=\CM W \otimes_{\CM W_\refl} \Hb_{k^\flat}(W_\refl)$. 
Here, $k^\flat \in \CM^{\aleph(W(k))}$ is such that 
$c^{W(k)}_{k^\flat} \in \CC(W(k))$ is the restriction 
of $c_k$ to $\Ref(W(k))$. Therefore, as above, we have 
\equat\label{eq:cm-quotient-wk}
\Zb_k = \Zb_{k^\flat}(W(k))^{W/W(k)}\qquad\text{and}\qquad 
\ZC_k = \ZC_{k^\flat}(W(k))/(W/W(k)).
\endequat
We deduce from this the following facts:

\bigskip

\begin{prop}\label{prop:ref-pasref-wk}
Let $q^\flat : \ZC_{k^\flat}(W(k)) \longto \ZC_k$ denote the quotient map. 
Then:
\begin{itemize}
\itemth{a} We have $\ZC_k^{\CM^\times}=q^\flat(\ZC_{k^\flat}(W(k))^{\CM^\times})$ 
and $(q^\flat)^{-1}(\ZC_{k^\flat}(W(k))^{\CM^\times}) = \ZC_k^{\CM^\times}$.

\itemth{b} The morphism $q^\flat$ induces isomorphisms 
$$q^\flat_* : \Hrm^\bullet(\ZC_k) \longiso 
\Hrm^\bullet(\ZC_{k^\flat}(W(k)))^{W/W(k)}$$
$$q_*^\flat : \Hrm_{\CM^\times}^\bullet(\ZC_k) \longiso 
\Hrm_{\CM^\times}^\bullet(\ZC_k(W(k)))^{W/W(k)}.
\leqno{\text{\it and}}$$
\end{itemize}
\end{prop}

\bigskip

Even though the case where $k=0$ serves as a base 
of our conjectures/questions, the really interesting 
case is when $W(k)=W$: equations~(\ref{eq:cm-quotient-wk}) and 
Proposition~\ref{prop:ref-pasref-wk} help us to recover 
properties of $\ZC_k(W)$ from those of $\ZC_k(W(k))$. 
For instance, Etingof-Ginzburg proved that, if $\ZC_k$ is smooth, 
then $W=W(k)$ (see~\cite[Prop.~3.10]{EG}).

\bigskip

\section{Braid group, Hecke algebra} 

\medskip
\def\braid{\BM}
\def\pure{\PM}

\boitegrise{{\bfit Hypothesis and notation.} {\it From now on, and until the end of 
this paper, we assume that 
$$W=\la \Ref(W) \ra$$
and we fix $k \in \CM^\aleph$. We set 
$$V_\reg=V \setminus \bigcup_{H \in \AC} H$$
and we recall that $V_\reg$ is the set of elements of $V$ 
whose stabilizer in $W$ is trivial (this is Steinberg's Theorem: 
see for instance~\cite[Theo.~4.7]{broue}). \\
\hphantom{aa} We fix $v_0 \in V_\reg$ and we denote by $\vba_0$ 
its image in $V_\reg/W$. We set 
$$\braid=\pi_1(V_\reg/W,\vba_0)\qquad\text{\it and}\qquad \pure=\pi_1(V_\reg,v_0).$$
Then the group $\braid$ (resp. $\pure$) is called the {\bfit braid group} (resp. the {\bfit pure braid group}) 
of $W$.}}{0.75\textwidth}

\medskip

\subsection{Generators of ${\boldsymbol{\braid}}$ and ${\boldsymbol{\pure}}$} 
If $H \in \AC$, we denote by $s_H$ the generator of $W_H$ of determinant 
$\z_{e_H}=\exp(2i\pi/e_H)$ and by $\sb_H$ a {\it braid reflection} around $H$ (as defined 
in~\cite[Def.~4.13]{broue}: they are called {\it generator of the monodromy} around 
$H$ in~\cite{BMR}). Through the exact sequence
\equat\label{eq:pbw-tresse}
1 \longto \pure \longto \braid \longto W \longto 1
\endequat
induced by the unramified covering $V_\reg \to V_\reg/W$, the image of $\sb_H$ is $s_H$ 
and so $\sb_H^{e_H} \in \PM$. Moreover,
\equat\label{eq:pure}
\braid=\langle (\sb_H)_{H \in \AC} \rangle\qquad\text{and}\qquad
\pure=\langle(\sb_H^{e_H})_{H \in \AC} \rangle.
\endequat

\subsection{Hecke algebra}
Let $F$ denote the number field generated by the traces of the elements of $W$ 
(it is generally called the {\it character field} of $W$). 
It is known~\cite{benard, bessis} that the algebra $FW$ is split. We denote 
by $\OC$ the ring of algebraic integers in $F$ and 
let $R=\OC[\qb^\CM]$ be the group algebra of $(\CM,+)$ over $\OC$, denoted with an exponential notation: 
namely, we have $\qb^a \qb^{a'}=\qb^{a+a'}$ for all $a$, $a' \in \CM$. We set $\qb=\qb^1$. The 
{\it Hecke algebra} with parameter $k$, denoted by $\HC_k(W)$, is the quotient of 
the group algebra $R\braid$ of $\braid$ over $R$ by the ideal generated by the elements 
$$\prod_{j=0}^{e_H-1} (\sb_H - \z_{e_H}^j \qb^{k_H,j}),$$
where $H$ runs over $\AC$. 

We denote by $T_H$ 
the image of $\sb_H$ in $\HC_k(W)$. We have
\equat\label{eq:th}
\prod_{j=0}^{e_H-1} (T_H - \z_{e_H}^j \qb^{k_H,j})=0.
\endequat
If $q$ is a non-zero complex number, let $\HC_k(W,q)$ denote a {\it specialization of 
$\HC_k(W)$ at $q$}. Namely, we choose a complex logarithm $v$ of $q$ and 
we denote by $\eval_v : R \to \CM$ the morphism of $\CM$-algebras 
such that $\qb^a \mapsto q^a=\exp(a v)$ for all $a \in \CM$. Then $\HC_k(W,q)$ 
is the $\CM$-algebra obtained by specialisation through $\eval_v$. 
This is clearly an abuse of notation, as the specialization might depend 
on the choice of the logarithm $v$ of $q$ (for instance whenever 
the parameter $k$ has some non-integer values). But it turns out that, 
in this survey, this notation will occur only whenever the specialization 
does not depend on this choice.

Recall that $\HC_k(W)$ is a free $R$-module of rank $|W|$ 
(see~\cite{ariki}, \cite{ariki-koike}, \cite{BMR}, \cite{chavli1}, 
\cite{chavli2}, \cite{chavli3}, \cite{marin1}, \cite{marin2}, \cite{marin3}, 
\cite{marin-pfeiffer} and~\cite{tsuchioka}) 
such that its specialization $\HC_k(W,1)$ is just the group algebra $\CM W$ 
of $W$ over $\CM$. 

\bigskip

\subsection{Hecke families}
Whenever $k_{\O,j} \in \ZM$ for all $(\O,j) \in \aleph$, Brou\'e and Kim~\cite{broue kim} 
defined a partition of $\Irr(W)$ into families, which they call {\it Rouquier $k$-families}. 
In~\cite[\S{6.5}]{calogero}, Rouquier and the author extended (easily) the definition 
of these families to gereral parameters $k$, and decided to call them {\it Hecke 
$k$-families}. We will stick to this last terminology in this paper. Let us explain 
this definition.

Let $K$ denote the fraction field of $R$. The $K$-algebra $K\HC_k(W)$ is 
split semisimple~\cite[Theo.~5.2]{malle} so, by Tits deformation 
Theorem~\cite[Theo.~7.4.6]{geck pfeiffer}, 
it is isomorphic to the group algebra $KW$. Therefore, its irreducible characters 
are in bijection with $\Irr(W)$. If $\chi \in \Irr(W)$, we denote by 
$\chi_k$ the corresponding irreducible character of $K\HC_k(W)$. 
Now, let $R_\cyclo$ denote the localization of $R$ defined by 
$$R_\cyclo=R[\bigl((1-\qb^a)^{-1}\bigr)_{a \in \CM \setminus \{0\}}].$$
We say the $\chi$ and $\chi'$ are {\it in the same Hecke $k$-family} 
if there is a primitive central idempotent $b$ of $R_\cyclo \HC_k(W)$ 
such that $\chi_k(b)=\chi_k'(b) \neq 0$.

We denote by $\Fam_k^\hecke(W)$ the set of Hecke $k$-families.

\bigskip

\subsection{Calogero-Moser families}\label{sub:cm-families}
Calogero-Moser families were defined by Gordon 
using {\it baby Verma modules}~\cite[\S{4.2}~and~\S{5.4}]{gordon}. We explain  
here an equivalent definition given in~\cite[\S{7.2}]{calogero}. 
If $\chi \in \Irr(W)$, we denote by $\o_\chi : \Zrm(\CM W) \to \CM$ 
its central character (i.e., $\o_\chi(z)=\chi(z)/\chi(1)$ is the scalar by which 
$z$ acts on an irreducible representation affording the character $\chi$). 
We denote by $e_\chi$ (or $e_\chi^W$ if necessary) the primitive central 
idempotent such that $\o_\chi(e_\chi)=1$. 
We say that two irreducible characters $\chi$ and $\chi'$ of $W$ belong to the same 
{\it Calogero-Moser $k$-family} 
if $\o_\chi \circ \Omeb^k = \o_{\chi'} \circ \Omeb^k$. 
If $\FG$ is a subset of $\Irr(W)$, we set
$$e_\FG=\sum_{\chi \in \FG} e_\chi \in \Zrm(\CM W).$$
Finally, we denote by $\Fam_k^\calo(W)$ the set of Calogero-Moser $k$-families. 
Then~\cite[$(16.1.2)$]{calogero} 
\equat\label{eq:im-omega}
\im(\Omeb^k) = \bigoplus_{\FG \in \Fam_k^\calo(W)} \CM e_\FG
\endequat
and $\im(\Omeb^k)$ can be identified with the algebra of functions 
on $\ZC_k^{\CM^\times}$. 

In other words, this defines a surjective map 
$$\zG_k : \Irr(W) \longto \ZC_k^{\CM^\times}$$
whose fibers are the Calogero-Moser $k$-families. 
If $p \in \ZC_k^{\CM^\times}$, we denote by $\FG_p$ 
(or $\FG_p^k$ if we need to emphasize the parameter) 
the corresponding Calogero-Moser $k$-family. The next conjecture 
can be found in~\cite{martino}:

\bigskip

\begin{quotation}
\begin{conj}[Martino]\label{conj:martino}
Let $k^\sharp$ be the parameter $(k_{\O,-j})_{(\O,j) \in \aleph} \in \CM^\aleph$, 
where the index $j$ is viewed modulo $e_\O$. Then each Calogero-Moser $k$-family is a union of Hecke 
$k^\sharp$-families.
\end{conj}
\end{quotation}

\bigskip

\begin{theo}\label{theo:martino}
Conjecture~\ref{conj:martino} is known to hold in the following cases\footnote{We refer to Shephard-Todd notation for irreducible complex reflection groups~\cite{ST}.}:
\begin{itemize}
\itemth{1} If $W$ is of type $G(de,e,r)$, with $d$, $e$, $r \ge 1$ and $e$ odd whenever $r=2$.

\itemth{2} If $W$ is of type $G_4$, $G_{12}$, $G_{13}$, $G_{20}$, $G_{22}$, $G_{23}=\Wrm(H_3)$ 
or $G_{28}=\Wrm(F_4)$.

\itemth{3} If $W$ is of type 
$G_5$, $G_6$, $G_7$, $G_8$, $G_9$, $G_{10}$, $G_{14}$, $G_{15}$, $G_{16}$ or $G_{24}$ 
for generic values of $k$.
\end{itemize}
\end{theo}

\bigskip

\begin{proof}
For~(1), see~\cite{martino, bellamy these, martino 2}. 
For~(2) and~(3), see~\cite{thiel} (except for~$G_{28}=\Wrm(F_4)$: for this one, 
see~\cite{bonnafe thiel}).
\end{proof}

\part{Questions about Calogero-Moser spaces}\label{part:cm}

\bigskip

\section{Cohomology}

\medskip

\bigskip

\subsection{Localization} 
We denote by $i_k : \ZC_k^{\CM^\times} \injto \ZC_k$ 
the closed immersion (here, $\ZC_k^{\CM^\times}$ denotes 
the reduced zero-dimensional variety of $\CM^\times$-fixed points). 
As explained in~\S\ref{sub:cm-families}, we have a natural isomorphism 
of algebras 
\equat\label{eq:coho-fixe}
\Hrm_{\CM^\times}(\ZC_k^{\CM^\times}) \simeq \CM[\hbar] \otimes \im(\Omeb^k).
\endequat
So we view the map $i_k^*$ as a morphism of algebras 
$$i_k^* : \Hrm_{\CM^\times}(\ZC_k) \longto \CM[\hbar] \otimes \im(\Omeb^k).$$
We can now state the following conjecture (see~\cite[\S{16.1}]{calogero} 
and~\cite[Conj.~3.3]{bonnafe shan}).

\bigskip

\begin{quotation}
\begin{conj}\label{conj:coho}
With the above notation, we have:
\begin{itemize}
\itemth{a} If $i \ge 0$, then $\Hrm^{2i+1}(\ZC_k)=0$.

\itemth{b} $\im(i_k^*) = \Rees_\FC(\im(\Omeb^k))$. 
\end{itemize}
\end{conj}
\end{quotation}

\bigskip

Recall from standard arguments~\cite[Prop.~2.4]{bonnafe shan} that this conjecture 
would imply a description of both the cohomology and the equivariant 
cohomology of $\ZC_k$:

\bigskip

\begin{prop}\label{prop:coho}
Assume that Conjecture~\ref{conj:coho} holds. Then:
\begin{itemize}
\itemth{a} If $i \ge 0$, then $\Hrm_{\CM^\times}^{2i+1}(\ZC_k)=0$.

\itemth{b} $\Hrm^{2 \bullet}_{\CM^\times}(\ZC_k) \simeq 
\Rees_\FC(\im(\Omeb^k))$ as $\CM[\hbar]$-algebras.

\itemth{c} $\Hrm^{2\bullet}(\ZC_k) \simeq \gr_\FC(\im(\Omeb^k))$.
\end{itemize}
\end{prop}

\bigskip

\begin{theo}\label{theo:conj-coho}
Conjecture~\ref{conj:coho} is known to hold in the following cases:
\begin{itemize}
\itemth{a} If $k=0$.

\itemth{b} If $\dim V =1$.

\itemth{c} If $\ZC_k$ is smooth.
\end{itemize}
\end{theo}

\bigskip

\begin{proof}
(a) follows from the fact that $\ZC_0=(V \times V^*)/W$ is contractible 
and $\im \Omeb^0=\CM$. For~(b), see~\cite[Theo.~18.5.8]{calogero} and~\cite[Prop.~1.6]{bonnafe shan}. 
For~(c), see~\cite[Theo.~A]{bonnafe shan}.
\end{proof}

\bigskip

\begin{exemple}\label{ex:cyclique}
It might be tempting to conjecture that the Calogero-Moser space 
$\ZC_k$ is rationally smooth and $p$-smooth if $p$ is a prime number 
not dividing $|W|$. Indeed, $\ZC_k$ is a deformation of $\ZC_0=(V \times V^*)/W$ 
which tends to be smoother and smoother as $k$ becomes more and more generic. 
However, both statements are false in general:
\begin{itemize}
\itemth{1} If $\dim V=1$ and $m=|W| \ge 2$ (so that $\aleph=((0,j))_{0 \le j \le m-1}$ and 
we write $k_j=k_{0,j}$ for simplicity), then it follows from~\cite[Theo.~18.2.4]{calogero} 
that 
$$\ZC_k=\bigl\{(x,y,z) \in \CM^3~|~\prod_{j=0}^{m-1} (z-mk_j)=xy\bigr\}.$$
Now, if $p$ is a prime number 
not dividing $m$ and smaller than $m$ (this always exists if $m \ge 3$), 
and if we choose $k$ 
such that $k_0=k_1=\cdots=k_{p-1}=0$ and $k_p=k_{p+1}=\cdots = k_{m-1}=1$, 
then 
$$\ZC_k=\{(x,y,z) \in \CM^3~|~z^p (z-m)^{m-p}=xy\}$$
contains a simple singularity 
of type $A_{p-1}$ and so $\ZC_k$ is rationally smooth but not $p$-smooth while 
$\ZC_0=\CM^2/\mub_m$ is $p$-smooth because $p$ does not divide $m$.

\itemth{2} If $W$ is of type $B_2$ and $(k_{\O,0},k_{\O,1})=(k_{\O',0},k_{\O',1})$ 
and $k_{\O,0} \neq k_{\O,1}$ (where $\O$ and $\O'$ are the two orbits 
of reflecting hyperplanes), then $\ZC_k$ admits a unique singular 
point and the singularity is equivalent to the singularity at $0$ 
of the orbit closure of the minimal nilpotent orbit of the Lie algebra 
$\sG\lG_3(\CM)$ (see~\cite[Theo.~1.3$($b$)$]{BBFJLS}): 
it is well-known that this orbit closure is not rationally smooth.\finl
\end{itemize}
\end{exemple}

\bigskip

\subsection{Morphisms between Calogero-Moser spaces}\label{sub:coho-filtration}
Let $(V',W')$ be another pair consisting of a finite dimensional 
complex vector space and a finite subgroup $W' \subset \Gb\Lb_\CM(V')$. 
We fix a parameter $k' \in \CM^{\aleph(V',W')}$ and, 
in this subsection, we will denote by 
a prime $?'$ the object $?$ defined using $(V',W')$ 
instead of $(V,W)$, i.e. the object $?(V',W')$. 
For instance, $\ZC_{k'}'=\ZC_{k'}(V',W')$ and 
$\aleph'=\aleph(V',W')$. 

\medskip

\boitegrise{{\bf Hypothesis.} 
{\it We assume in this subsection that 
we are given a $\CM^\times$-equivariant morphism of varieties 
$\ph : \ZC_{k'}' \to \ZC_k$}.}{0.75\textwidth}

\medskip

We denote 
by $\ph_\fix : \ZC_{k'}^{\prime \CM^\times} \to \ZC_k^{\CM^\times}$ 
the induced map. Then $\ph_\fix$ induces a morphism of algebras 
$$\ph_\fix^\# : \im \Omeb^k \longto \im \Omeb^{\prime k'}$$ 
through the formula
$$\ph_\fix^\#(e_{\FG_p}) = \sum_{p' \in \ph_\fix^{-1}(p)} e_{\FG_{p'}'}'.$$
The following proposition should be compared with~\cite[Cor.~1.5]{bonnafe maksimau}:

\bigskip

\begin{prop}\label{prop:coho-graduation}
Assume that Conjecture~\ref{conj:coho} holds for both $\ZC_k$ and 
$\ZC_{k'}'$. Then 
$$\ph_\fix^\#(\FC_j \im \Omeb^k) \subset \FC_j' \im \Omeb^{\prime k'}$$
for all $j$. 
\end{prop}

\bigskip

\begin{proof}
The maps $\ph$ and $\ph_\fix$ induce maps between equivariant cohomology groups  
which we denote by 
$$\ph^* : \Hrm_{\CM^\times}^\bullet(\ZC_k) \longto 
\Hrm_{\CM^\times}^\bullet(\ZC_{k'}')$$
$$\ph_\fix^* : \Hrm_{\CM^\times}^\bullet(\ZC_k^{\CM^\times}) \longto 
\Hrm_{\CM^\times}^\bullet(\ZC_{k'}^{\prime \CM^\times}).\leqno{\text{and}}$$ 
The functoriality properties of cohomology imply that the 
diagram
$$\diagram
\Hrm_{\CM^\times}^\bullet(\ZC_k) \rrto^{\DS{\ph^*}} \ddto_{\DS{i_k^*}}&& 
\Hrm_{\CM^\times}^\bullet(\ZC_{k'}') \ddto^{\DS{i_{k'}^{\prime *}}} \\
\\
\Hrm_{\CM^\times}^\bullet(\ZC_k^{\CM^\times}) \rrto^{\DS{\ph_\fix^*}} && 
\Hrm_{\CM^\times}^\bullet(\ZC_{k'}^{\prime \CM^\times})
\enddiagram$$
is commutative. Recall from~(\ref{eq:coho-fixe}) that we identify 
$\Hrm_{\CM^\times}^\bullet(\ZC_k^{\CM^\times})$ 
(resp. $\Hrm_{\CM^\times}^\bullet(\ZC_{k'}^{\prime \CM^\times})$) 
with $\CM[\hbar] \otimes \im \Omeb^k$ 
(resp. $\CM[\hbar] \otimes \im \Omeb^{\prime k'}$). Through 
this identification, the map $\ph_\fix^*$ becomes $\Id_{\CM[\hbar]} \otimes \ph_\fix^\#$ 
by construction. Therefore, the above commutative diagram yields 
a commutative diagram
$$\diagram
\Hrm_{\CM^\times}^\bullet(\ZC_k) \xto[0,3]^{\DS{\ph^*}} \ddto_{\DS{i_k^*}}&&& 
\Hrm_{\CM^\times}^\bullet(\ZC_{k'}') \ddto^{\DS{i_{k'}^{\prime *}}} \\
\\
\CM[\hbar] \otimes \im \Omeb^k
\xto[0,3]^{\DS{\Id_{\CM[\hbar]} \otimes \ph_\fix^\#}} &&& 
\CM[\hbar] \otimes \im \Omeb^{\prime k'}.
\enddiagram$$
So $\Id_{\CM[\hbar]} \otimes \ph_\fix^\#(\im i_k^*) \subset \im i_{k'}^{\prime *}$. 
As we assume that Conjecture~\ref{conj:coho} holds for both $\ZC_k$ and $\ZC_{k'}'$, 
this is exactly the statement of the proposition.
\end{proof}

\bigskip

\begin{rema}\label{rem:applications}
In the next section, we propose some conjecture 
which would give many examples of morphisms between Calogero-Moser 
spaces. In all the cases where these conjectures are proved, 
the above Proposition~\ref{prop:coho-graduation} gives a highly 
non-trivial link between the character tables of $W$ and $W'$ 
(see for instance~\cite[Cor.~4.22]{bonnafe maksimau} for the case where 
$W=G(l,1,n)$).\finl
\end{rema}

\bigskip

\section{Symplectic leaves and fixed points}\label{sec:symplectic}

\medskip

If $\t \in \Nrm_{\Gb\Lb_\CM(V)}(W)$ has finite order and satisfies $\t(k)=k$, then $\t$ acts on 
the Calogero-Moser space $\ZC_k$. We are interested in this section in the 
variety of fixed points $\ZC_k^\t$ of $\t$ (endowed with its reduced 
structure) and its symplectic leaves. Since $W$ acts trivially on $\ZC_k$, 
the action of $\t$ on $\ZC_k$ depends only on its coset $W\t$.

We say that $\t$ is {\it $W$-full} if $\dim(V^\t)=\max_{w \in W} \dim(V^{w\t})$ 
(see~\cite[\S{1.F.4}]{auto}). Recall that $\t$ is $W$-full 
if and only if the natural map $V^\t \longto (V/W)^\t$ is onto~\cite[$($3.2$)$]{auto}, 
(the argument is due to Springer~\cite{springer}). Since $W$ acts trivially on $\ZC_k$, we may replace 
$\t$ by any $w\t$ and assume that $\t$ is $W$-full. Therefore, we will work in this section 
under the following hypothesis:

\medskip

\boitegrise{{\bf Hypothesis and notation.} {\it We fix in this section, and only in this section, 
a {\bfit $W$-full} element $\t$ of finite 
order in $\Nrm_{\Gb\Lb_\CM(V)}(W)$ and we assume that $\t(k)=k$.}}{0.75\textwidth}

\medskip

As in~\cite{auto}, let $W_\t$ denote the quotient 
$\Sigma/\Pi$, where $\Sigma$ (resp. $\Pi$) is the setwise (resp. pointwise) 
stabilizer of $V^\t$. 
A parabolic subgroup 
$P$ of $W$ is called {\it $\t$-split} if $P$ is the stabilizer in $W$ of 
a vector belonging to $V^\t$.
Note that a $\t$-split parabolic subgroup is $\t$-stable. If $P$ is a $\t$-split 
parabolic subgroup of $W$, we set
$$\Nrmov_{W_\t}(P_\t)=\Nrm_{W_\t}(P_\t)/P_\t.$$
Then $\Nrmov_{W_\t}(P_\t)$ acts faithfully on the vector space $(V^P)^\t$. 
So one can define Calogero-Moser spaces associated with 
the pair $((V^P)^\t,\Nrmov_{W_\t}(P_\t))$, even though 
$\Nrmov_{W_\t}(P_\t)$ is not necessarily a reflection group for its action on $(V^P)^\t$. 

\bigskip

\subsection{Symplectic leaves}\label{sub:symp} 
Brown-Gordon~\cite[\S{3.5}]{brown gordon} defined a stratification of any complex affine Poisson variety 
into {\it symplectic leaves}. They also proved that the Calogero-Moser space $\ZC_k$ 
has only finitely many symplectic leaves~\cite[Theo.~7.8]{brown gordon}. 
As explained in~\cite[\S{4.A}]{auto}, this implies that the variety $\ZC_k^\t$ admits a stratification  
into symplectic leaves and that there are only finitely many of them. We denote 
by $\symp(\ZC_k^\t)$ the set of its symplectic leaves. Such a symplectic 
leaf is called {\it $\t$-cuspidal} if it has dimension $0$ (we also 
talk about {\it $\t$-cuspidal} points\footnote{We can also say that a 
Calogero-Moser $k$-family is {\it $\t$-cuspidal} if it is associated to a 
$\t$-cuspidal point. Of course, a $\t$-cuspidal family is $\t$-stable.}). 
Note that this notion can be defined even if $\t$ 
is not $W$-full. Let $\Cus_k^\t(V,W)$ denote the 
set of pairs $(P,p)$ where $P$ is a $\t$-split parabolic subgroup 
of $W$ and $p$ is a $\t$-cuspidal point of $\ZC_{k_P}(V/V^P,P)$ 
(here $k_P$ is the restriction of $k$ to $\aleph(V/V^P,P)$).

Then $W_\t$ acts on $\Cus_k^\t(V,W)$ 
and the next result is proved in~\cite[Theo.~A]{auto} 
(whenever $\t=\Id_V$, it is independently due to Bellamy~\cite{bellamy cuspidal} 
and Losev~\cite{losev}). 

\bigskip

\begin{theo}\label{theo:leaves}
Recall that $\t$ is $W$-full. 
Then there is a natural bijection 
$$\symp(\ZC_k^\t) \longiso \Cus_k^\t(V,W)/W_\t.$$
Moreover, the dimension of the symplectic leaf associated with the $W_\t$-orbit of 
$(P,p)$ through this bijection is equal to $2 \dim (V^P)^\t$.
\end{theo}

\bigskip

We refer to~\cite[Lem.~8.4]{auto} for the explicit description of the bijection: 
if $(P,p) \in \Cus_k^\t(V,W)$, we denote by $\leaf_{P,p}$ its associated 
symplectic leaf. 
Recall also~\cite[Rem.~4.2]{auto} that all symplectic leaves are 
$\CM^\times$-stable. 

\bigskip

\subsection{Normalization} 
Let $(P,p) \in \Cus_k^\t(V,W)$. Then $\overline{\leaf_{P,p}}$ carries a Poisson 
structure and so does its normalization $\overline{\leaf_{P,p}}^\nor$ (see~\cite{kaledin}). 
We proposed in~\cite[Conj.~B]{auto} the following conjecture:

\bigskip

\begin{quotation}
\begin{conj}\label{conj:leaves}
There exists a parameter $k_{P,p} \in \CM^{\aleph((V^P)^\t,\Nrmov_{W_\t}(P_\t))}$ 
such that the varieties 
$\overline{\leaf_{P,p}}^\nor$ and $\ZC_{k_{P,p}}((V^P)^\t,\Nrmov_{W_\t}(P_\t))$ 
are isomorphic as Poisson varieties endowed with a $\CM^\times$-action. 
\end{conj}
\end{quotation}

\bigskip

\begin{theo}\label{theo:normalization}
Conjecture~\ref{conj:leaves} is known to hold in the following cases:
\begin{itemize}
\itemth{a} If $k=0$.

\itemth{b} If $\ZC_k$ is smooth.

\itemth{c} If $W$ is a Weyl group of type $B$ (i.e. $C$) and $\t=\Id_V$.

\itemth{d} If $W$ is of type $D$ and $\t$ is a diagram automorphism.

\itemth{e} If $W$ is dihedral and $\t$ is the non-trivial diagram automorphism.

\itemth{f} If $W$ is of type $G_4$.
\end{itemize}
\end{theo}

\bigskip

\begin{proof}
See~\cite[Prop.~6.7~and~\S{9}]{auto} for more details: this relies on works of 
Bellamy-Maksimau-Schedler~\cite{be-sc}, 
Maksimau and the author~\cite{bonnafe maksimau} and Thiel and the author~\cite{bonnafe thiel}.
\end{proof}

\bigskip

\subsection{Regular case}\label{sub:regular} 
We say that the element $\t$ of $\Nrm_{\Gb\Lb_\CM(V)}(W)$ is {\it regular} 
if $V_\reg^\t \neq \vide$. In this case, $\ZC_k^\t$ admits a unique irreducible 
component of maximal dimension~\cite[Prop.~2.4]{regular}: we denote by by $(\ZC_k^\t)_\maxi$. 
It has dimension $2\dim(V^\t)$. The following result has been proved in~\cite[Theo.~2.8]{regular} 
(here, if $\chi$ is a $\t$-stable character of $W$, we choose an extension 
$\chit$ of $\chi$ to $W\langle \t \rangle$):

\bigskip

\begin{theo}\label{theo:chi-tau}
Assume that $\t$ is a regular element. 
Let $p \in \ZC_k^{\CM^\times}$ be such that $\t(p)=p$ and 
$\sum_{\chi \in \FG_p^\t} |\chit(\t)|^2  \neq 0$. 
Then $p$ belongs to $(\ZC_k^\t)_\maxi$. 
\end{theo}

\bigskip

Moreover, it is conjectured in~\cite[Conj.~2.6]{regular} that the converse holds:

\bigskip

\begin{quotation}
\begin{conj}\label{conj:chi-tau}
Assume that $\t$ is a regular element. 
Let $p \in \ZC_k^{\CM^\times}$ be such that $\t(p)=p$. Then $p$ belongs to $(\ZC_k^\t)_\maxi$ 
if and only if $\sum_{\chi \in \FG_p^\t} |\chit(\t)|^2  \neq 0$.
\end{conj}
\end{quotation}

\bigskip

Note that Conjecture~\ref{conj:chi-tau} holds for $W=\SG_n$ by~\cite[Exam.~5.7]{regular}. 

\bigskip

\section{Special features of Coxeter groups}\label{sec:coxeter conjectures}

\medskip

\boitegrise{{\bf Hypothesis and notation.} 
{\it We assume in this section, and only in this section, 
that there exists a $W$-stable $\RM$-vector subspace $V_\RM$ of $V$ 
such that $V = \CM \otimes_\RM V_\RM$ as a $W$-module, that 
$k$ takes only {\bfit real values} and that $c_k(s) \ge 0$ for all $s \in \Ref(W)$.
}}{0.75\textwidth}

\medskip

First, note that the reflections of $W$ have order $2$ 
(so that $e_H=2$ for all $H \in \AC$) and that we have a bijection 
between $\Ref(W)$ and $\AC$. This implies that 
$$c_k(s)=k_{H,1}-k_{H,0},$$
where $H$ is the reflecting hyperplane of $s$. Note that $k=k^\sharp$. 

\def\schur{{\sb\cb\hb}}

\subsection{Lusztig families} 
If $\chi \in \Irr(W)$, we denote by $\schur^{(k)}_\chi \in \OC[\qb^\RM]$ 
the {\it Schur element} associated with the irreducible character $\chi_k$ 
of the Hecke algebra $\HC_k(W)$ (see~\cite[\S{7.2}]{geck pfeiffer}): 
since $\RM$ is an ordered group, we can set 
$a_\chi^{(k)} = \val \schur_\chi^{(k)}$ and $A_\chi^{(k)}=\deg \schur_\chi^{(k)}$: this defines two maps 
$a^{(k)}$, $A^{(k)} : \Irr(W) \to \RM$. Using the map $a^{(k)}$ and the notion of {\it $J$-induction}, 
Lusztig~\cite[\S{22}]{lusztig} defined the notion of 
{\it $k$-constructible character} (or {\it $c_k$-constructible character}) 
of $W$. Let $\irrc_k(W)$ denote the graph defined as follows:
\begin{itemize}
\item[$\bullet$] The set of vertices of $\irrc_k(W)$ is $\Irr(W)$.

\item[$\bullet$] There is an edge between two vertices if they both occur in a same 
$k$-constructible character.
\end{itemize}
A {\it Lusztig $k$-family} is a subset of $\Irr(W)$ consisting of 
the vertices of a connected component of $\irrc_k(W)$. We denote 
by $\Fam_k^\lus(W)$ the set of Lusztig $k$-families of $W$. 
It turns out that 
\equat\label{eq:lusztig-hecke}
\Fam_k^\hecke(W)=\Fam_k^\lus(W)
\endequat
(see~\cite{chlou} and~\cite{lusztig}). Martino's Conjecture~\ref{conj:martino} 
has a more precise version in the Coxeter case:

\bigskip

\begin{quotation}
\begin{conj}[Gordon-Martino]\label{conj:gordon-martino}
If $W$ is a Coxeter group, then 
$$\Fam_k^\calo(W)=\Fam_k^\lus(W).$$
\end{conj}
\end{quotation}

\bigskip

It follows from the definition~\cite{lusztig} 
that 
\equat\label{eq:a constant hecke}
\text{\it the two maps $a^{(k)}$, $A^{(k)} : \Irr(W) \to \RM$ are constant 
on Lusztig $k$-families.}
\endequat
So the next proposition is a strong argument 
in favor of Conjecture~\ref{conj:gordon-martino}:

\bigskip

\begin{prop}\label{prop:a+A}
The map $a^{(k)}+A^{(k)} : \Irr(W) \to \RM$ is constant on Calogero-Moser 
$k$-families.
\end{prop}

\bigskip

\begin{proof}
If $\chi$ and $\chi'$ belong to the same Calogero-Moser family, 
then $\o_\chi(\Omeb^k(\euler))=\o_{\chi'}(\Omeb^k(\euler))$. 
But the scalar $\o_\chi(\Omeb^k(\euler))$ is, up to a suitable 
renormalization by a fixed affine transformation of $\RM$, 
equal to $a_\chi^{(k)} + A_\chi^{(k)}$ (see~\cite[Lem.~7.2.1]{calogero} 
and~\cite[4.21]{broue michel}). So the result follows.
\end{proof}

\bigskip

\begin{rema}\label{rem:lusztig}
Using the Kazhdan-Lusztig basis of the Hecke algebra $\HC_k(W)$ 
and the associated partition of $W$ into two-sided cells~\cite[\S{8}]{lusztig} 
(see also~\cite[Def.~6.1.4]{bonnafe kl}), 
Lusztig defined a partition of $\Irr(W)$ into {\it Kazhdan-Lusztig 
$k$-families} as follows: two irreducible characters $\chi$ and $\chi'$ of $W$ 
belong to the same Kazhdan-Lusztig $k$-family if $\chi_k$ and $\chi_k'$ both 
occur in the left module associated to a same two-sided cell. 
Let $\Fam_k^\kl(W)$ denote the set of Kazhdan-Lusztig $k$-families 
of $W$. 

Lusztig conjectured~\cite[\S{23}]{lusztig} that 
$\Fam_k^\lus(W)=\Fam_k^\kl(W)$. This 
conjecture is known to hold in the following cases:
\begin{itemize}
\item[$\bullet$] If $c_k$ is constant~\cite[Prop.~23.3]{lusztig} (note that this covers 
the case where $W$ has only one conjugacy class of reflections, i.e. if  
$W$ is of type $ADE$ or $I_2(2m+1)$).

\item[$\bullet$] If $W$ is dihedral~\cite{lusztig}.

\item[$\bullet$] If $W$ is of type $F_4$~\cite{geck f4}.

\item[$\bullet$] If $W$ is of type $B_n$ and $c_k(t) > (n-1) c_k(s_1)$, 
where the Coxeter graph is given by
\centerline{\begin{picture}(160,40)
\put(10,20){\circle{10}}\put(45,20){\circle{10}}\put(80,20){\circle{10}}\put(150,20){\circle{10}}
\put(14,17){\line(1,0){27}}\put(14,23){\line(1,0){27}}\put(50,20){\line(1,0){25}}
\put(85,20){\line(1,0){15}}\put(130,20){\line(1,0){15}}
\put(7,30){$t$}\put(42,30){$s_1$}\put(77,30){$s_2$}\put(143,30){$s_{n-1}$}
\put(110,17.4){$\cdots$}
\end{picture}}
(see~\cite{bonnafe iancu} and~\cite{bonnafe bilatere}).
\end{itemize}
Note that this conjecture involves only the Hecke algebra and is not 
related to the geometry of the Calogero-Moser space and the theme of 
this paper.\finl
\end{rema}

\bigskip

\subsection{Cuspidal families}
Lusztig also introduced the important 
notion of {\it $\t$-cuspidal} Lusztig $k$-family~\cite[\S{8.1}]{lusztig orange} 
(note that the definition in~\cite[\S{8.1}]{lusztig orange} is for the {\it equal parameter} case, 
but the definition can easily be extended to general parameters, as explained 
in~\cite[\S{2.5}]{bellamy thiel}). As there is also a notion of $\t$-cuspidal Calogero-Moser family 
(see~\S\ref{sub:symp}), it is natural to expect that the equality predicted by 
Gordon-Martino conjecture preserves this feature:

\bigskip

\begin{quotation}
\begin{conj}\label{conj:cuspidal-cm-kl}
If $W$ is a finite Coxeter group, and if $\t \in \Nrm_{\Gb\Lb_\RM(V_\RM)}(W)$ is $W$-full 
and satisfies $\t(k)=k$, then the $\t$-cuspidal Lusztig $k$-families coincide with the $\t$-cuspidal 
Calogero-Moser $k$-families.
\end{conj}
\end{quotation}

\bigskip

If $\t=\Id_V$, this conjecture has been proposed by Bellamy and Thiel~\cite[Conj.~B]{bellamy thiel}. 
The $\t$-cuspidal Lusztig $k$-families have been classified (see~\cite[\S{8.1}]{lusztig orange} 
for the equal parameter case and~\cite[\S{6},~\S{7}]{bellamy thiel} for the unequal parameter case) 
and it turns out that there is at most one $\t$-cuspidal Lusztig $k$-family. 
So Conjecture~\ref{conj:cuspidal-cm-kl} would imply that there is at most 
one $\t$-cuspidal point in $\ZC_k$ whenever $W$ is a Coxeter group~\cite[Conj.~D]{bellamy thiel}.

\bigskip

\begin{theo}\label{theo:cuspidal-unique}
Conjectures~\ref{conj:gordon-martino} and~\ref{conj:cuspidal-cm-kl} are known to hold for $W$ of type 
$A$, $B=C$, $D$,  $I_2(m)$, $H_3$ or $F_4$ (with the restriction that $\t=\Id_V$ if $W$ is of type $F_4$). 
\end{theo}

\bigskip

\begin{proof}
For Conjecture~\ref{conj:gordon-martino}, 
see~\cite[Theo.~5.6]{gordon} for type $A$, see~\cite{gordon martino} for 
types~$B=C$ and $D$, see~\cite{bellamy these} for type $I_2(m)$ 
and~\cite{bonnafe thiel} for types $H_3$ and $F_4$.

For Conjecture~\ref{conj:cuspidal-cm-kl}, 
see~\cite[Theo.~A]{bellamy thiel} for types $A$, $B=C$, $D$ and $I_2(m)$, 
and~\cite{bonnafe thiel} for types $H_3$ and $F_4$.
\end{proof}

\bigskip

\part{Unipotent representations of finite reductive groups}\label{part:unip}

Throughout this part, we will only consider algebraic varieties 
and algebraic groups defined over an algebraic closure of a finite
field. If $\Gb$ is an algebraic group, we denote by $\Zrm(\Gb)$ its 
center. If $\Sb$ is a torus, we denote by $Y(\Sb)$ its lattice of 
one-parameter subgroups.

Let $\groups$ denote the class of triples $(q,\Gb,F)$ 
where:
\begin{itemize}
\item[$\bullet$] $q$ is a power of some prime number $p$.

\item[$\bullet$] $\Gb$ is a connected reductive group defined over an 
algebraic closure $\FM$ of the finite field $\gfp$ with $p$ elements.

\item[$\bullet$] $F : \Gb \to \Gb$ is a Frobenius endomorphism of $\Gb$ 
endowing $\Gb$ with a rational structure over the finite 
subfield $\gfq$ of $\FM$ with $q$ elements.
\end{itemize}
This part provides a quick survey on unipotent representations 
of the finite reductive group 
$\Gb^F$ (where $(q,\Gb,F) \in \groups$) and their associated 
structures (cuspidal representations, Harish-Chandra theory, 
$d$-Harish-Chandra theory...)\footnote{Our formalism 
here does not include the 
Suzuki and Ree groups}

\medskip

\boitegrise{{\bf Hypothesis and notation.} {\it We fix in this 
part, and only in this part, a triple 
$\GC=(q,\Gb,F) \in \groups$. We denote by $p$ the unique prime 
number dividing $q$ and by $\FM$ the algebraic closure 
of $\gfp$ over which $\Gb$ is defined. We fix a prime number 
$\ell$ different from $p$ and we denote by $\qlb$ an algebraic 
closure of $\QM_{\! \ell}$. 
Note that $\Gb$ is not necessarily split over $\gfq$. 
}}{0.75\textwidth}

If $\Xb$ is an algebraic variety 
over $\FM$, we denote by $\Hrm_c^j(\Xb)$ its $j$-th $\ell$-adic 
cohomology group with compact support with coefficients in $\qlb$: 
it is a finite dimensional $\qlb$-vector space. We set 
$\Hrm_c^\bullet(\Xb)=\bigoplus_{j \geqslant 0} \Hrm_c^j(\Xb)$. 

We fix an $F$-stable Borel subgroup $\Bb$ of $\Gb$ and 
an $F$-stable maximal torus $\Tb$ of $\Bb$. 
Let $\BCB=\Gb/\Bb$ denote the flag variety. 
Let $\WC=\Nrm_\Gb(\Tb)/\Tb$ denote the Weyl group of $\Gb$. It is acted on by $F$ 
and we denote by $\t$ the automorphism of $\WC$ induced by $F$. 
If $\OCB$ is a $\Gb$-orbit in $\BCB \times \BCB$, we denote by 
$$\Xb_\OCB=\{g\Bb \in \BCB~|~(g\Bb,F(g)\Bb) \in \OCB\}.$$
Then $\Xb_\OCB$ is called a {\it Deligne-Lusztig variety}: it is acted on 
the left by the finite group $\Gb^F$. Hence, the vector spaces 
$\Hrm_c^j(\Xb_\OCB)$ inherits a structure of $\qlb \Gb^F$-module. An 
irreducible representation of $\Gb^F$ (over $\qlb$) is called {\it unipotent} if it 
appears in such a $\qlb\Gb^F$-module, for some $\OCB$ and some $j$. The set 
of isomorphism classes of irreducible unipotent representations of 
$\Gb^F$ will 
be denoted by $\unipc(\GC)$. We define a {\it unipotent} 
representation of $\Gb^F$ to be a direct sum of irreducible unipotent 
representations. 

We define a {\it Levi subgroup} of $\GC$ to be a triple $\LC=(q,\Lb,F)$ 
where $\Lb$ is an $F$-stable Levi complement of 
a parabolic subgroup of $\Gb$. In this case, 
we set $W_\GC(\LC)=\Nrm_{\Gb^F}(\Lb)/\Lb^F$: this group 
acts on $\unipc(\LC)$ and, if $\l \in \unipc(\LC)$, we denote by 
$W_\GC(\LC,\l)$ its stabilizer.

\def\harishc{\Hrm\Crm}
\def\harish{\Hrm\Crm}

The set of unipotent representations admits several interesting partitions, 
which are related to the different problems one may consider: 
Harish-Chandra series for an algebraic parametrization, 
$d$-Harish-Chandra series for blocks in transverse 
characteristic, families for computing character values 
through the theory of character sheaves... All these 
partitions interconnect in a very subtle way.

\bigskip

\section{Harish-Chandra theories}

\medskip

\subsection{Classical Harish-Chandra theory}
If $\Lb$ is an $F$-stable Levi complement of an $F$-stable 
parabolic subgroup $\Pb$ of $\Gb$, we denote by 
$$\fonction{\RC_{\Lb \subset \Pb}^\Gb}{\qlb \Lb^F\module}{\qlb\Gb^F\module}{M
}{\Ind_{\Pb^F}^{\Gb^F} \Mti}$$
the {\it Harish-Chandra induction} functor~\cite[Chap.~4]{dmbook}. 
Here, $\Mti$ 
denotes the inflation of $M$ through the surjective morphism 
$\Pb^F \surto \Lb^F$. If we denote by $\Ub_\Pb$ the unipotent radical 
of $\Pb$, then $\RC_{\Lb \subset \Pb}^\Gb$ admits a left and right 
adjoint functor $\lexp{*}{\RC}_{\Lb \subset \Pb}^\Gb$ given by 
$$\fonction{\lexp{*}{\RC}_{\Lb \subset \Pb}^\Gb}{\qlb \Gb^F\module}{\qlb\Lb^F\module}{N
}{N^{\Ub_\Pb^F}.}$$
The functor $\lexp{*}{\RC}_{\Lb \subset \Pb}^\Gb$ is called the {\it Harish-Chandra 
restriction}. 
Note that both functors send a unipotent representation to a unipotent 
representation. An irreducible unipotent representation $N$ is called 
{\it cuspidal} if $\lexp{*}{\RC}_{\Lb \subset \Pb}^\Gb N=0$ 
for any $F$-stable Levi complement of a proper $F$-stable parabolic subgroup 
$\Pb$ of $\Gb$. 
We denote by $\unipc_\cus(\GC)$ the set of (isomorphism classes of) 
cuspidal unipotent irreducible representations of $\Gb^F$. 

It turns out that both Harish-Chandra functors do not depend 
on the choice of the $F$-stable parabolic subgroup $\Pb$ admitting 
$\Lb$ as a Levi complement. We denote by $\CUSC(\GC)$ 
the set of pairs $(\LC,\l)$, where $\LC=(q,\Lb,F)$ and $\Lb$ is an $F$-stable Levi 
complement of an $F$-stable parabolic subgroup of $\Gb$, and $\l \in \unipc_\cus(\LC)$. 
We denote by $\unipc(\GC,\LC,\l)$ the set of unipotent 
irreducible representations occuring in $\RC_{\Lb \subset \Pb}^\Gb \l$, 
where $\Pb$ is any $F$-stable parabolic subgroup admitting 
$\Lb$ as a Levi complement: this set is called the {\it Harish-Chandra series} 
associated with the pair $(\LC,\l)$. 
This set depends on the pair $(\LC,\l)$ up to $\Gb^F$-conjugacy. 
We denote by $\CUSC(\GC)/\!\!\sim$ the set of $\Gb^F$-conjugacy 
classes of elements of $\CUSC(\GC)$. 
The {\it Harish-Chandra theory} can then be summarized as follows~\cite[\S{5}]{lusztig coxeter}:

\bigskip

\begin{theo}[Lusztig]\label{theo:hc}
We have
$$\unipc(\GC)=\dot{\bigcup_{(\LC,\l) \in \CUSC(\GC)/\sim}} 
\unipc(\GC,\LC,\l),$$
where $\dot{\cup}$ means a disjoint union. 

Moreover, if $(\LC,\l) \in \CUSC(\GC)$, with $\LC=(q,\Lb,F)$ and if $\Pb$ is any 
$F$-stable parabolic subgroup admitting $\Lb$ as a Levi complement, then:
\begin{itemize}
\itemth{a} The group $W_\GC(\LC)$ is a Weyl group 
for its action on the lattice 
$$\{y \in Y(\Zrm(\Lb))~|~F(y)=q y\}.$$
Moreover, $W_\GC(\LC,\l)=W_\GC(\LC)$. 

\itemth{b} 
The endomorphism algebra $\End_{\Gb^F} \RC_{\Lb \subset \Pb}^\Gb \l$ 
is isomorphic to a Hecke algebra of the form $\HC_{k_\LC}(W_\GC(\LC),q)$, 
where $k_\LC \in \aleph(W_\GC(\LC))$ does not depend 
on $\l$ (and is integer valued).

\itemth{c} The statement {\rm (a)} induces a bijection 
$\harishc^{\GC,\LC,\l} : \Irr W_\GC(\LC) \longiso \unipc(\GC,\LC,\l)$. 
In other words,
$$\unipc(\GC)=\dot{\bigcup_{(\LC,\l) \in \CUSC(\GC)/\sim}} 
\harishc^{\GC,\LC,\l}\bigl(\Irr W_\GC(\LC) \bigr).$$
\end{itemize}
\end{theo}

\bigskip

\begin{exemple}[Principal series]\label{ex:principal}
As $\Tb$ is an $F$-stable maximal torus of the $F$-stable 
Borel subgroup $\Bb$, we may set $\TC=(q,\Tb,F)$ and so 
there is a Harish-Chandra series associated with the trivial 
character (denoted by $1$) of $\Tb^F$ (indeed, as there is no proper 
parabolic subgroup of $\Tb$, $1$ is a cuspidal unipotent representation). 
Then $W_\GC(\TC)=\WC^\t$ and the injective map 
$\harish^{\GC,\TC,1} : \Irr(\WC^\t) \longto \unipc(\GC)$ will 
be simply denoted by $\harish^\GC$. Its image is called 
the {\it principal} series of unipotent representations of $\GC$. 

If $\t=\Id_\WC$ (i.e., if $\Gb/\Zrm(\Gb)$ is split), then the parameter 
$k_\TC$ of Theorem~\ref{theo:hc}(b) is given by $(k_\TC)_{H,0}=1$ and 
$(k_\TC)_{H,1}=0$ for any reflecting hyperplane $H$ of $\WC$. This parameter 
will be denoted by $\ksp$ in the next part.\finl
\end{exemple}

\bigskip

\subsection{Deligne-Lusztig induction} 
For going further, we need to recall the construction of Deligne-Lusztig induction. 
If $\Pb$ is a (not necessarily $F$-stable) parabolic subgroup of $\Gb$ 
admitting an $F$-stable Levi complement $\Lb$, then we define the 
variety $\Yb_\Pb$ (also called a {\it Deligne-Lusztig variety}) by
$$\Yb_\Pb=\{g \Ub_\Pb \in \Gb/\Ub_\Pb~|~g^{-1}F(g) \in \Ub_\Pb \cdot F(\Ub_\Pb)\},$$
where $\Ub_\Pb$ denotes the unipotent radical of $\Pb$. Then $\Yb_\Pb$ 
inherits a left action of $\Gb^F$ and a right action of $\Lb^F$ which commute, 
endowing the $\ell$-adic cohomology groups with compact support $\Hrm_c^j(\Yb_\Pb)$ with 
a structure of a $(\qlb \Gb^F,\qlb \Lb^F)$-bimodule. This allows us to define 
a map $R_{\Lb \subset \Pb}^\Gb : \ZM\Irr(\Lb^F) \longto \ZM\Irr(\Gb^F)$ 
between Grothendieck groups by the formula
$$R_{\Lb \subset \Pb}^\Gb(\isomorphisme{M}) = 
\sum_{k \ge 0} (-1)^j \isomorphisme{\Hrm_c^j(\Yb_\Pb) \otimes_{\qlb \Lb^F} M}.$$
This map is called the {\it Deligne-Lusztig induction} and is the shadow 
of a functor $\RC_{\Lb \subset \Pb}^\Gb : \Drm^b(\qlb \Lb^F\module) \longto 
\Drm^b(\qlb\Gb^F\module)$ between bounded derived categories, which is defined by 
$$\RC_{\Lb \subset \Pb}^\Gb(M) = \Hrm^\bullet_c(\Yb_\Pb) \otimes_{\qlb \Lb^F} M.$$
Note that we work in a semisimple world, so that any complex of $\qlb \G$-modules 
(where $\G$ is a finite group) is quasi-isomorphic 
to its cohomology: here, $\Hrm^\bullet_c(\Yb_\Pb)$ is viewed as a complex 
of $(\qlb \Gb^F,\qlb \Lb^F)$-bimodules 
whose $j$-th term is $\Hrm^j_c(\Yb_\Pb)$ and whose differential is zero. 

If the parabolic subgroup $\Pb$ is $F$-stable, then the Deligne-Lusztig 
functor $\RC_{\Lb \subset \Pb}^\Gb$ is just the functor induced by the Harish-Chandra 
functor at the level of derived categories (this justifies the use of the same notation). 
Both $R_{\Lb \subset \Pb}^\Gb$ and $\RC_{\Lb \subset \Pb}^\Gb$ admits 
adjoints (in two different significations of {\it adjoint}) which we denote 
by $\lexp{*}{R}_{\Lb \subset \Pb}^\Gb$ and $\lexp{*}{\RC}_{\Lb \subset \Pb}^\Gb$ 
and which are defined by
$$\lexp{*}{\RC_{\Lb \subset \Pb}^\Gb}(N) = 
\Hrm^\bullet_c(\Yb_\Pb)^* \otimes_{\qlb \Gb^F} N$$
$$\lexp{*}{R_{\Lb \subset \Pb}^\Gb}(\isomorphisme{N}) = 
\sum_{k \ge 0} (-1)^j \isomorphisme{\Hrm_c^j(\Yb_\Pb)^* \otimes_{\qlb \Gb^F} N}.
\leqno{\text{and}}$$

\bigskip

\subsection{${\boldsymbol{d}}$-Harish-Chandra theory} 
We fix a natural number $d \ge 1$ and we denote by $\z_d$ a primitive $d$-th root of unity. 
The {\it $d$-Harish-Chandra theory} 
of Brou\'e-Malle-Michel~\cite{BMM} is an analogue of Harish-Chandra theory 
where the $F$-stable Levi subgroups of $F$-stable parabolic subgroups 
are replaced by a larger class of $F$-stable Levi subgroups, using Deligne-Lusztig 
induction instead of Harish-Chandra induction. Whenever $d=1$, we retrieve 
the usual Harish-Chandra theory. Let us summarize it.

Let $\Sb$ be a torus over $\FM$, endowed with a Frobenius endomorphism $F : \Sb \to \Sb$ 
associated with an $\gfq$-structure on $\Sb$. Let $\Phi_d$ denote the $d$-th cyclotomic 
polynomial. Then $\Sb$ is called a {\it $\Phi_d$-torus} 
if the following two conditions are satisfied:
\begin{itemize}
\item[$\bullet$] $\Sb$ is split over $\FM_{\! q^d}$.

\item[$\bullet$] If $\Sb'$ is an $F$-stable subtorus of $\Sb$ different from $1$, 
and if $e$ divides $d$ and is different from $d$, then $\Sb'$ is not 
split over $\FM_{\! q^e}$.
\end{itemize}
A Levi subgroup $\LC=(q,\Lb,F)$ of $\GC$ is called {\it $d$-split} if $\Lb$ is the centralizer 
of an $F$-stable $\Phi_d$-torus of $\Gb$. 

\begin{rema}\label{rem:d=1}
Let $\Lb$ be an $F$-stable Levi subgroup of $\Gb$. Then $(q,\Lb,F)$ is $1$-split 
if and only if $\Lb$ is the Levi complement of an $F$-stable parabolic sugroup 
of $\Gb$.\finl
\end{rema}

An irreducible unipotent representation $N$ of $\Gb^F$ is called 
{\it $d$-cuspidal} if $\lexp{*}{R}_{\Lb \subset \Pb}^\Gb \isomorphisme{N}=0$ 
for any pair $(\Lb,\Pb)$ where $\Pb$ is a parabolic subgroup of $\Gb$ 
and $\Lb$ is an $F$-stable Levi complement of $\Pb$ which is 
$d$-split as a Levi subgroup of $\Gb$. 
We denote by $\unipc_\cus^d(\GC)$ the set of (isomorphism classes of) 
$d$-cuspidal irreducible unipotent representations of $\Gb^F$. 

We denote by $\CUSC^d(\GC)$ 
the set of pairs $(\LC,\l)$, where $\LC=(q,\Lb,F)$ is a $d$-split Levi subgroup 
of $\Gb$ and $\l$ 
is a $d$-cuspidal irreducible unipotent representation of $\Lb^F$. 
We denote by $\unipc(\GC,\LC,\l)$ the set of unipotent 
irreducible representations occuring in $R_{\Lb \subset \Pb}^\Gb \isomorphisme{\l}$, 
where $\Pb$ is a parabolic subgroup admitting 
$\Lb$ as a Levi complement: this set is called the {\it $d$-Harish-Chandra series} 
associated with the pair $(\LC,\l)$. 
This set depends on the pair $(\LC,\l)$ up to $\Gb^F$-conjugacy. 
We denote by $\CUSC^d(\GC)/\!\!\sim$ the set of $\Gb^F$-conjugacy 
classes of elements of $\CUSC^d(\GC)$. 
The {\it $d$-Harish-Chandra theory} can then be summarized as follows~\cite[Theo.~3.2]{BMM}:

\begin{theo}[Brou\'e-Malle-Michel]\label{theo:d-hc}
We have
$$\unipc(\GC)=\dot{\bigcup_{(\LC,\l) \in \CUSC^d(\GC)/\sim}} 
\unipc(\GC,\LC,\l),$$
where $\dot{\cup}$ means a disjoint union. 

Moreover, if $(\LC,\l) \in \CUSC^d(\GC)$ with $\LC=(q,\Lb,F)$ and if $\Pb$ is a 
parabolic subgroup admitting $\Lb$ as a Levi complement, then:
\begin{itemize}
\itemth{a} The group $W_\GC(\LC,\l)$ is a complex reflection group 
for its action on 
$$\{y \in \CM \otimes_\ZM Y(\Zrm(\Lb))~|~F(y)=\z_d q y\}.$$

\itemth{b} There exists a bijection 
$\harishc_d^{\GC,\LC,\l} : \Irr W_\GC(\LC,\l) \longiso \unipc(\GC,\LC,\l)$ 
and a sign function $\Irr W_\GC(\LC,\l) \longto \{1,-1\}$ 
intertwining ordinary induction and Lusztig induction~\cite{BMM}\footnote{Sorry for being somewhat vague 
in this survey.}. 
In other words,
$$\unipc(\GC)=\dot{\bigcup_{(\LC,\l) \in \CUSC^d(\GC)/\sim}} 
\harishc_d^{\GC,\LC,\l}\bigl(\Irr W_\GC(\LC,\l) \bigr).$$

\end{itemize}
\end{theo}

Note that this theorem sounds like a miracle and is proved through a case-by-case 
analysis: it would be better explained 
by the following conjecture (which is true for $d=1$ by Theorem~\ref{theo:hc}).

\medskip

\begin{quotation}
\begin{conj}[Brou\'e-Malle-Michel]\label{conj:BMM}
Let $(\LC,\l) \in \CUSC^d(\GC)$ with $\LC=(q,\Lb,F)$. Then there exists a parabolic subgroup 
$\Pb$ of $\Gb$, admitting $\Lb$ as a Levi complement, and such that:
\begin{itemize}
\itemth{a} The $\qlb\Gb^F$-modules $\Hrm_c^j(\Yb_\Pb) \otimes_{\qlb\Lb^F} \l$ 
and $\Hrm_c^{j'}(\Yb_\Pb) \otimes_{\qlb\Lb^F} \l$ have no common irreducible 
constituent if $j \neq j'$.

\itemth{b} The endomorphism algebra of the complex of $\qlb\Gb^F$-modules 
$\RC_{\Lb \subset \Pb}^\Gb(\l)$ is isomorphic to 
a Hecke algebra $\HC_{k_{\LC,\l}}(W_\GC(\LC,\l),\z_d^{-1}q)$ for some 
parameter $k_{\LC,\l} \in \aleph(W_\GC(\LC,\l))$.
\end{itemize}
\end{conj}
\end{quotation}

\bigskip

Of course, Conjecture~\ref{conj:BMM} can easily be reduced to the case where 
$\Gb$ is quasi-simple. In this case, and for $d> 1$, the full Conjecture~\ref{conj:BMM} is known 
only whenever $d$ is the Coxeter number (see Lusztig's fundamental paper~\cite{lusztig coxeter}, 
which served as an inspiration for the conjecture) or whenever $\Gb$ is 
of type $A_d$ (see~\cite{dm}). Part~(a) of Conjecture~\ref{conj:BMM} is known 
if $\Gb$ is of type $A$ and $\LC$ is a torus (see~\cite{bonnafe dudas rouquier}, 
which relies in an essential way on the work of Dudas~\cite{dudas}). For this part~(a), 
some other cases have been solved by Digne-Michel-Rouquier~\cite{DMR} and Digne-Michel~\cite{dm}. 

For part~(b), note also that, in many important cases, a map from the group algebra of 
the braid group of $W_\GC(\LC,\l)$ to the endomorphism 
algebra of the complex of $\qlb\Gb^F$-modules 
$\RC_{\Lb \subset \Pb}^\Gb(\l)$ has been constructed~\cite{broue michel, broue malle, dm}, 
but it is generally not known whether it is onto and if it factorizes (exactly) through 
the Hecke algebra $\HC_{k_{\LC,\l}}(W_\GC(\LC,\l),\z_d^{-1}q)$. There are, however, 
partial results in this direction~\cite{broue michel, broue malle, dm}.

Also, extra-properties 
that should be satisfied by the Hecke algebra involved in Conjecture~\ref{conj:BMM} 
imply that, if $\chi \in \Irr(W_\GC(\LC,\l))$, then
\equat\label{eq:deg hc}
\deg \Hrm^{\GC,\LC,\l}(\chi)=\pm \frac{R_{\Lb \subset \Pb}^\Gb(\l)(1)}{\schur_\chi^{k_{\LC,\l}}}.
\endequat
This imposes huge constraints on the parameter $k_{\LC,\l}$ and allows to determine it 
explicitly in almost all cases~\cite{BMM} (including the classical groups). 

\bigskip

\begin{rema}\label{rem:blocs}
Let $\ell$ be a prime number different from $p$ and assume for simplicity 
that $\ell \ge 5$ and $\ell$ is {\it very good} for $\Gb$. Let $d$ denote the 
smallest integer such that $\ell$ divides $q^d-1$. Then two irreducible 
unipotent representations of $\Gb^F$ belong to the same $\ell$-block 
if and only if they belong to the same $d$-Harish-Chandra series 
(see~\cite[Theo.~5.24]{BMM} for the case where $\ell$ does not divide 
the order of the Weyl group and~\cite[Theo.~22.4]{CE} for the general 
case).\finl
\end{rema}

\bigskip

\section{Families}

\medskip

\subsection{Almost characters}\label{sub:almost}
If $\chi$ is a $\t$-stable 
irreducible character of $\WC$, we fix once and for all 
an extension $\chit$ of $\chi$ to $\WC \rtimes \langle \t \rangle$. 
If $w \in \WC$, we set $\Tb_w=g_w \Tb g_w^{-1}$, where $g_w \in \Gb$ 
is chosen so that $g_w^{-1}F(g_w) \in \Nrm_\Gb(\Tb)$ is a representative of $w$. 
Then $\Tb_w$ is an $F$-stable maximal torus so, using Deligne-Lusztig induction, 
we can define
\equat
R_\chi^\GC=
\frac{1}{|\WC|} \sum_{w \in \WC} \chit(w\t) R_{\Tb_w}^\Gb(1_{\Tb_w^F}) \quad \in \CM\unipc(\GC).
\endequat
Then $R_\chi^\GC$ is called an {\it almost character} of $\Gb^F$. Note that 
$R_\chi^\GC$ depends on the choice of $\chit$, but only up to 
multiplication by a root of unity. 

\bigskip
\def\Gr{{\mathfrak{Gr}}}

\subsection{Families}\label{sub:familles}
Lusztig~\cite[Chap.~4]{lusztig orange} has defined a partition of $\unipc(\GC)$ into 
{\it families}: let us recall 
his construction. Define a graph $\Gr_\GC$ on 
$\Irr(\WC)^\t$ as follows:
\begin{itemize}
\itemth{G1} The set of vertices of $\Gr_\GC$ is $\Irr(\WC)^\t$.

\itemth{G2} Two $\t$-stable irreducible characters $\chi$ and $\chi'$ 
are linked by an edge in the graph $\Gr_\GC$ if $R_\chi^\GC$ and $R_{\chi'}^\GC$ 
have a common irreducible constituent.
\end{itemize}
If $\CG$ is a connected component of $\Gr_\GC$, we denote by $\FG_\CG^\GC$ the 
set of irreducible unipotent representations $\g \in \unipc(\GC)$ such that 
$\langle R_\chi^\GC,\g \rangle_{\Gb^F} \neq 0$ for some $\chi \in \CG$. The subset 
$\FG_\CG^\GC$ of $\unipc(\GC)$ is called a {\it unipotent Lusztig family} of $\GC$. 
We denote by $\Famuni(\GC)$ the set of 
such families. By construction, 
the unipotent Lusztig families form a partition of $\unipc(\GC)$.

One of the main results in Lusztig's work on unipotent representations 
is the list of following compatibilities 
between this partition and Harish-Chandra series~\cite{lusztig orange} 
(some of them are proved by a case-by-case analysis):

\bigskip

\begin{theo}[Lusztig]\label{theo:familles}
With the above notation, we have:
\begin{itemize}
\itemth{a} If $(\LC,\l) \in \CUSC(\GC)$ and $\FG \in \Famuni(\GC)$, then 
$(\harishc^{\GC,\LC,\l})^{-1}(\FG)$ is empty or belongs to $\Fam_{k_\LC}^\lus(W_\GC(\LC))$ 
(recall that $W_\GC(\LC)=W_\GC(\LC,\l)$ and that $k_\LC=k_{\LC,\l}$ does not 
depend on $\l$).

\itemth{b} If $\unipc_\cus(\GC)$ is non-empty, then it is 
contained in a single family, which will be denoted by $\FG_\cus^\GC$. 

\itemth{c} If $\t=\Id_V$, then the principal series (see Example~\ref{ex:principal}) 
satisfies the following property: for any $\FG \in \Famuni(\GC)$ and any $\chi \in \Irr(\WC)$, 
then $R_\chi^\GC \in \CM\FG$ if and only if $\harish^\GC(\chi) \in \FG$. In particular, 
every family meets the principal series.
\end{itemize}
\end{theo}

\bigskip

Note that the analogue of statement~(a) for $d$-Harish-Chandra theory 
(instead of classical Harish-Chandra theory) is probably true (by replacing Lusztig 
families Calogero-Moser families) but it is 
still not known up to now. The analogue of statement~(b) for 
$d$-Harish-Chandra theory is false in general (for instance, if $d$ is large enough, 
then all irreducible unipotent representations are $d$-cuspidal). The analogue of statement~(c) 
for $\t \neq \Id_V$ is false in general (for instance, in twisted type $A_{n-1}$ with $n \ge 3$, 
every family is a singleton but there are irreducible unipotent representations not belonging 
to the principal series).

\bigskip

\part{Genericity vs Calogero-Moser spaces}
\label{part:generique}

\boitegrise{{\bf Hypothesis.} {\it We assume in this third part 
that there exists a rational structure $V_\QM$ 
on $V$ which is stable under the action of $W$ (i.e., $W$ 
is a Weyl group). We also assume that $V^W=0$. We denote 
by $\ksp \in \CM^\aleph$ the {\bfit spetsial} parameter, that is, 
the parameter such that $(\ksp)_{H,0}=1$ and $(\ksp)_{H,1}=0$ 
for all $H \in \AC$. \\
\hphantom{aa} We also fix an element $\t \in \Nrm_{\Gb\Lb_\QM(V_\QM)}(W)$ 
of finite order.}}{0.75\textwidth}

\bigskip

We denote by $\groups(W\t)$ the class of triples $(q,\Gb,F) \in \groups$ such that, 
if $\Tb$ is any $F$-stable maximal torus of $\Gb$, then 
the pair $(\QM \otimes Y(\Tb/\Zrm(\Gb)),N_\Gb(\Tb)/\Tb)$ 
is isomorphic to $(V_\QM,W)$ and, moreover, $\t$ stabilizes 
the lattice $Y(\Tb/\Zrm(\Gb))$ of $V_\QM$ and there exists 
$w \in W$ such that $F(y)=qw\t(y)$ for all $y \in Y(\Tb/\Zrm(\Gb))$. 

\bigskip

This part may be viewed as the aim of this survey article, where we propose  
several conjectures which compare the geometry (fixed points, symplectic leaves) 
of the Calogero-Moser space $\ZC_\ksp$ 
with the different partitions (families, $d$-Harish-Chandra series) 
of unipotent characters of triples belonging to $\groups(W\t)$. 
A first general remark (due to Lusztig) is that most of these partitions 
do not depend that much on the triple $\GC \in \groups(W\t)$: 
they mainly depend only on the coset $W\t$. This phenomenon, called {\it genericity}, 
was developed and formalized by Brou\'e-Malle-Michel~\cite{BMM} and will be 
explained in Section~\ref{sec:genericity}.

The conjectures will be stated precisely in Section~\ref{sec:main}. If $\GC \in \groups(W\t)$, 
they propose conjectural links between:
\begin{itemize}
\item[$\bullet$] the partition of $\unipc(\GC)$ into families and the $\CM^\times$-fixed points 
in $\ZC_\ksp^\t$;

\item[$\bullet$] the partition into $d$-Harish-Chandra series and symplectic leaves 
of $\ZC_\ksp^{\z_d \t}$, where $\z_d$ is a primitive $d$-th root of unity.
\end{itemize}
In the second point, the most spectacular conjecture relates the 
parameter involved in the description of the normalization of the 
closure of a symplectic leaf as a Calogero-Moser space (i.e. the parameter $k_{P,p}$ 
of Conjecture~\ref{conj:leaves}) and the parameter of the Hecke algebra 
which conjecturally describes the endomorphism algebra of the cohomology 
of some Deligne-Lusztig variety. 

\bigskip

\section{Genericity}\label{sec:genericity}

\bigskip
\def\Ord{{\mathrm{Ord}}}

\subsection{Rough definition} 
The notions of {\it generic groups}, {\it generic unipotent representations}... have 
been defined rigorously in~\cite{BMM}. In this survey, we will not recall 
this precise definition, which would require to introduce again much more notation. 
We will use throughout this part a rather vague definition: when 
some structure associated with any $\GC = (q,\Gb,F) \in \groups(W\t)$ depends only 
on $W\t$ and not on the triple $\GC$, we will say that this structure 
behaves {\it generically}. 

A first example is the order of $\Gb^{\prime F}$, 
where $\Gb'$ is the derived subgroup of $\Gb$. 
Indeed, if $m=\dim_\CM V$, there exists a choice of algebraically independent homogeneous 
generators $f_1$,\dots, $f_m$ of $\CM[V]^W$ which are eigenvectors 
for the action of $\t$ (and we denote by $d_j$ the degree of $f_j$ 
and by $\xi_j$ the eigenvalue 
corresponding to $f_j$). We can then define the following polynomial 
$\Ord_{W\t}(\qb) \in \QM[\qb]$:
$$\Ord_{W\t}(\qb)=q^{|\AC|} \prod_{j=1}^m (\qb^{d_j}-\xi_j).$$
Then this polynomial does not depend on the precise choice of the $f_j$'s, and 
\equat\label{eq:order-g}
|\Gb^{\prime F}|=\Ord_{W\t}(q)
\endequat
for all $(q,\Gb,F) \in \groups(W\t)$. 

\bigskip

\begin{rema}[${\boldsymbol{d}}$-splitness]\label{rem:d-split-generic}
Let $d \ge 1$ and let $\z_d$ denote a primitive $d$-th root of unity. We set 
$$\d(d)=\max_{w \in W} \dim V^{\z_d w\t}$$
and we denote by $w_d$ an element of $W$ such that 
$\dim V^{\z_d w_d \t}=\d(d)$. Recall~\cite{springer} that $\d(d)$ is the number 
of $j \in \{1,2,\dots,m\}$ such that $\z_d^{d_j}=\xi_j$. 

Then $w_d \t$ is well-defined up to 
$W$-conjugacy~\cite{springer} and any subspace of the form 
$V^{\z_d w \t}$ for some $w \in W$ is contained in a subspace of the form 
$x(V^{\z_d w_d \t})$ for some $x \in W$. We set $\t_d =\z_dw_d\t \in \Nrm_{\Gb\Lb_\CM(V)}(W)$. 
Note that this choice of $w_d$ 
implies that the element $\t_d$ is $W$-full~\cite{springer}. Note also that 
$\ZC_\ksp^{\t_d}=\ZC_\ksp^{\z_d \t}$. 

Then, for any $\GC \in \groups(W\t)$, 
the conjugacy classes of $d$-split Levi subgroups of $\GC$ are in bijection with 
the $W_{\t_d}$-orbits of $\t_d$-split parabolic subgroups of $W$: the correspondence 
assigns to the conjugacy class of the $d$-split Levi subgroup 
its Weyl group, suitably embedded in $W$ (see~\cite{broue malle french} for details).\finl
\end{rema}

\bigskip

\subsection{Genericity of unipotent representations}
For our purpose, the most important result about genericity is the following 
theorem, which says that the unipotent representations  
of $\GC \in \groups(W\t)$ and their degree 
behave generically. 

\bigskip

\begin{theo}[Lusztig]\label{theo:lusztig-generique}
There exists a finite set $\unip(W\t)$ endowed with a map $\degb_{W\t} : \unip(W\t) \longto \QM[\qb]$, both 
depending only on the coset $W\t$ and 
such that, for each triple $\GC \in \groups(W)$, 
there exists a well-defined bijection 
$$\rhob^\GC : \unip(W\t) \longiso \unipc(\GC)$$
satisfying 
$$\deg~\rhob^\GC_\g = (\degb_{W\t}\,\g)(q)$$
for all $\g \in \unip(W)$.
%
%
\end{theo}

\bigskip

This Theorem~\ref{theo:lusztig-generique} 
follows from the classification of unipotent representations obtained by Lusztig~\cite{lusztig orange} 
and a case-by-case analysis. More recent works of Lusztig~\cite{lusztig centre} provide 
general explanations for the existence of the finite set $\unip(W\t)$ and 
the bijection $\rhob^\GC$ but do not explain the polynomial behaviour 
of the degree of the unipotent representations.

It turns out that the different structures on $\unipc(\GC)$ 
(families, Harish-Chandra series, partitions into $\ell$-blocks...) 
can also be read only from the finite set 
$\unip(W\t)$, as it will be explained below. In other words, they 
behave generically. 
Our aim here is to provide numerical evidences that 
these extra-structures can be read from the geometry 
of the Calogero-Moser space $\ZC_\ksp$ ($\CM^\times$-fixed points, 
symplectic leaves, fixed points under the action of $\mub_d$...). 

\bigskip

%
%
%

\bigskip

\subsection{Almost characters}
We denote by $\CM\unip(W\t)$ the set of formal $\CM$-linear 
combinations of elements of $\unip(W\t)$ and we still denote 
by $\rhob^\GC : \CM\unip(W\t) \longto \CM\unipc(\GC)$ the $\CM$-linear 
extension of the bijection $\rhob^\GC$. 

With this notation, the almost characters behave generically. In other words, 
there exists a (necessarily unique) 
family $(R_\chi)_{\chi \in \Irr(W)^\t}$ of elements of $\CM\unip(W\t)$ 
such that
\equat\label{eq:almost-generique}
\rhob^\GC(R_\chi)=R_\chi^\GC
\endequat
for any $\GC \in \groups(W\t)$ (see~\cite[Chap.~4]{lusztig orange}). In particular, 
the graph $\Gr_\GC$ constructed in~\S\ref{sub:familles} is generic: 
it will also be denoted by $\Gr_{W\t}$. 

\bigskip

\subsection{Families} 
Families of unipotent characters behave generically. Indeed, it follows from~\eqref{eq:almost-generique} 
that there exists a partition of $\unip(W\t)$ into {\it unipotent Lusztig families} 
(we denote by $\Fam_\uni(W\t)$ the set of such families) such that, 
\equat\label{eq:fam-gen}
\rhob^\GC(\Fam_\uni(W\t))=\Famuni(\GC)
\endequat
for any $\GC \in \groups(W\t)$. If $\CG$ is a connected component of the graph $\Gr_{W\t}$, 
we denote by $\FG_\CG^\uni \subset \unip(W\t)$ the associated generic family. 
Lusztig proved the following important result~\cite{lusztig orange}:

\bigskip

\begin{theo}[Lusztig]\label{theo:familles-ksp}
The map 
$$\fonctio{\Fam_\ksp^\lus(W)^\t}{\Fam_\uni(W\t)}{\CG}{\FG_{\CG^\t}^\uni}$$
is well-defined and bijective.
\end{theo}

\bigskip

This Theorem contains in particular the fact that, if $\CG$ is a $\t$-stable 
Lusztig $\ksp$-family of characters of $W$, then $\CG^\t \neq \vide$. But this 
follows directly from the fact that every Lusztig $\ksp$-family contains 
a unique character with minimal $b$-invariant (the {\it $b$-invariant} 
of an irreducible character $\chi$ is the minimal value of $j$ such that $\chi$ 
occurs in the $j$-th symmetric power $\Srm^j(V)$), called the {\it special} character 
of the family~\cite[\S{12}]{lusztig special}: if the family is $\t$-stable, then its special 
character is necessarily $\t$-stable.

\bigskip

\subsection{Lusztig's $\ab$-function}
If $\g \in \unip(W\t)$, we denote by $a_\g$ (resp. $A_\g$) the valuation 
(resp. the degree) of the polynomial $\degb_{W\t}\, \qb$. It follows from 
Lusztig's work~\cite{lusztig orange} that
\equat\label{eq:a constant}
\text{\it the functions $a$, $A : \unip(W\t) \longto \ZM_{\geqslant 0}$ are constant 
on families.}
\endequat

\bigskip

\subsection{${\boldsymbol{d}}$-Harish-Chandra series} 
It turns out that $d$-Harish-Chandra theory behaves also generically~\cite[Theo.~3.2]{BMM}. 
More precisely, if $\GC=(q,\Gb,F) \in \groups(W\t)$, then:
\begin{itemize}
\item[$\bullet$] Let $\unip_\cus^d(W\t)$ denote the set of $\g \in \unip(W\t)$ 
such that $\degb_{W\t}\,\g$ is divisible by $(\qb-\z_d)^{\dim V^{\t_d}}$. 
Then $(\rhob^\GC)(\unip_\cus^d(W\t))=\unipc_\cus^d(\GC)$ 
(see~\cite[Prop.~2.9]{BMM}). 

\item[$\bullet$] The $\Gb^F$-conjugacy classes of $d$-split Levi subgroups of 
$\GC$ are in bijection with the $W_{\t_d}$-conjugacy classes of $d$-split 
parabolic subgroups of $W$ (see Remark~\ref{rem:d-split-generic}): if the $d$-split Levi subgroup $\LC$ 
of $\GC$ corresponds to the $d$-split parabolic subgroup $P$ 
under this bijection, then $\LC \in \groups(P\t_d)$ 
and $W_\GC(\LC) \simeq \Nrmov_{W_{\t_d}}(P_{\t_d})$ 
(see~\cite[Prop.~3.12]{broue malle french}).

Through this isomorphism and the bijection $\rhob^\LC$, the group $\Nrmov_{W_{\t_d}}(P_{\t_d})$ acts 
on $\unip(P\t_d)$ and stabilizes the subset 
$\unip_\cus^d(\LC)$. Moreover, the action of 
$\Nrmov_{W_{\t_d}}(P_{\t_d})$ on $\unip(P\t_d)$ is generic. If $\l \in \unip(P\t_d)$, we denote by 
$\Nrmov_{W_{\t_d}}(P_{\t_d},\l)$ 
its stabilizer in $\Nrmov_{W_{\t_d}}(P_{\t_d})$. 

\item[$\bullet$] If we denote by $\Cus^d(W\t)$ the set of pairs 
$(P,\l)$ where $P$ is a $d$-split parabolic subgroup of $W$ and 
$\l \in \unip_\cus^d(P\t_d)$, then the previous point defines a natural bijection 
between $\Cus^d(W\t)/\sim$ and $\CUSC^d(\GC)/\sim$.

\item[$\bullet$] If $(P,\l) \in \Cus^d(W\t)$ corresponds to 
$(\LC,\rhob_\l^\LC) \in \CUSC^d(\GC)$, then the 
maps $\harishc_d^{\GC,\LC,\rhob_\l^\LC}$ and $\rhob^\GC$ 
define an injection $\harish_d^{W,P,\l} : \Irr(\Nrmov_{W_{\t_d}}(P_{\t_d},\l)) \injto \unip(W\t)$ 
which behaves generically. Its image is denoted by $\unip_d(W\t,P,\l)$. Then
\equat\label{eq:hc-generique}
\unip(W\t) = \dot{\bigcup_{(P,\l) \in \CUSC^d(W)}}
\harish_d^{W,P,\l}(\Irr(\Nrmov_{W_{\t_d}}(P_{\t_d},\l))).
\endequat
Moreover, the parameter $k_{\LC,\rhob_\l^\LC}$ in Conjecture~\ref{conj:BMM} is generic, 
i.e. depends only on $(P,\l)$. It will be denoted by $k_{P,\l}$.
\end{itemize}
Here, all the statements stated without reference can be found in~\cite[Theo.~3.2]{BMM}.

\bigskip

\begin{rema}\label{rem:grc}
Let $(P,\l) \in \Cus^d(W\t)$. 
It follows from the classification of such pairs (see~\cite{BMM}) that:
\begin{itemize}
\item[$\bullet$] $\Nrmov_{W_{\t_d}}(P_{\t_d},\l)$ is always a reflection group 
for its action on $(V^P)^{\t_d}$.

\item[$\bullet$] Examples where $\Nrmov_{W_{\t_d}}(P_{\t_d},\l) \neq \Nrmov_{W_{\t_d}}(P_{\t_d})$ 
are very rare. For instance, this never happens if $d=1$ (see Theorem~\ref{theo:hc}(a)) 
or if $W$ is of type $A$ (see~\cite[\S{3.A}]{BMM}).\finl
\end{itemize}
\end{rema}

\bigskip

\begin{exemple}[Principal series]\label{ex:principal generic}
Let us describe the generic version of Example~\ref{ex:principal}. 
First, $\unip(1)$ consists of a single element 
that we may (and will) denote by $1$. Then the cuspidal pair $(1,1) \in \Cus(W\t)$ 
corresponds to the pair $(\TC,1) \in \CUSC(\GC)$ associated with 
an $F$-stable maximal torus of an $F$-stable Borel subgroup and the map $\harish^\GC=\harish^{\GC,\TC,1}$ 
will be simply denoted by $\harish^W : \Irr(W^\t) \longinjto \unip(W\t)$, 
instead of $\harish^{W,1,1}$.

If $\t=\Id_V$, then the parameter $k_{1,1}$ is equal to $\ksp$ (see Example~\ref{ex:principal}).\finl
\end{exemple}

\bigskip

\subsection{${\boldsymbol{d}}$-Harish-Chandra theory and filtration} 
Assume in this subsection, and only in this subsection, that $\t=\Id_V$ 
(on the reductive group side, this means that we work in the split case). 
Let $\Zrm(\CM W)^\lus$ denote the subalgebra of $\Zrm(\CM W)$ whose basis 
is given by $(e_\CG^W)_{\CG \in \Fam_\ksp^\lus(W)}$. Let $(P,\l) \in \Cus^d(W)$. 
We define a morphism of algebras
$$(\harish_d^{W,P,\l})^\# : \Zrm(\CM W)^\lus \longto \Zrm(\CM \Nrmov_{W_{\t_d}}(P_{\t_d},\l))$$
by
$$(\harish_d^{W,P,\l})^\#(e_\CG^W) = 
\sum_{\substack{\chi \in \Irr \Nrmov_{W_{\t_d}}(P_{\t_d},\l) \\ 
\text{such that~}
\harish_d^{W,P,\l}(\chi) \in \FG_\CG^\uni}} e_\chi^{\Nrmov_{W_{\t_d}}(P_{\t_d},\l)}.$$

\medskip

\begin{quotation}
\begin{conj}\label{conj:d-hc-filtration}
Assume that $\t=\Id_V$ and fix $(P,\l) \in \Cus^d(W)$. Then
$$(\harish_d^{W,P,\l})^\#(\FC_j \Zrm(\CM W)^\lus) \subset \FC_j \Zrm(\CM \Nrmov_{W_{\t_d}}(P_{\t_d},\l))$$
for all $j$.
\end{conj}
\end{quotation}

\medskip

\begin{rema}\label{rem:d-hc-filtration}
This conjecture seems to come from nowhere and provides a strange link between 
the character tables of $W$ and $\Nrmov_{W_{\t_d}}(P_{\t_d},\l)$. However, if we believe in the links between 
representation theory of finite reductive groups and geometry of Calogero-Moser spaces 
(as developed in the next section), then this Conjecture~\ref{conj:d-hc-filtration} 
is just a consequence of this philosophy (see the upcoming Proposition~\ref{prop:fil}) 
and of Conjecture~\ref{conj:coho}. 
This is an example where the geometry of Calogero-Moser spaces suggests 
unexpected properties of the representation theory of finite reductive groups.

Note that Conjecture~\ref{conj:d-hc-filtration} holds in the following cases:
\begin{itemize}
\item[$\bullet$] If $W$ is of type $A$ (see~\cite[Cor.~4.22]{bonnafe maksimau} and 
the explanations given in Section~\ref{sec:type a}).

\item[$\bullet$] A pretty convincing result is that it holds if $\dim V \le 8$ (so this includes 
the type $E_8$): this has been checked through computer 
calculations based on all the functions implemented by 
Jean Michel~\cite{jean}\footnote{We wish to thank again warmly Jean Michel for writing the programs 
for performing these calculations.}.\finl
\end{itemize}
\end{rema}

\medskip

\begin{exemple}\label{ex:filtration bmm}
Let $z=\sum_{s \in \Ref(W)} s \in \Zrm(\CM W)$. Then 
$z=\sum_{\chi \in \Irr(W)} (|\AC|-(a_\chi^{(\ksp)}+A_\chi^{(\ksp)}) e_\chi^W$ 
(see~\cite[Cor.6.9]{spetses 1} and~\cite[Lem.~7.2.1]{calogero}).
So $z \in \FC_1 \Zrm(\CM W)^\lus$ by~\eqref{eq:a constant hecke}. 
Now, if $\psi \in \Irr \Nrmov_{W_{\t_d}}(P_{\t_d},\l)$ is such that 
$\harish_d^{W,P,\l}(\psi)$ belongs to the same family as 
$\harish^W(\chi)$, then it follows from~\eqref{eq:deg hc} that 
$|\AC|-(a_\chi^{(\ksp)}-A_\chi^{(\ksp)})=M-(a_\psi^{(k_{P,\l})}+A_\psi^{(k_{P,\l})})$ 
for some $M$ which does not depend on $\psi$ or $\chi$ (and only on $(W,P,\l)$). 
Therefore, 
$$(\harish_d^{W,P,\l})^\#(z)=\sum_{\chi \in \Irr \Nrmov_{W_{\t_d}}(P_{\t_d},\l)} 
(M-(a_\psi^{(k_{P,\l})}+A_\psi^{(k_{P,\l})})) e_\psi^{\Nrmov_{W_{\t_d}}(P_{\t_d},\l)}.$$
In other words, it follows from~\cite[Cor.6.9]{spetses 1} and~\cite[Lem.~7.2.1]{calogero} 
that there exists $M' \in \CM$ such that 
$$(\harish_d^{W,P,\l})^\#(z)=M' + \sum_{s \in \Ref(\Nrmov_{W_{\t_d}}(P_{\t_d},\l))} c_{k_{P,\l}}(s) s 
\in \FC_1 \Zrm(\CM \Nrmov_{W_{\t_d}}(P_{\t_d},\l)),$$
as desired.\finl
\end{exemple}

\section{Coincidences, conjectures}\label{sec:main}

\medskip

\subsection{Families}
By Theorem~\ref{theo:familles}(b), the set $\Fam_\uni(W\t)$ 
is in bijection with the set $\Fam_\ksp^\lus(W)^\t$. We conjecture the 
first link between the geometry of $\ZC_\ksp$ and the representation 
theory of finite reductive groups:

\medskip

\begin{quotation}
\begin{conj}\label{conj:familles-uni-cm}
There exists a unique bijection 
$$\Phi : (\ZC_\ksp^{\CM^\times})^\t \longiso \Fam_\uni(W\t)$$
such that, for any $p \in (\ZC_\ksp^{\CM^\times})^\t$ 
and for any $\chi \in (\FG_p^\ksp)^\t$, the almost character $R_\chi$ 
belongs to $\CM \Phi(p)$.
\end{conj}
\end{quotation}

\bigskip

Whenever $\t=\Id_V$, this Conjecture~\ref{conj:familles-uni-cm} is equivalent 
to Gordon-Martino Conjecture~\ref{conj:gordon-martino} (see Theorem~\ref{theo:familles-ksp}). 
For the rest of this section, we assume that Conjecture~\ref{conj:familles-uni-cm} holds, 
and we keep the notation $\Phi : (\ZC_\ksp^{\CM^\times})^\t \longiso \Fam_\uni(W\t)$. 

\bigskip

\subsection{Fixed points and ${\boldsymbol{d}}$-cuspidality}
We expect that $d$-cuspidality of unipotent representations 
and $\t_d$-cuspidality of points in $\ZC_\ksp^{\t_d}=\ZC_\ksp^{\z_d \t}$ 
are linked as follows:

\bigskip

\begin{quotation}
\begin{conj}\label{conj:d-cusp}
Assume here that Conjecture~\ref{conj:familles-uni-cm} holds. 
Let $p \in (\ZC_\ksp^{\CM^\times})^\t$ be such that there 
exists $\l \in \Phi(p) \subset \unip(W\t)$ which is $d$-cuspidal. Then $p$ is $\t_d$-cuspidal.
\end{conj}
\end{quotation}

\bigskip

Note that the converse to Conjecture~\ref{conj:d-cusp} does not hold 
in general, even for $d=1$ (for instance for $W$ of type $D$, 
as it will be explained in~\S\ref{sub:D}). 

\bigskip

\subsection{${\boldsymbol{d}}$-Harish-Chandra theory and symplectic leaves}
Assume in this subsection that Conjectures~\ref{conj:familles-uni-cm} and~\ref{conj:d-cusp} 
hold for $W$ and all its parabolic subgroups. 
Fix $(P,\l) \in \Cus^d(W\t)$ and let $p$ be the point of $\ZC_\ksp(V/V^P,P)$ corresponding 
to the Lusztig family of $\l$ through Conjecture~\ref{conj:familles-uni-cm}. 
Then $p$ is $\t_d$-cuspidal by Conjecture~\ref{conj:d-cusp}. Therefore, 
one can associate to the pair $(P,p)$ a symplectic leaf $\leaf_{P,p}$ of $\ZC_\ksp^{\t_d}$. 

\bigskip

\begin{quotation}
\begin{conj}\label{conj:d-hc}
Recall that we assume that Conjectures~\ref{conj:familles-uni-cm} and~\ref{conj:d-cusp} 
hold for $W$ and all its parabolic subgroups. Then, with the above notation:
\begin{itemize}
\itemth{a} Let $p' \in (\ZC_\ksp^{\CM^\times})^\t$. Then $p' \in \overline{\leaf}_{P,p}$ 
if and only if the $d$-Harish-Chandra series $\harish^{W,P,\l}(\Irr(\Nrmov_{W_{\t_d}}(P_{\t_d},\l)))$ 
meets the family $\Phi(p')$.

\itemth{b} There exists a parameter $k_{P,p} \in \aleph((V^P)^{\t_d},\Nrmov_{W_{\t_d}}(P_{\t_d}))$ 
such that:
\begin{itemize}
\itemth{b1} $\overline{\leaf}_{P,p}^\nor \simeq 
\ZC_{k_{P,p}}((V^P)^{\t_d},\Nrmov_{W_{\t_d}}(P_{\t_d}))$ 
as Poisson varieties endowed with a $\CM^\times$-action.

\itemth{b2} The parameter $k_{P,\l} \in \aleph((V^P)^{\t_d},\Nrmov_{W_{\t_d}}(P_{\t_d},\l))$ 
involved in Conjecture~\ref{conj:BMM}(b) (and which is generic by the previous section) 
is the restriction of $k_{P,p}$ to $\Nrmov_{W_{\t_d}}(P_{\t_d},\l)$. 
\end{itemize}
\end{itemize}
\end{conj}
\end{quotation}

\bigskip

\begin{commentary}\label{com:parameter}
In the previous conjecture, the existence of a parameter $k_{P,p}$ 
satisfying~(b1) is just a restatement of Conjecture~\ref{conj:leaves}: 
the main point of the above conjecture is that its restriction should coincide 
with the parameter of Conjecture~\ref{conj:BMM}(b), which has 
to do with a completely different context ($\ell$-adic cohomology of 
Deligne-Lusztig varieties). In some sense, this is a justification 
of this long paper.\finl
\end{commentary}

\bigskip

The correspondence outlined in Conjecture~\ref{conj:d-hc} should also be compatible 
in a more precise way with Harish-Chandra theory. For this survey, keep the 
notation of the above conjecture and assume moreover that 
$\Nrmov_{W_{\t_d}}(P_{\t_d},\l)=\Nrmov_{W_{\t_d}}(P_{\t_d})$ 
(recall from Remark~\ref{rem:grc} that this is the most probable situation). 
Assume also that Conjecture~\ref{conj:d-hc} holds. Then, by~(b1), we get a $\CM^\times$-equivariant 
morphism of varieties
$$\psi : \ZC_{k_{P,p}}((V^P)^{\t_d},\Nrmov_{W_{\t_d}}(P_{\t_d})) \longto \ZC_{\ksp}$$
whose image is the closure $\overline{\leaf}_{P,p}$ of the symplectic leaf 
$\leaf_{P,p}$ of $\ZC_\ksp^{\t_d}$. By restriction to the $\CM^\times$-fixed points, 
we get a map
$$\psi_\fix : \ZC_{k_{P,p}}((V^P)^{\t_d},\Nrmov_{W_{\t_d}}(P_{\t_d}))^{\CM^\times} \longto 
\ZC_{\ksp}^{\CM^\times}$$
whose image is contained in $(\ZC_{\ksp}^{\CM^\times})^\t$. On the other hand, 
Conjecture~\ref{conj:familles-uni-cm} provides a surjective map 
$\Phi^* : \unip(W\t) \longsurto (\ZC_{\ksp}^{\CM^\times})^\t$ (whose fibers are 
the unipotent Lusztig families) and the definition of Calogero-Moser 
families provides a surjective map 
$$\zG_{k_{P,p}} : \Irr(\Nrmov_{W_{\t_d}}(P_{\t_d})) \longsurto 
\ZC_{k_{P,p}}((V^P)^{\t_d},\Nrmov_{W_{\t_d}}(P_{\t_d})^{\CM^\times}.$$ 
Finally, recall that $d$-Harish-Chandra theory (see Theorem~\ref{theo:d-hc} and its generic 
version) provides an injective map 
$$\harish^{W,P,\l} : \Irr(\Nrmov_{W_{\t_d}}(P_{\t_d})) \longinjto \unip(W\t).$$
We expect all these maps to be compatible in the following sense:

\bigskip

\begin{quotation}
\begin{conj}\label{conj:d-hc-commute}
Assume that Conjectures~\ref{conj:familles-uni-cm},~\ref{conj:d-cusp} and~\ref{conj:d-hc} 
hold and that $\Nrmov_{W_{\t_d}}(P_{\t_d},\l)=\Nrmov_{W_{\t_d}}(P_{\t_d})$. Then the 
diagram
$$
\diagram
\Irr(\Nrmov_{W_{\t_d}}(P_{\t_d})) \ar@{->>}[rr]^{\DS{\zG_{k_{P,p}}\hphantom{AAAA}}} 
\ar@{^{(}->}[dd]_{\DS{\harish^{W,P,\l}}} && 
\ZC_{k_{P,p}}((V^P)^{\t_d},\Nrmov_{W_{\t_d}}(P_{\t_d}))^{\CM^\times} \ddto^{\DS{\psi_\fix}}  \\
&& \\
\unip(W\t) \ar@{->>}[rr]^{\DS{\Phi^*}} && (\ZC_{\ksp}^{\CM^\times})^\t
\enddiagram$$
is commutative.
\end{conj}
\end{quotation}

\bigskip

The conjectures stated in this section, together with 
Conjecture~\ref{conj:coho} on the cohomology of Calogero-Moser spaces, 
imply the Conjecture~\ref{conj:d-hc-filtration}. Let us give some details. 
First, as in~\S\ref{sub:coho-filtration}, the morphism $\psi$ 
induces a morphism of algebras
$$\psi_\fix^\# : \im \Omeb^{\ksp}_W \longto 
\im \Omeb^{k_{P,p}}_{\Nrmov_{W_{\t_d}}(P_{\t_d})}.$$
Then, if we assume that $\t=\Id_V$ and that 
Conjectures~\ref{conj:gordon-martino} and~\ref{conj:familles-uni-cm} hold, 
the map $\psi_\fix^\#$ is just the map $(\harish_d^{W,P,\l})^\#$ 
of Conjecture~\ref{conj:d-hc-filtration}. So Proposition~\ref{prop:coho-graduation} 
has the following consequence:

\bigskip

\begin{prop}\label{prop:fil}
With the above notation, assume that 
$\Nrmov_{W_{\t_d}}(P_{\t_d},\l)=\Nrmov_{W_{\t_d}}(P_{\t_d})$ and 
that Conjectures~\ref{conj:coho},~\ref{conj:gordon-martino}~\ref{conj:familles-uni-cm},~\ref{conj:d-cusp},~\ref{conj:d-hc} and~\ref{conj:d-hc-commute} hold. 
Then Conjecture~\ref{conj:d-hc-filtration} holds for the $d$-cuspidal 
triple $(W,P,\l)$. 
\end{prop}

\bigskip

\begin{rema}\label{rem:filtration-bmm}
The consequence of Proposition~\ref{prop:fil} does not involve anymore the geometry 
of Calogero-Moser spaces but only the representation theory of finite reductive groups. 
Therefore, the validity of Conjecture~\ref{conj:d-hc-filtration} in many cases 
(see Remark~\ref{rem:d-hc-filtration} and Example~\ref{ex:filtration bmm}) 
is a good indication that the general philosophy of this paper has some 
reasonable foundation.\finl
\end{rema}

\bigskip

\begin{exemple}[Principal series]\label{ex:principal conjectures}
Assume in this example, and only in this example, that $\t=\Id_V$, 
that Conjecture~\ref{conj:familles-uni-cm} (i.e. Gordon-Martino's Conjecture~\ref{conj:gordon-martino}) 
holds and that $(P,\l) =(1,1) \in \Cus(W)$. 
Then Conjecture~\ref{conj:d-hc}(a) is a restatement of the fact 
that every family meets the principal series while Conjecture~\ref{conj:d-hc}(b) 
and Conjecture~\ref{conj:d-hc-commute} are vacuous.\finl
\end{exemple}

\bigskip

\begin{exemple}[Regular element]\label{ex:regular}
Assume in this example, and only in this example, that $\t=\Id_V$ and $d$ is chosen such that 
$\t_d$ is regular (see~\S\ref{sub:regular} for the definition: from this definition, the trivial 
subgroup of $W$ is $\t$-split). Then, $(1,1) \in \Cus^d(W)$ and $W_{\t_d}=C_W(w_d)$. 

On the unipotent representation side, if $\CG$ is a Lusztig $\ksp$-family, then 
it has been checked by J. Michel (unpublished) that the unipotent Lusztig family $\FG_\CG^\uni$ 
(see Theorem~\ref{theo:familles-ksp}) meets the $d$-Harish-Chandra series 
$\unip_d(W,1,1)$ if and only if $\sum_{\chi \in \CG} |\chi(w_d)|^2 \neq 0$. 
Moreover, he also checked that 
$$\sum_{\chi \in \CG} |\chi(w_d)|^2=\sum_{\substack{\psi \in \Irr C_W(w_d) \\ \text{such that~}
\harish_d^{W,1,1}(\psi) \in \FG_\CG^\uni}} \psi(1)^2.$$

On the Calogero-Moser space side, the closure of the symplectic leaf associated with $(1,1)$ 
is just the irreducible component of maximal dimension $(\ZC_\ksp^\t)_\maxi$ defined 
in~\S\ref{sub:regular}. So the above facts about unipotent representations justify, through the 
philosophy of this section, Conjecture~\ref{conj:chi-tau} and~\cite[Conj~5.2]{regular}.\finl
\end{exemple}

\bigskip

\part{Examples}

\def\Part{{\mathrm{Part}}}
\def\Core{{\mathrm{Cor}}}
\def\cor{{\mathrm{cor}}}
\def\quo{{\mathrm{quo}}}
\def\hook{{\mathrm{hk}}}
\def\pard{{\mathrm{par}}_d}

\boitegrise{{\bf Hypothesis.} {\it As in the third part, 
we assume that there exists a rational structure $V_\QM$ 
on $V$ which is stable under the action of $W$ (i.e., $W$ 
is a Weyl group) and that $V^W=0$. 
We also fix an element $\t \in \Nrm_{\Gb\Lb_\QM(V_\QM)}(W)$ 
of finite order, an integer $d \ge 1$ and a primitive 
$d$-th root of unity $\z_d$.}}{0.75\textwidth}

\bigskip

We aim to illustrate the Conjectures stated in Section~\ref{sec:main} 
by several examples:
\begin{itemize}
\itemth{a} We prove that Conjectures~\ref{conj:familles-uni-cm},~\ref{conj:d-cusp},~\ref{conj:d-hc} 
and~\ref{conj:d-hc-commute} hold in rank $2$ for $d$ equal to the Coxeter number.

\itemth{b} We also prove that, assuming Brou\'e-Malle-Michel Conjecture~\ref{conj:BMM} 
(and particularly the conjectural value of $k_{P,\l}$), 
they hold in type A. 

\itemth{c} For classical types, we only prove 
Conjectures~\ref{conj:familles-uni-cm}, as well as Conjecture~\ref{conj:d-cusp} whenever $d=1$ 
(classical Harish-Chandra theory).
\end{itemize}
As explained in Commentary~\ref{com:parameter}, the most intriguing question 
is Conjecture~\ref{conj:d-hc}(b), which predicts the equality of parameters coming from two 
extremely different contexts (cohomology of some Deligne-Lusztig 
variety vs symplectic leaves of Calogero-Moser spaces). Even for classical Harish-Chandra 
theory (i.e. whenever the Deligne-Lusztig variety is zero-dimensional), 
this is somewhat unexpected and certainly reflects some deep connections. 
In the examples treated in this part, we will mainly focus on this question.

\bigskip

\section{Rank 2}\label{sec:2}

\medskip

The case of type $A$ being treated in the upcoming Section~\ref{sec:type a}, 
we will just consider here the types $B_2$ and $G_2$. We will 
not fill the details for proving all the Conjectures: 
indeed, the groups are small enough so that the remaining details can be filled 
by the reader. So, as explained in the introduction to this part, we 
only give the details for Conjecture~\ref{conj:d-hc}(b).

\bigskip

%

\begin{theo}\label{theo:rang 2}
Assume that $W$ is of type $B_2$ or $G_2$ and that $d$ is the Coxeter number. 
Then Conjectures~\ref{conj:familles-uni-cm},~\ref{conj:d-cusp},~\ref{conj:d-hc} 
and~\ref{conj:d-hc-commute} hold.
\end{theo}

\bigskip

\def\cox{{\mathrm{cox}}}

\begin{proof}
Let $s$ and $t$ be the two simple reflections of $W$. 
Let $c=st$ be a standard Coxeter element of $W$ and let $\OCB_c$ denote its 
corresponding $\Gb$-orbit in $\BCB \times \BCB$. Then $\z_d c$ is $W$-full 
(so we may take $\t_d=\z_dc$) 
and $W_{\t_d}=C_W(\z_d c)=\langle c \rangle$ is the cyclic group 
of order $d$. As there is only one reflecting hyperplane for $W_{\t_d}$, 
the parameters for $W_{\t_d}$ will be denoted by $k=(k_0,k_1,\dots,k_{d-1})$. 
We denote by $k^\cox$ the parameter given by:
$$
k^\cox=
\begin{cases}
(0,1,2,1) & \text{if $d=4$,}\\
(0,1,2,1,1,1) & \text{if $d=6$.}
\end{cases}
$$
Also, there is (up to $\Gb^F$-conjugacy), only one proper $d$-split Levi subgroup, 
namely the Coxeter torus $\Tb_c$. Computing the Deligne-Lusztig induction 
of the trivial character of $\Tb_c^F$ amounts to determining the 
cohomology of the Deligne-Lusztig variety $\Xb_{\OCB_c}$. 
This has been done by Lusztig~\cite{lusztig coxeter}, and it follows from his 
work that Conjecture~\ref{conj:BMM} holds in this case.

Let us give more details. First, he proved Conjecture~\ref{conj:BMM}(a) about the 
disjointness of the cohomology groups~\cite[Theo.~6.1]{lusztig coxeter} and that the 
endomorphism algebra of the $\Gb^F$-module $\Hrm_c^\bullet(\XCB_{\OCB_c})$ 
is generated by the Frobenius endomorphism $F$ and he computed the eigenvalues 
of $F$ in all cases~\cite[Table~7.3]{lusztig coxeter}. This leads to the following 
presentation for this endomorphism algebra:
$$\begin{cases}
\text{Generator: $F$ (the Frobenius endomorphism),} \\
\text{Relation: $\prod_{j=0}^{d-1} (F-\z_d^j(\z_d^{-1}q)^{k_j^\cox})=0$.}
\end{cases}$$
In other words,
\equat\label{eq:b2-g2-cox-hc}
\End_{\Gb^F} \Hrm_c^\bullet(\XCB_{\OCB_c}) \simeq \HC_{k^\cox}(W_{\t_d},\z_d^{-1}q).
\endequat

On the other hand, the computation of the fixed point subvariety 
$\ZC_\ksp^{\mub_d}$ has been done in~\cite[Theo.~7.1]{bonnafe diedral} and the result 
is given by:
$$\ZC_\ksp^{\mub_d}\simeq \{(x,y,z) \in \CM^3~|~(z^2-d^2)z^{d-2}=xy\}.$$
Setting $z'=z+d$, we get
$$\ZC_\ksp^{\mub_d}\simeq \{(x,y,z') \in \CM^3~|~z'(z'-2d)(z'-d)^{d-2}=xy\}.$$
In other words, 
\equat\label{eq:b2-g2-cox-cm}
\ZC_\ksp^{\mub_d} \simeq \ZC_{k^\cox}(V^{\t_d},W_{\t_d})
\endequat
(see Example~\ref{ex:cyclique}(a)).

We see that the same parameter occurs in~\eqref{eq:b2-g2-cox-hc} 
and~\eqref{eq:b2-g2-cox-cm}: this shows that
Conjecture~\ref{conj:d-hc}(b) holds in this case, as desired.
\end{proof}

\bigskip

\bigskip

\section{Some combinatorics}

\medskip

We refer to~\cite[\S{2.7}]{james} for facts about abaci, $d$-cores, $d$-quotients of 
partitions that will be used here.

\bigskip

\subsection{Notation}
A {\it partition} is a sequence $\l=(\l_k)_{k \ge 1}$ of non-negative integers 
such that $\l_k \ge \l_{k+1}$ for all $k$ 
and $\l_k =0 $ for $k \gg 0$. Let $\Part$ denote the set of all partitions. 
If $\l \in \Part$, we set $|\l|=\sum_{k \ge 1} \l_k$ and $a_\l=\sum_{k \ge 1} (k-1)\l_k$, 
and we denote by $Y(\l)$ the {\it Young diagram} 
of $\l$, that is, the set of pairs of natural numbers 
$(i,j)$ such that $j \ge 1$ and $1 \le i \le \l_j$. If $y \in Y(\l)$, we denote 
by $\hook_\l(y)$ the {\it hook length} of $\l$ based at $y$, i.e. the number of $(i',j') \in Y(\l)$ 
such that $i' \ge i$, $j' \ge j$ and $(i-i')(j-j')=0$. Let
$$\degb\,\l = \qb^{a_\l} 
\frac{\DS{\prod_{k=1}^{|\l|} (\qb^k-1)}}{\DS{\prod_{y \in Y(\l)} (\qb^{\hook_\l(y)}-1)}}.$$
It turns out that $\degb\,\l \in \ZM[\qb]$.

Let $d \ge 1$. A partition $\l$ is called a {\it $d$-core} if $\hook_\l(y) \neq d$ for 
all $y \in Y(\l)$. The subset of $\Part$ consisting of $d$-cores 
is denoted by $\Core_d$. 
An element $\l=(\l^{(1)},\dots,\l^{(d)})$ of the set $\Part^d$ of $d$-uples of partitions 
is called a {\it $d$-partition}: we set $|\l|=|\l^{(1)}|+\cdots+|\l^{(d)}|$. If $\l \in \Part$, we denote 
by $\cor_d(\l) \in \Core_d$ its {\it $d$-core} and by $\quo_d(\l) \in \Part^d$ 
its $d$-quotient. The map
\equat\label{eq:d-coeur}
\fonction{\cor_d \times \quo_d}{\Part}{\Core_d \times \Part^d}{\l}{(\cor_d(\l),\quo_d(\l))}
\endequat
is bijective. Its inverse will be denoted by 
$$\pard : \Core_d \times \Part^d \longiso \Part.$$
It follows from the definition of both maps that
\equat\label{eq:taille-core}
|\l|=|\cor_d(\l)|+d\,|\quo_d(\l)|.
\endequat
If $r \ge 0$, we denote by $\Part(r)$ (resp. $\Part^d(r)$, resp. $\Core_d(r)$) the set 
of $\l \in \Part$ (resp. $\l \in \Part^d$, resp. $\l \in \Core_d$) such that 
$|\l|=r$. We also set $\Core_d(\equiv r)$ for the set of $\l \in \Core_d$ 
such that $|\l| \le r$ and $|\l| \equiv r \mod d$. In other words, 
$\Core_d(\equiv r)=\cor_d(\Part(r))$. 

\medskip

%
If $\l \in \Part^d(r)$, we denote by $\chi_\l$ the associated 
irreducible character of the complex reflection group 
$G(d,1,r)$, following the convention in~\cite{geck jacon}. 
If $k \in \CM^{\aleph(G(d,1,r))}$ and $\l \in \Part^d(r)$, we denote 
by $z_\l^k$ the element of $\zG_k(\chi_\l) \in \ZC_k(d,1,r)^{\CM^\times}$ defined 
in~\S\ref{sub:cm-families}. Note that we do not need to emphasize $d$ or $r$, 
as they are determined by $\l$.

\bigskip

\subsection{Abaci}\label{sub:abaque}
A {\it $d$-abacus} is an abacus with $d$ runners. 
If $\g \in \Core_d$, we denote by $\AG(\g)$ its $d$-abacus, 
with the convention that the first runner contains the first empty 
box. Let $b(\g)=(b_0(\g),b_1(\g),\dots,b_{d-1}(\g))$ denote the 
sequence defined as follows: $b_j(\g)$ is the number of beads 
on the $(j+1)$-th runner of $\AG(\g)$ minus the number of beads 
on the first runner. Let $\Res_d(\g)=(\r_0(\g),\dots,\r_{d-1}(\g))$ 
denote the $d$-residue of $\g$. It is defined as follows: 
$\r_k(\g)$ is the number of pairs $(i,j) \in Y(\g)$ such that 
$i-j \equiv k \mod d$. 

\bigskip

\begin{exemple}\label{ex:abaque}
Let $\g = (5,2,1) \in \Part(8)$. Its Young diagram is 
$$Y(\g)\hphantom{AAAAA}\yng(5,2,1)\hphantom{AAAAAAAA}$$
It is easily seen that $\g$ is a $4$-core, and its $4$-abacus $\AG(\g)$ 
is given by

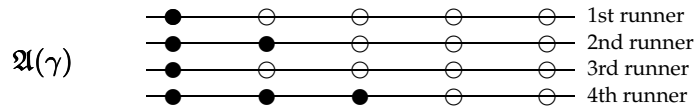
\centerline{\begin{picture}(160,50)
\put(0,40){\line(1,0){160}}
\put(0,30){\line(1,0){160}}
\put(0,20){\line(1,0){160}}
\put(0,10){\line(1,0){160}}
\put(10,40){\circle*{6}}
\put(10,40){\circle*{6}}\put(45,40){\circle{6}}\put(80,40){\circle{6}}\put(115,40){\circle{6}}\put(150,40){\circle{6}}
\put(10,30){\circle*{6}}\put(45,30){\circle*{6}}\put(80,30){\circle{6}}\put(115,30){\circle{6}}\put(150,30){\circle{6}}
\put(10,20){\circle*{6}}\put(45,20){\circle{6}}\put(80,20){\circle{6}}\put(115,20){\circle{6}}\put(150,20){\circle{6}}
\put(10,10){\circle*{6}}\put(45,10){\circle*{6}}\put(80,10){\circle*{6}}\put(115,10){\circle{6}}\put(150,10){\circle{6}}
\put(-50,20){$\AG(\g)$}\put(165,38){$\SS{\text{1st runner}}$}
\put(-50,20){$\AG(\g)$}\put(165,28){$\SS{\text{2nd runner}}$}
\put(-50,20){$\AG(\g)$}\put(165,18){$\SS{\text{3rd runner}}$}
\put(-50,20){$\AG(\g)$}\put(165, 8){$\SS{\text{4th runner}}$}
\end{picture}}
\noindent Then $b(\g)=(0,1,0,2)$ while $\Res_d(\g)=(3,2,2,1)$.\finl
\end{exemple}

\bigskip

Let us define two sequences $k^\g=(k^\g_j)_{0 \le j \le d-1}$ and 
$l^\g=(l^\g_j)_{0 \le j \le d-1}$ associated with a $d$-core $\g$:
\begin{itemize}
\item[$\bullet$] $k^\g_j=d b_j(\g)+j$.

\item[$\bullet$] $l_j^\g= b_0(\g)+b_1(\g)+\cdots+b_{d-1}(\g) + 
\begin{cases}
d(\r_{1-j}(\g)-\r_{-j}(\g))+j-1 & \text{if $1 \le j \le d-1$,}\\
d(\r_1(\g)-\r_0(\g))+d-1 & \text{if $j=0$.}\\
\end{cases}$
\end{itemize}
Here, the index in $\r_{1-j}(\g)$ or $\r_{-j}(\g)$ must be understood modulo $d$. 
The next result will be useful in the next section:

\bigskip

\begin{prop}\label{prop:k=l}
Let $\g$ be a $d$-core and let $m$ denote its length (i.e., the number 
of non-zero parts). Then
$$k^\g_j=l_{j+1-m}^\g$$
for all $j \in \ZM/d\ZM$.
\end{prop}

\bigskip

\begin{proof}
If $j \in \ZM/d\ZM$, we denote by $\jba$ its unique representative in $\{0,1,\dots,d-1\}$. 
Note that $m=b_0(\g)+b_1(\g) + \cdots + b_{d-1}(\g)$. 
We argue by induction on the length $m$ of $\g$. 

If $\g=\vide$, then 
$k^\g=(0,1,\dots,d-1)$ and $l^\g=(d-1,0,1,\dots,d-2)$. 
This shows the result whenever $m=0$. 

Assume now that $m \ge 1$ and write $\g=(\g_1,\g_2,\g_3,\cdots)$ with $\g_1 >\! 0$. 
Let $\g'=(\g_2,\g_3,\cdots)$. 
Then $\g'$ is a $d$-core so, by the induction hypothesis, 
Proposition~\ref{prop:k=l} holds for $\g'$. In other words, 
$k^{\g'}_j=l_{j+2-m}^{\g'}$ for all $j \in \ZM/d\ZM$.

Let us first compare the sequences $k^\g$ and $k^{\g'}$. For this, let $y$ denote the 
unique element of $\{0,1,\dots,d-1\}$ such that $b_y(\g)=b_y(\g')+1$. Then, if $j \neq y$, 
we get $b_j(\g)=b_j(\g')$. Therefore,
$$k_j^\g=
\begin{cases}
k_j^{\g'} & \text{if $j \neq y$,}\\
k_j^{\g'}+d & \text{if $j=y$.} 
\end{cases}\leqno{(\#)}$$

Let us now compare the sequences $l^\g$ and $l^{\g'}$. For this, let $x$ denote 
the unique element of $\{0,1,\dots,d-1\}$ such that $\g_1 \equiv x \mod d$. 
Then, for $j \in \ZM/d\ZM$, we get 
$$\r_j(\g)=\r_{j-1}(\g')+(\g_1-x)/d + \d_x(j)$$
where $\d_x$ is identically $0$ if $x=0$ and
$$\d_x(j)=
\begin{cases}
1 & \text{if $\jba=0$ or $1 \le d-\jba \le x-1$,}\\
0  & \text{if $x \le d-\jba \le d-1$,}\\
\end{cases}
$$
if $x \ge 1$. 
Also, note that $\g_1=d(b_y(\g)-1) + y+1-m$, so $y+1-m \equiv x \mod d$. 
Let us also write $l_j^\g$, for $j \in \ZM/d\ZM$, as follows:
$$l_j^\g=m+d(\r_{1-j}(\g)-\r_{-j}(\g))+\jba-1+d \d_{j,0}$$
where $\d_{j,0}$ is the Kronecker symbol. Putting things together, 
one gets:
\eqna
l_j^\g &=& m+d(\r_{-j}(\g')-\r_{-1-j}(\g')+\d_x(1-j)-\d_x(-j)) +\jba-1+d \d_{j,0} \\
&=& 1+l_{j+1}^{\g'} - \overline{j+1}+1-d\d_{j+1,0} + d(\d_x(1-j)-\d_x(-j))+\jba-1+d \d_{j,0}.\\
\endeqna
But $\overline{j+1}=\jba + 1 - d\d_{j+1,0}$, so
$$l_j^\g = l_{j+1}^{\g'} + d(\d_x(1-j)-\d_x(-j))+d \d_{j,0}.$$
Two cases may occur:
\begin{itemize}
\item[$\bullet$] If $x=0$, then $\d_x$ is identically $0$ and so 
$d(\d_x(1-j)-\d_x(-j))+d \d_{j,0}=d\d_{j,0}=d\d_{j,-x}$. 

\item[$\bullet$] If $x > 0$, then there are only two values of $j \in \ZM/d\ZM$ for which $\d_x(1-j)-\d_x(-j)$ 
is non-zero, namely $j=0$ and $j=x$. If $j=0$, this returns $-1$ while, if $j=x$, this returns $1$. 
Therefore, we have again $d(\d_x(1-j)-\d_x(-j))+d \d_{j,0}=d\d_{j,x}$.
\end{itemize}
Finally, $l_j^\g = l_{j+1}^{\g'} + d \d_{j,-x}$ and $x \equiv y+1-m \mod d$ so, 
by the induction hypothesis and~$(\#)$, 
$$l_j^\g = k_{j+m-1}^{\g'}+d \d_{j,m-y} = k_{j+m-1}^\g - d \d_{j+m-1,y}+d \d_{j,x} = k_{j+m-1}^\g,$$
as expected.
\end{proof}

\bigskip

\section{The smooth example: type A}\label{sec:type a}

\medskip

The Calogero-Moser space $\ZC_\ksp$ is smooth if and only if 
$W$ is a Weyl group of type $A$. This simplifies drastically its 
geometry ($\CM^\times$-fixed points, symplectic leaves,...) and all the  
conjectures proposed in Part~\ref{part:cm} are true in this 
case~\cite{bonnafe shan, bonnafe maksimau}. 

On the other hand, the almost characters of some $\GC \in \groups$ 
are all irreducible characters if and only if $\GC$ is of type $A$. 
This also simplifies drastically its representation theory 
(unipotent Lusztig families, $d$-Harish-Chandra theory, blocks,...). 

In the spirit of this paper, these two facts should be the shadow of a common phenomenon. 
We do not propose an explanation for it, but we give details about how the combinatorics 
on both sides fit perfectly in type $A$, so that all the conjectures stated in Part~\ref{part:generique} 
hold in this case (provided that Conjecture~\ref{conj:BMM} holds).  

\bigskip

\boitegrise{{\bf Hypothesis.} {\it From now on, and until the end of 
this section, we fix a natural number $n \ge 2$ and we assume that
$$V=\{(\xi_1,\dots,\xi_n) \in \CM^n~|~\xi_1+\cdots+\xi_n=0\}$$
and that $W=\SG_n$ acting on $V$ by permutation of the coordinates. 
The Calogero-Moser space $\ZC_\ksp(V,\SG_n)$ will be simply denoted 
by $\ZC(n)$.\\
\hphantom{AA} We set $r=\lfloor n/d\rfloor$ and we denote by $w_d$ a product of 
$r$ disjoint cycles of length $d$. Then $\z_d w_d$ is $\SG_n$-split.}}{0.8\textwidth}

\bigskip

The $\z_dw_d$-split parabolic subgroups of $\SG_n$ are those of the form 
$\SG_m$, where $m \le n$ and $m \equiv n \mod d$. In this case, 
$(\SG_n)_{\z_dw_d} \simeq G(d,1,r)$ and 
\equat\label{eq:n-sn}
\Nrmov_{(\SG_n)_{\z_d w_d}}((\SG_m)_{\z_d w_d}) \simeq G(d,1,(n-m)/d).
\endequat
If $\g \in \cor_d(\equiv n)$, we set $r_\g=r_\g(n)=(n-|\g|)/d$ and we denote by 
$k_\g \in \CM^{\aleph(G(d,1,r_\g))}$ the parameter defined by:
$$\begin{cases}
((k_\g)_{\Ker(x_1-x_2),0},(k_\g)_{\Ker(x_1-x_2),1})=(d,0),& \text{if $r_g \ge 2$,}\\
((k_\g)_{\Ker(x_1),0},(k_\g)_{\Ker(x_1),1},\dots,(k_\g)_{\Ker(x_1),d-1})=k^\g,& 
\text{if $r_g \ge 1$ and $d \ge 2$,}
\end{cases}$$
where $k^\g$ is the sequence defined in~\S\ref{sub:abaque}. If $r_\g=0$, then $k_\g$ 
is the zero parameter of the trivial group.

If $\l \in \Part(n)$, we denote by $z_\l$ the image of $\chi_\l$ through the map 
$\zG_\ksp : \Irr(\SG_n) \longto \ZC(n)^{\CM^\times}$. Similarly, if $\mu \in \Part^d(m)$, 
and if $k \in \CM^{\aleph(G(d,1,m))}$, we denote by $z_\mu^k$ the image 
of $\chi_\mu$ through the map $\zG_k : \Irr G(d,1,m) \longto \ZC_k(G(d,1,m))^{\CM^\times}$. 

\bigskip

\subsection{Geometry of ${\boldsymbol{\ZC(n)}}$} 
Recall that $\ZC(n)$ is smooth~\cite[Cor.~16.2]{EG}. This has several 
consequences:
\begin{itemize}
\item[$\bullet$] First, the map $\zG_\ksp : \Irr(\SG_n) \longto \ZC(n)^{\CM^\times}$ defined 
in~\S\ref{sub:cm-families} is bijective~\cite[Cor.~5.8]{gordon}. This means that 
\equat\label{eq:zlambda-sn}
\text{\it the map $\Part(n) \longto \ZC(n)^{\CM^\times}$, $\l \longmapsto z_\l$ is bijective.}
\endequat

\item[$\bullet$] The variety $\ZC(n)$ has only one symplectic leaf (it is a symplectic variety). 
Therefore, if $d \ge 1$, then $\ZC(n)^{\mub_d}$ is also smooth and symplectic, so 
its symplectic leaves coincide with its irreducible components: in particular, 
they are closed normal subvarieties of $\ZC(n)$, so coincide 
with the normalization of their closure (which is involved in 
Conjecture~\ref{conj:leaves}).
\end{itemize}

The next theorem, which describes these symplectic leaves, has been obtained 
by Maksimau and the author~\cite[Theo.~4.21]{bonnafe maksimau}:

\bigskip

\begin{theo}\label{theo:fixes-an}
With the above notation, we have:
\begin{itemize}
\itemth{a} Let $\l \in \Part(n)$. Then $z_\l$ is $\z_d$-cuspidal if and only if $\l$ is a $d$-core.

\itemth{b} The map 
$$\fonctio{\Core_d(\equiv n)}{\Cus_\ksp^{\z_d w_d}(\ZC(n))}{\g}{(\SG_{|\g|},z_\g)}$$
is bijective. We denote by $\leaf_\g(n)$ the symplectic leaf $\leaf_{\SG_{|\g|},z_\g}$ 
of $\ZC(n)^{\mub_d}$.

\itemth{c} If $\g \in \Core_d(\equiv n)$, then there exists a $\CM^\times$-equivariant 
isomorphism of varieties 
$$i_\g : \ZC_{k_\g}(G(d,1,r_\g)) \longiso \leaf_\g(n)$$
such that 
$$i_\g(z_\mu)=z_{\pard(\g,\mu)}$$
for all $\mu \in \Part^d(r_\g)$ (in particular, $\dim \leaf_\g(n)=2r_\g$).
\end{itemize}
\end{theo}

\begin{proof}
All the results have been proved in~\cite[Theo.~4.21]{bonnafe maksimau}, except that we need 
to make some comments about the parameter. So let $\g \in \Core_d(\equiv n)$ and 
let $l_\g \in \CM^{\aleph(G(d,1,r_\g))}$ be the parameter defined by:
$$\begin{cases}
((l_\g)_{\Ker(x_1-x_2),0},(l_\g)_{\Ker(x_1-x_2),1})=(d,0),& \text{if $r_g \ge 2$,}\\
((l_\g)_{\Ker(x_1),0},(l_\g)_{\Ker(x_1),1},\dots,(l_\g)_{\Ker(x_1),d-1})=l^\g,& 
\text{if $r_g \ge 1$ and $d \ge 2$,}
\end{cases}$$
where $l^\g$ is the sequence defined in~\S\ref{sub:abaque}. Then the result 
of~\cite[Theo.~4.21]{bonnafe maksimau} says that $\leaf_{\SG_{|\g|},z_\g} \simeq \ZC_{l_\g}(G(d,1,r_\g))$. 
However, Proposition~\ref{prop:k=l} says that $l^\g$ is obtained from $k^\g$ 
by a cyclic permutation, and so $\ZC_{k_\g}(G(d,1,r_\g)) \simeq \ZC_{l_\g}(G(d,1,r_\g))$ 
by~\cite[$($3.5.4$)$]{calogero}.
\end{proof}

%
%

\bigskip

\subsection{Unipotent representations: the split case}
We fix in this subsection a triple $\GC=(q,\Gb,F) \in \groups(\SG_n)$. 
In other words, $\Gb^{\prime F}$ is a split group of type $A_{n-1}$ 
(recall that the definition of $\groups(\SG_n)$ implies no 
restriction on the rational structure of the center of $\Gb$).
Then it is well-known that 
\equat\label{eq:gln}
\unipc(\GC)=\{R_{\chi}^\GC~|~\chi \in \Irr(\SG_n)\}
\qquad
\text{and}\qquad \deg R_{\chi_\l}^\GC=(\degb\,\l)(q).
\endequat
In other words, we may define the set $\unip(\SG_n)$, the bijection $\rhob^\GC$ 
and the map $\degb_{\SG_n}$ as follows:
$$\begin{cases}
\unip(\SG_n)=\Part(n),\\
\rhob^\GC_\l=R_{\chi_\l}^\GC \quad \text{for any $\l \in \Part(n)$,}\\
\degb_{\SG_n}=\degb.
\end{cases}$$
The partition into families is pretty easy in this case:
\equat\label{eq:famille-gln}
\text{\it All the unipotent Lusztig families are singletons.}
\endequat
A generic translation is given by:
\equat\label{eq:famille-gln-bis}
\text{\it the map $\Part(n) \longto \Fam_\uni(\SG_n)$, $\l \longmapsto \{\l\}$ 
is bijective.}
\endequat
The following result has been proved in~\cite[Pages~45-47]{BMM}:

\bigskip

\begin{theo}[Brou\'e-Malle-Michel]\label{theo:bmm-gln}
With the above notation, we have:
\begin{itemize}
\itemth{a} Let $\l \in \Part(n)$. Then the unipotent character $R_{\chi_\l}$ is $d$-cuspidal 
if and only if $\l$ is a $d$-core.

\itemth{b} 
The map 
$$\fonctio{\Core_d(\equiv n)}{\Cus_d(\SG_n)}{\g}{(\SG_{|\g|},R_{\chi_\g})}$$
is bijective. If $\g \in \Core_d(\equiv n)$, then 
$$\Nrmov_{(\SG_n)_{\z_dw_d}}((\SG_{|\g|})_{\z_dw_d}) \simeq G(d,1,r_\g).$$

\itemth{c} If $\g \in \Core_d(\equiv n)$, let $\harish_d^\g$ denote the bijection 
$\harish_d^{\SG_n,\SG_m,R_{\chi_\g}} : \Irr(G(d,1,r_\g)) \longto \unip(\SG_n,\SG_{|\g|},R_{\chi_\g})$ 
defined by the $d$-Harish-Chandra theory. Then
$$\harish_d^\g(\chi_\l)=R_{\chi_{\pard(\g,\l)}}$$
for all $\l \in \Part^d(r_\g)$. 
\end{itemize}
\end{theo}

\bigskip

Now, fix a $d$-core $\g \in \Core_d(\equiv n)$ and let $\LC_\g=(g,\Lb_\g,F) \in \groups(\SG_{|\g|})$ 
be such that $\Lb_\g$ is a $d$-split Levi subgroup of $\Gb$ (if $\Gb^F=\Gb\Lb_n(\gfq)$, then 
$\Lb_\g^F\simeq \Gb\Lb_{|\g|}(\gfq) \times (\FM_{\!q^d}^\times)^{r_\g}$). 
Conjecture~\ref{conj:BMM} 
predicts the existence of a parabolic subgroup $\Pb_\g$ such that the Deligne-Lusztig 
variety $\Yb_{\Ub_{\Pb_\g}}$ satisfies the following two properties:
\begin{itemize}
\itemth{a} The $\qlb\Gb^F$-modules $\Hrm_c^j(\Yb_{\Pb_\g}) \otimes_{\qlb\Lb_\g^F} R_{\chi_\g}^{\LC_\g}$ 
and $\Hrm_c^{j'}(\Yb_\Pb) \otimes_{\qlb\Lb^F} R_{\chi_\g}^{\LC_\g}$ have no common irreducible 
constituent if $j \neq j'$.

\itemth{b} $\End_{\Gb^F} \RC_{\Lb_\g \subset \Pb_\g}^\Gb(R_{\chi_\g}^{\LC_\g}) 
\simeq \HC_{k_\g^\#}(G(d,1,r_\g),\z_d^{-1}q)$ for some parameter $k_\g^\#$. 
\end{itemize}
This conjecture is far from being proved (see the next remark) but 
Brou\'e-Malle~\cite[Prop.~2.10]{broue malle} proposed a conjectural value for $k_\g^\#$:

\bigskip

\begin{quotation}
\begin{conj}[Brou\'e-Malle]\label{conj:bm}
If $\g \in \Core_d(\equiv n)$, then Conjecture~\ref{conj:BMM} holds 
for the pair $(\LC_\g,R_{\chi_\g}^{\LC_\g})$ with parameter $k_\g^\#=k_\g$.
\end{conj}
\end{quotation}

\bigskip

We find remarkable that the parameter $k_\g$ predicted, in this 
particular case, by Brou\'e-Malle in 1993 in the context of Deligne-Lusztig 
varieties coincides (up to a non-relevant cyclic permutation) with the 
parameter $l_\g$ found in 2018 by Maksimau and the author when studying 
Calogero-Moser spaces.

\bigskip

\begin{rema}\label{rem:bmm-gln}
The disjunction of the cohomology 
(see the above statement~(a)) has been proved 
in~\cite[Theo.~4.3]{bonnafe dudas rouquier} whenever $\Lb_\g$ is a torus (i.e. $|\g|=0$ or $1$), 
based on earlier works of Dudas~\cite[Cor.~3.2]{dudas}. 
The statement~(b) is only known if $d=n$ (Lusztig~\cite[\S{7.3}]{lusztig coxeter}) 
or $d=n-1$ (Digne-Michel~\cite[Theo.~10.1]{dm}).\finl
\end{rema}

\bigskip

%
%

The comparison between Theorems~\ref{theo:fixes-an} and~\ref{theo:bmm-gln} yields:

\bigskip

\begin{theo}\label{theo:A}
If $W\t=\SG_n$, then:
\begin{itemize}
\itemth{a} Conjectures~\ref{conj:familles-uni-cm},~\ref{conj:d-cusp},~\ref{conj:d-hc}(a) 
and~\ref{conj:d-hc-commute} hold.

\itemth{b} If moreover Conjecture~\ref{conj:bm} holds, then Conjecture~\ref{conj:d-cusp}(b) holds.
\end{itemize}
\end{theo}

\bigskip

\subsection{The non-split case} 
In type $A$, the non-split case corresponds to the case where $F$ acts on the Weyl group 
$\SG_n$ by the diagram automorphism (i.e. conjugation by the longest element). In our generic 
description, this corresponds to the coset $-\SG_n$ and objects in $\groups(-\SG_n)$ 
(for instance, we may take for $(q,\Gb,F) \in \groups(-\SG_n)$ the triple where 
$\Gb=\Gb\Lb_n(\overline{\FM}_{\! q})$ and $\Gb^F$ is the general unitary group). 
Ennola duality establishes a bijection between $\unip(\SG_n)$ and $\unip(-\SG_n)$ 
and this bijection transforms $d$-Harish-Chandra series into $d'$-Harish-Chandra 
series, where
$$d'=
\begin{cases}
2d & \text{if $d \equiv 1 \!\!\!\!\mod 2$,}\\
d/2 & \text{if $d \equiv 2 \!\!\!\!\mod 4$,}\\
d & \text{if $d \equiv 0 \!\!\!\!\mod 4$.}
\end{cases}$$
Therefore, the non-split case follows directly from the split one by applying Ennola duality 
(see~\cite[Rem.~2.11]{broue malle}). 
More precisely:

\bigskip

\begin{theo}\label{theo:2A}
If $W\t=-\SG_n$, then:
\begin{itemize}
\itemth{a} Conjectures~\ref{conj:familles-uni-cm},~\ref{conj:d-cusp},~\ref{conj:d-hc}(a) 
and~\ref{conj:d-hc-commute} hold.

\itemth{b} If moreover Conjecture~\ref{conj:bm} holds, then Conjecture~\ref{conj:d-cusp}(b) holds.
\end{itemize}
\end{theo}

\bigskip

\section{Classical groups and Harish-Chandra theory}\label{sec:classique}

\bigskip

We aim to check the following result:

\bigskip

\begin{theo}\label{theo:classique}
Assume that $W$ is a Weyl group of classical type and that $d=1$. Then 
Conjectures~\ref{conj:familles-uni-cm},~\ref{conj:d-cusp} and~\ref{conj:d-hc} hold.
\end{theo}

\bigskip

The rest of this section is devoted to the proof of this Theorem. 
Note that the most difficult part of the work has been previously 
done by Gordon-Martino~\cite{gordon martino}, Bellamy-Thiel~\cite{bellamy thiel} 
and Bellamy-Maksimau-Schedler~\cite{be-sc} on the Caloger-Moser space side 
(as well as an application of Bellamy-Maksimau-Schedler result by the author~\cite[Cor.~9.8]{auto}), and by 
Lusztig~\cite{lusztig classique, lusztig BC, lusztig D} 
on the unipotent representations side. 
The main interest of this section is to relate all these results together 
following the philosophy of this survey.

\vskip1cm
\def\SQ{\Db}
\def\SQB{\Bb\Cb}

\boitegrise{{\bf Notation.} {\it If $n \ge 0$, we set $W_n=G(2,1,n)$ and $W_n'=G(2,2,n)$, 
so that $W_n$ is a Weyl group of type $B_n$ (i.e. $C_n$) and $W_n'$ is a Weyl 
group of type $D_n$. We denote by $\t=\diag(-1,1,\dots,1) \in W_n$: it induces 
the non-trivial involutive diagram automorphism 
of $W_n'$. The canonical basis of $V=\CM^n$ is denoted by $(y_1,\dots,y_n)$ and its dual 
basis is denoted by $(x_1,\dots,x_n)$.}}{0.75\textwidth} 

\bigskip

Note that $W_n=\langle \t \rangle \ltimes W_n'$. 

\bigskip

\subsection{Families}
First, Conjectures~\ref{conj:familles-uni-cm} and~\ref{conj:d-cusp} for $d=1$ 
follow immediately from Theorem~\ref{theo:cuspidal-unique} and~\cite[Theo.~A]{bellamy thiel}. 

\bigskip

\subsection{Type B or C} 
We denote by $\SQB(n)$ the set of $r \in \ZM_{\geqslant 0}$ 
such that $r^2+r \le n$. If $r$, $m \ge 0$, we denote by $k[r]$ the element of 
$\CM^{\aleph(W_m)}$ defined by
$$\begin{cases}
k[r]_{\O,0}=r,\quad k[r]_{\O,1}=0,\\
k[r]_{\O',1}=1,\quad k[r]_{\O',1}=0,\\
\end{cases}$$
where $\O$ (resp. $\O'$) is the orbit of the reflecting hyperplane $\Ker(x_1)$ 
(resp. $\Ker(x_1-x_2)$). Note that $\ksp=k[1]$. The notation $k[r]$ is somewhat ambiguous, 
as it does not refer to the natural number $m$, but it will be used only whenever 
$m$ is clear from the context.
The generic version of Harish-Chandra theory can be summarized as 
follows~\cite{lusztig classique},~\cite[Tab.~II]{lusztig cbms} 
(note that $\Nrmov_{W_n}(W_m) \simeq W_{n-m}$ for all $m \le n$):

\bigskip

\begin{theo}[Lusztig]
Let $n \ge 2$. Then:
\begin{itemize}
\itemth{a} $\unip_\cus(W_n)$ is non-empty if and only if $n=r^2+r$ for some $r \in \ZM_{\geqslant 0}$. 
In this case, it contains only one element, which will be denoted by $\cus_n$. 

\itemth{b} The map
$$\fonctio{\SQB(n)}{\Cus(W_n)/W_n}{r}{(W_{r^2+r},\cus_{r^2+r})}$$
is bijective.

\itemth{c} Let $r \in \SQB(n)$. 
If $\GC=(q,\Gb,F)\in\groups(W_n)$ and if $\LC=(q,\Lb,F) \in \groups(W_{r^2+r})$ is such that 
$\Lb$ is a $1$-split Levi subgroup of $\Gb$, then
$$\End_{\Gb^F} \RC_\Lb^\Gb \rhob_{\cus_{r^2+r}}^\GC \simeq \HC_{k[2r+1]}(W_{n-(r^2+r)}).$$
\end{itemize}
\end{theo}

\bigskip

On the Calogero-Moser space side, Bellamy-Thiel~\cite[Theo.~6.24]{bellamy thiel} 
and Bellamy-Maksimau-Schedler~\cite{be-sc} proved the following result (note that the proof 
of~(c) by Bellamy-Maksimau-Schedler relies on the description of $\ZC_k(W_n)$ as quiver varieties):

\bigskip

\begin{theo}[Bellamy-Thiel, Bellamy-Maksimau-Schedler]
Let $n \ge 2$. Then:
\begin{itemize}
\itemth{a} $\ZC_\ksp(W_n)$ contains a cuspidal point if and only if $n=r^2+r$ for some 
$r \in \ZM_{\geqslant 0}$. 
In this case, it contains only one cuspidal point, which will be denoted by $z_n^\cus$. 
It is equal to $z^\ksp_{(r^{r+1},\vide)}$.

\itemth{b} The map
$$\fonctio{\SQB(n)}{\Cus_{\Id_V}(\ZC_\ksp(W_n))/W_n}{r}{(W_{r^2+r},z_{r^2+r}^\cus)}$$
is bijective. We denote by $\leaf_r(n)$ the symplectic leaf of $\ZC_\ksp(W_n)$ 
indexed by $(W_{r^2+r},z_{r^2+r}^\cus)$ through Theorem~\ref{theo:leaves}.

\itemth{c} Let $r \in \SQB(n)$. Then
$$\overline{\leaf_r(n)}^\nor \simeq \ZC_{k[2r+1]}(W_{n-(r^2+r)}).$$
\end{itemize}
\end{theo}

\bigskip

The comparison between the above two theorems proves Conjecture~\ref{conj:d-hc} 
in type $B$ or $C$, up to the verification that the cuspidal unipotent representation 
belongs the same unipotent Lusztig family as $\harish^W(\chi_{(r^{r+1},\vide)})$. 
We need for this the combinatorics of symbols~\cite[\S{3}]{lusztig classique} and 
its link with unipotent representations~\cite[Theo.~8.2]{lusztig classique}. 
Whenever $n=r(r+1)$, then the cuspidal unipotent representation $\cus_n$ is parametrized by the 
symbol $\begin{pmatrix} 1 & 2 & \cdots & 2r & 2r+1 \\  && \vide && \end{pmatrix}$ 
(with defect $2r+1$) 
while $\harish^W(\chi_{(r^{r+1},\vide)})$ is parametrized by the 
symbol $\begin{pmatrix} r+1 & r+2 & \cdots & 2r & 2r+1 \\  1 & 2 & \cdots & r & \end{pmatrix}$ 
(with defect $1$). Since both symbols have the same entries, $\cus_{r(r+1)}$ and 
$\harish^W(\chi_{(r^{r+1},\vide)})$ belong to the same unipotent Lusztig 
family~\cite[Theo.~5.8]{lusztig BC}, as desired.

\bigskip

\subsection{Type D}\label{sub:D}
We denote by $\SQ(n)$ the set of $r \in \ZM_{\geqslant 0}$ such that $r^2 \le n$ and $r \neq 1$. 
For $j \in \{0,1\}$, we set 
$$\SQ_j(n)=\{0\} \cup \{r \in \SQ(n)~|~r \ge 2~\text{and}~r \equiv j \!\!\!\mod 2\}.$$
In type $D$, there are two kinds of possible rational structures (and a third one, 
if $n=4$, inducing an order $3$ automorphism of $W_4'$: it will not be considered 
here), a split one 
and a non-split one. They correspond respectively to 
the elements $\Id_V=\t^0$ and $\t$ of the normalizer of $W_n'$. Note that
$$\Nrmov_{(W_n')_{\t^j}}((W_m')_{\t^j}) \simeq 
\begin{cases}
W_n' & \text{if $(m,j) = (0,0)$,}\\
W_{n-1} & \text{if $(m,j)=(0,1)$,}\\
W_{n-m} & \text{if $m \ge 2$.}\\
\end{cases}
$$
We summarize the generic version of Harish-Chandra theory in both 
cases~\cite{lusztig classique}, 
\cite[Tab.~II]{lusztig cbms}:

\bigskip

\begin{theo}[Lusztig]
We have:
\begin{itemize}
\itemth{a} $\unip_\cus(W_n')$ (resp. $\unip_\cus(W_n'\t_n)$) 
is non-empty if and only if $n=r^2$ for some $r \in \ZM_{\geqslant 0}$, $r$ even 
(resp. $r$ odd or $r=0$). 
In this case, it contains only one element, which will be denoted by $\cus_n'$. 

\itemth{b} If $j \in \{0,1\}$, the map
$$\fonctio{\SQ_j(n)}{\Cus(W_n'\t_n^j)/W_n'}{r}{(W_{r^2}',\cus_{r^2}')}$$
is bijective.

\itemth{c} Let $j \in \{0,1\}$ and let $r \in \SQ_j(n)$. 
If $\GC=(q,\Gb,F)\in\groups(W_n'\t_n^j)$ and if $\LC=(q,\Lb,F) \in \groups(W_{r^2}')$ is such that 
$\Lb$ is a $1$-split Levi subgroup of $\Gb$, then
$$\End_{\Gb^F} \RC_\Lb^\Gb \rhob_{\cus_{r^2}'}^\GC \simeq 
\begin{cases}
\HC_\ksp(W_n') & \text{if $(r,j) = (0,0)$,}\\
\HC_{k[2]}(W_{n-1}) & \text{if $(r,j)=(0,1)$,}\\
\HC_{k[2r]}(W_{n-r^2}) & \text{otherwise.}\\
\end{cases}$$
\end{itemize}
\end{theo}

\bigskip

In~\cite[Cor.~9.8]{auto}, the author determined the partition into symplectic leaves 
(as well as their structure) of both $\ZC_\ksp(W_n')$ and $\ZC_\ksp(W_n')^{\t_n}$, 
but it must be said that the essential part of the work was done by 
Bellamy-Thiel~\cite[Prop.~4.17]{bellamy thiel} and Bellamy-Maksimau-Schedler~\cite{be-sc}:

\bigskip

\begin{theo}[Bellamy-Thiel, Bellamy-Maksimau-Schedler]
Let $n \ge 4$. Then:
\begin{itemize}
\itemth{a} $\ZC_\ksp(W_n')$ (resp. $\ZC_\ksp(W_n')^{\t_n}$) admits a cuspidal 
point if and only if there exists $r \ge 2$ such that $n=r^2$. In this case, 
it contains only one element, which will be denoted by $y_n^\cus$. By extension, we 
set $y_0^\cus$ for the unique cuspidal point of $\ZC_\ksp(0,1)$. 

\itemth{b}If $j \in \{0,1\}$, the map
$$\fonctio{\SQ_j(n)}{\Cus^{\t_n^j}(\ZC_\ksp(W_n')/W_n'}{r}{(W_{r^2}',y_{r^2}^\cus)}$$
is bijective. We denote by $\leaf_r'(n)$ the cuspidal leaf of $\ZC_\ksp(W_n')^{\t_n^j}$ 
indexed by $(W_{r^2}',y_{r^2}^\cus)$.

\itemth{c} Let $j \in \{0,1\}$ and let $r \in \SQ_j(n)$. Then
$$\overline{\leaf_r'(n)}^\nor \simeq 
\begin{cases}
\ZC_\ksp(W_n') & \text{if $(r,j) = (0,0)$,}\\
\ZC_{k[2]}(W_{n-1}) & \text{if $(r,j)=(0,1)$,}\\
\ZC_{k[2r]}(W_{n-r^2}) & \text{otherwise.}\\
\end{cases}$$
\end{itemize}
\end{theo}

\bigskip

The comparison between the above two theorems proves Conjecture~\ref{conj:d-hc} 
in type $D$, in both the untwisted and the twisted case, up to the verification 
that the cuspidal families correspond on both side: this is done thanks to the 
combinatoric of symbols, in the same way as for type $B$ or $C$.

\part{Spetses}

\section{What is a spets?}

\medskip

As explained in Section~\ref{sec:genericity}, and as noticed in many papers 
on the subject, essential features of the representation theory of a finite 
reductive group are controlled by its Weyl group and can be recovered 
from structures built from it. The {\it Spetses}\footnote{Spetses is a greek island 
where a conference on finite groups was organized in 1993: this program started there, during a coffee break...} 
program initiated 
by Brou\'e-Malle-Michel~\cite{malle gd, malle icm, spetses 1, spetses 2}, which takes its origin 
in their work on genericity~\cite{BMM, broue malle}, 
proposes to attach to some finite complex reflection groups (called {\it spetsial}, see below) 
some numerical data (``unipotent representations'', 
``degrees'') which admits partitions into ``families'', ``$d$-Harish-Chandra series'' 
satisfying the same kind of properties as in the case of Weyl groups. 
This was soon corroborated by computations done by Lusztig~\cite{lusztig h4} and 
Malle (unpublished) for finite Coxeter groups that are not Weyl groups. 
This suggests that there should be a mysterious object (the {\it spets}) admitting 
some kind of representation theory similar to the representation theory 
of finite reductive groups.

We try to summarize it (very) quickly in this section, and see what are the possible links 
with the material of this paper. We come back to the general situation where $W$ is a complex reflection 
group. The {\it spetsial} parameter of $W$, denoted by $\ksp$, is defined 
by $(\ksp)_{\O,0}=1$ and $(\ksp)_{\O,j}=0$ for all $\O \in \AC/W$ and $1 \le j \le e_\O-1$. 
Brou\'e-Malle-Michel asked whether one can associate with any reflection 
group $W$ several combinatorial data:
\begin{itemize}
\itemth{S1} A set $\unip(W)$, whose elements are called 
{\it irreducible unipotent representations} of the {\it spets} attached 
to $W$, even though there is no group and no representation attached to them.

\itemth{S2} A map $\degb : \unip(W) \longto \CM[\qb]$.
For $\z$ a root of unity, an irreducible unipotent representation 
$\r \in \unip(W)$ is called {\it $\z$-cuspidal} if $\degb \r$ is divisible by 
$(\qb-\z)^{\dim(V/W)^\z}$.

\itemth{S3} A $\z$-Harish-Chandra theory: in other words, a partition of $\unip(W)$ 
into $\z$-Harish-Chandra series built on the same model as in Theorem~\ref{theo:d-hc}. 
In particular, to each Harish-Chandra series is associated a Hecke algebra of the 
stabilizer of the corresponding $\z$-cuspidal pair $(P,\l)$, with a well-defined 
parameter $k_{P,\l}$.

\itemth{S4} Almost characters: these are formal complex linear combinations 
of irreducible unipotent representations that can be used to define a partition of 
$\unip(W)$ into unipotent Lusztig families in the same way as in~\S\ref{sub:familles}.
\end{itemize}
All these data should satisfy compatibility conditions (axioms) which mimic what is 
known or conjectured for finite reductive groups. The group $W$ is said to be 
{\it spetsial} if some divisibility property of its Schur elements $\schur_\chi^\ksp$ 
holds~\cite[\S{3}]{malle icm}. It turns out that 
many complex reflection groups are not spetsial, but some of them are. 
The list of irreducible 
spetsial groups is as follows:
\begin{itemize}
 \item[$\bullet$] The groups $G(e,1,n)$ and $G(e,e,n)$ for any $e \ge 1$;
 
 \item[$\bullet$] The primitive groups $G_j$, for 
$$j \in \{4,6,8,14,23,24,25,26,27,28,29,30,32,33,34,35,36,37\}.$$
\end{itemize}
Being spetsial is easily seen to be a necessary condition for admitting combinatorial data 
as in~(S1),~(S2),~(S3) and~(S4) satisfying the list of axioms, but it is somewhat astonishing 
that it is also a sufficient condition~\cite{malle gd, spetses 1, spetses 2}. 
The natural remaining question is to figure out if there is a category (the spets?) 
lying above all these combinatorial data. Note the first attempts in this direction 
using fusion systems and $\ell$-compact groups by Kessar-Malle-Semeraro~\cite{semeraro, kms1, kms2}. 

Of course, if $W$ is a Weyl group (i.e. a finite rational reflection group), then 
$W$ is spetsial and one recovers the generic theory of unipotent representations 
of finite groups of the form $\Gb^F$ where $(q,\Gb,F) \in \groups(W)$\footnote{Note that, 
in this case, the $\z$-Harish-Chandra series depend only on the order $d$ of $\z$ and coincide 
with $d$-Harish-Chandra series}. In the late 80's, Lusztig associated to each 
finite Coxeter group which is not a Weyl group a combinatorial datum as in~(S1) and~(S2) 
satisfying a few axioms (this was finally published in 1993; see~\cite{lusztig h4}). 
For $H_2$, $H_3$ and $H_4$, this was rediscovered by Malle in 1992 (unpublished). 
In this case, the almost characters as in~(S4) were obtained for dihedral groups 
by Lusztig and for $H_4$ by Malle~\cite{lusztig malle}. 
About the same period, Malle~\cite{malle gd} 
proved that the imprimitive complex reflection groups $G(e,1,n)$ and $G(e,e,n)$ 
can be endowed with data satisfying~(S1),~(S2),~(S3),~(S4). 
The case of the primitive complex reflection groups 
has been done by Brou\'e-Malle-Michel~\cite{spetses 1, spetses 2}. 

\medskip

We expect that, for spetsial groups, all the above Brou\'e-Malle-Michel 
constructions are compatible with the geometry of the Calogero-Moser space 
associated with $W$ at {\it spetsial} parameter, and that all the conjectures 
stated in Section~\ref{sec:main} remain valid in this context. In other words, 
is the spets attached to $W$ hidden in the (Poisson) geometry of $\ZC_\ksp(W)$? 

In the upcoming section, we illustrate again these coincidences in 
the smallest non-cyclic primitive complex reflection group, namely the group $G_4$.

\bigskip

\section{A primitive example}

\bigskip

\boitegrise{{\bf Hypothesis.} {\it 
We assume in this section, and only in this section, that $W$ is of type $G_4$. 
In other words, we set
$$
s=\begin{pmatrix}
1 & 0 \\
0 & \zeta_3 \\
\end{pmatrix}
\qquad\text{and}\qquad 
t=\frac{1}{3}\begin{pmatrix}
2\z_3 + 1 & 2(\z_3 - 1) \\
\z_3 - 1 &  \z_3 + 2 \\
\end{pmatrix},$$
and we assume that $W=\langle s,t \rangle=G_4$. Here, $\z_3$ 
is a primitive third root of unity.}}{0.75\textwidth}

\bigskip

If $(\d,\b) \in \{(1,0),(1,4),(1,8),(2,1),(2,3),(2,5),(3,2)\}$, there is a unique 
irreducible character of $G_4$ of degree $\d$ and $b$-invariant $\b$: it will be denoted by $\phi_{\d,\b}$. 
We have
$$\Irr(G_4)=\{\phi_{\d,\b}~|~(\d,\b) \in \{(1,0),(1,4),(1,8),(2,1),(2,3),(2,5),(3,2)\}\}.$$
Note that $\phi_{1,0}=1$ is the trivial character, that $\phi_{1,4}=\e$, $\phi_{1,8}=\e^2$, 
that $\phi_{2,1}$ and $\phi_{2,3}$ are the characters afforded by the representations 
$V$ and $V^*$ respectively, that $\phi_{2,5}$ is the character afforded by 
$V \otimes \e \simeq V^* \otimes \e^2$ and that $\phi_{3,2}$ corresponds to the second 
symmetric power $\Srm^2(V) \simeq \Srm^2(V^*)$. 
We denote by $C_3$ the parabolic subgroup $\langle s \rangle$ of $W$: it is a cyclic 
group of order $3$. 

\bigskip
\def\petitespace{\vphantom{\DS{\frac{\DS{A}}{\DS{A}}}}}
\def\grandespace{\vphantom{\begin{pmatrix} \frac{A^1}{B_1} \\ \frac{A^1}{B_1} \end{pmatrix}}}

\subsection{Unipotent representations} 
All the facts stated without proof in this paragraph are taken from~\cite[\S{A.4}]{spetses 2} 
(see also~\cite{jean}). 
The set $\unip_\cus(C_3)$ contains a single 
element (which will be denoted by $\cus_{C_3}$) and the set $\unip_\cus(G_4)$ 
contains also a single element (which will be denoted by $\cus_{G_4}$). The (classical) 
$1$-Harish-Chandra theory of the spets $G_4$ may be summarized as follows:
\begin{itemize}
\item[$\bullet$] There are three Harish-Chandra series, namely 
$\unip(G_4,1,1)$, $\unip(G_4,C_3,\cus_{C_3})$ and $\{\cus_{G_4}\}$. 

\item[$\bullet$] $\unip(G_4,1,1)$ is the principal series, and we set 
$\r_{\d,\b}=\harish^{G_4}(\phi_{\d,\b})$.

\item[$\bullet$] We have $\Nrmov_{G_4}(C_3,\cus_{C_3})=\Nrmov_{G_4}(C_3) \simeq \mub_2$ and we 
set $\r_{C_3,+}=\harish^{G_4,C_3,\cus_{C_3}}(1)$ and $\r_{C_3,-}=\harish^{G_4,C_3,\cus_{C_3}}(\s)$, 
where $\s$ is the inclusion $\mub_2 \injto \CM^\times$. 
\end{itemize}
Therefore,
\equat\label{eq:unipotent-g4}
\unip(G_4)=\{\r_{1,0},\r_{1,4},\r_{1,8},\r_{2,1},\r_{2,3},\r_{2,5},\r_{3,2},
\r_{C_3,+},\r_{C_3,-},\cus_{G_4}\}.
\endequat
The unipotent Lusztig families are the following four subsets 
$\FG_\clubsuit$, $\FG_\diamondsuit$, $\FG_\heartsuit$, $\FG_\spadesuit$ of $\unip(G_4)$:
\equat\label{eq:familles-g4}
\begin{cases}
\FG_\clubsuit^\uni=\{\r_{1,0}\},\\ 
\FG_\diamondsuit^\uni=\{\r_{3,2}\},\\ 
\FG_\heartsuit^\uni=\{\r_{2,1},\r_{2,3},\r_{C_3,+}\},\\
\FG_\spadesuit^\uni=\{\r_{1,4},\r_{1,8},\r_{2,5},\r_{C_3,-},\cus_{G_4}\}.
\end{cases}
\endequat
%
The next table summarizes the 
numerical data attached to the spets $G_4$. Let us explain the information 
given in this table:
\begin{itemize}
\item[$\bullet$] The first column contains the list of irreducible unipotent representations 
$\r$ of the spets $G_4$.

\item[$\bullet$] The second column contains the degree of $\r$, where 
$\Phi_e$ denotes the $e$-th cyclotomic polynomial, 
$\Phi_3'=\qb-\z_3$, $\Phi_3''=\qb - \z_3^{-1}$, $\Phi_6'=\qb -\z_6$, $\Phi_6''=\qb-\z_6^{-1}$ 
and $\sqrt{-3}=2\z_3+1$.

\item[$\bullet$] The third column gives the family of $\r$.

\item[$\bullet$] The fourth (resp. fifth) column gives the cuspidal pair $(P,\l)$ parametrizing 
the $\z_4$-Harish-Chandra series (resp. $\z_6$-Harish-Chandra series) to which $\r$ belongs. Note that 
an empty box means that $\r$ is $\z_4$-cuspidal (resp. $\z_6$-cuspidal) and that $(1,1)_d$ 
denotes the $\z_d$-cuspidal pair associated with the trivial parabolic subgroup, 
which if $\t_d$-split \footnote{The interesting 
$\z$-Harish-Chandra 
series are those attached to a root of unity $\z$ 
of order equal to $1$, $2$, $3$, $4$ or $6$; the $\z=-1$ (resp. the $\z=\z_3$) case 
can be deduced from the $\z=1$ (resp. $\z=\z_6$) case thanks to Ennola duality~\cite[Axiom~5.13]{spetses 2}, 
which essentially amounts to replacing $\qb$ by $-\qb$ in this case.}. 
\end{itemize}

$$
\begin{array}{!{\vline width 2pt} c !{\vline width 1pt} c|c|c|c!{\vline width 2pt}}
\hlinewd{2pt}
\petitespace \r & \degb(\r) & \text{Family} & \text{$\z_4$-series} & \text{$\z_6$-series} \\
\hlinewd{1pt}
\r_{1,0} & 1 & \clubsuit & (1,1)_4 & (1,1)_6 \petitespace \\
\hline 
\r_{1,4} & -\sqrt{-3}/6 \qb^4 \Phi_3''\Phi_4\Phi_6'' & \spadesuit & & (1,1)_6 \petitespace \\
\hline
\r_{1,8} & \sqrt{-3}/6 \qb^4 \Phi_3'\Phi_4\Phi_6' & \spadesuit & &  \petitespace \\
\hline
\r_{2,1} & (3+\sqrt{-3})/6 \qb \Phi_3' \Phi_4 \Phi_6'' & \heartsuit & & (1,1)_6 \petitespace \\
\hline 
\r_{2,3} & (3-\sqrt{-3})/6 \qb \Phi_3'' \Phi_4 \Phi_6' & \heartsuit &  & \petitespace \\
\hline
\r_{2,5} & 1/2\qb^4 \Phi_2^2\Phi_6 & \spadesuit & (1,1)_4 & \petitespace \\
\hline
\r_{3,2} & \qb^2 \Phi_3 \Phi_6 & \diamondsuit & (1,1)_4 & \petitespace \\
\hline
\r_{C_3,+} & -\sqrt{-3}/3 \qb \Phi_1\Phi_2\Phi_4 & \heartsuit & & (1,1)_6 \petitespace \\
\hline
\r_{C_3,-} & -\sqrt{-3}/3 \qb^4  \Phi_1\Phi_2\Phi_4 & \spadesuit & & (1,1)_6 \petitespace \\
\hline
\cus_{G_4} & -1/2 \qb^4 \Phi_1^2 \Phi_3 & \spadesuit & (1,1)_4 & (1,1)_6 \petitespace \\
\hlinewd{2pt}
\end{array}
$$
We conclude this subsection by giving the parameters $k_{P,\l}$ for 
all $(P,\l) \in \Cus^d(G_4)$ and $d \in \{1,4,6\}$. Whenever the relative 
Weyl group $\Nrmov_{G_4}(P,\l)$ is cyclic and isomorphic to $\mub_d$, then the parameter 
will be given as a list $(k_0,k_1,\dots,k_{d-1})$ of complex numbers:
\begin{itemize}
\item[$\bullet$] $k_{1,1}=\ksp$.

\item[$\bullet$] $k_{C_3,\cus_{C_3}}=(3,0)$.

\item[$\bullet$] $k_{(1,1)_4}=(3,0,1,0)$.

\item[$\bullet$] $k_{(1,1)_6}=(2,0,0,1,0,1)$.
\end{itemize}

\bigskip

\subsection{Calogero-Moser space} 
As there is only one orbit of reflecting hyperplanes (call it $\O$), we will simply denote 
parameters $k \in \CM^{\aleph(G_4)}$ by a triple $(k_0,k_1,k_2) \in \CM^3$ 
where $k_j=k_{\O,j}$. For instance, $\ksp=(1,0,0)$. 
Descriptions of the Calogero-Moser space $\ZC_k(G_4)$ have been given 
in~\cite{bonnafe maksimau} and~\cite{bonnafe thiel}. Note 
that the descriptions given in both cases are for parameters $k=(k_0,k_1,k_2) \in \CM^3$ 
satisfying $k_0+k_1+k_2=0$: this is not restrictive, thanks to Remark~\ref{rem:parametres particuliers}. 
So, we set $\kspo=(2/3,-1/3,-1/3)$, and then $\ZC_\ksp(G_4)=\ZC_{\kspo}(G_4)$. 
Specializing the presentation~\cite{bonnafe thiel} at $\kspo$, we get that
$\ZC_\kspo$ is the closed subvariety of $\CM^8$ consisting of points $(x_1,x_2,y_1,y_2,a,b,c,e) \in \CM^8$ 
such that 
$$
\begin{cases}
ab + 12ce + 2x_1y_1 - 15e^4 + 234e^2 + 192e=0,\\
3 a y_1 e + 4 b c - 9 b e^3 + 126 b e + 2 x_1 y_2 =0,\\
3 a^2 e - 2 b x_2 + 8 c x_1 - 9 x_1 e^3 - 108 x_1 e =0,\\
4 a c - 9 a e^3 + 126 a e + 3 b x_1 e + 2 x_2 y_1 =0,\\
2 a y_2 - 3 b^2 e - 8 c y_1 + 9 y_1 e^3 - 108 y_1 e =0,\\
-a^3 - 3 a x_1 e^2 + 48 a x_1 + 2 a y_1^2 - b^3 + 2 b x_1^2 \\
\hphantom{AAAA} - 3 b y_1 e^2 + 48 b y_1 - 8 c x_2 - 8 c y_2 + 10 x_2 e^3 \\
\hphantom{AAAA} - 156 x_2 e + 128 x_2 + 10 y_2 e^3 - 156 y_2 e - 128 y_2 =0,\\
16 c^2 + 720 c e + 9 x_1 y_1 e^2 + 2 x_2 y_2 - 27 e^6 + 864 e^3 + 6804 e^2 =0,\\
-2 a y_1^2 + b^3 + 3 b y_1 e^2 - 48 b y_1 + 8 c y_2 - 10 y_2 e^3 + 156 y_2 e + 128 y_2 =0,\\
5 a^2 y_1 + 444 a b + 5 b^2 x_1 + 280 c e^3 + 4848 c e - 1280 c + 60 x_1 y_1 e^2 \\
\hphantom{AAAA} + 648 x_1 y_1 + 10 x_2 y_2 - 360 e^6 + 7200 e^3 + 88776 e^2 + 44928 e=0.\\
\end{cases}
$$
The action of $\CM^\times$ is given by
\equat\label{eq:action-c-g4}
\xi \cdot (x_1,x_2,y_1,y_2,a,b,c,e)=(\xi^4 x_1,\xi^6 x_2,\xi^{-4} y_1,\xi^{-6} y_2,\xi^2 a,\xi^{-2} b,c,e).
\endequat
An immediate computation shows that $\ZC_\kspo^{\CM^\times}$ contains $4$ points, given by
$$z_\clubsuit=(0,0,0,0,0,0,468,8),\qquad z_\diamondsuit =(0,0,0,0,0,0,0,0),$$
$$z_\heartsuit:=(0,0,0,0,0,0,-45,2)\qquad\text{and}\qquad
z_\spadesuit=(0,0,0,0,0,0,-18,-4).$$
We denote by $\FG_\bigstar^\calo$ the Calogero-Moser $\kspo$-family associated 
with $z_\bigstar$. Then
\equat\label{eq:cm-familles-g4}
\begin{cases}
\FG_\clubsuit^\calo=\{\phi_{1,0}\},\\
\FG_\diamondsuit^\calo=\{\phi_{3,2}\},\\
\FG_\heartsuit^\calo=\{\phi_{2,1},\phi_{2,3}\}\\
\FG_\spadesuit^\calo = \{\phi_{1,4},\phi_{1,8},\phi_{2,5}\}.\\
\end{cases}
\endequat
The comparison of~\eqref{eq:familles-g4} and~\eqref{eq:cm-familles-g4} proves 
the {\it spetsial} analogue of Conjecture~\ref{conj:familles-uni-cm}.

\bigskip

\subsubsection{Symplectic leaves of $\ZC_\kspo$} 
Let $\SC$ denote the singular locus of $\ZC_\kspo$. It has been computed 
in~\cite{bonnafe thiel} and it is proved there that it is irreducible 
of dimension $2$ and that
\equat\label{eq:singular-g4}
z_\clubsuit, z_\diamondsuit \not\in \SC\qquad\text{and}\qquad
z_\heartsuit, z_\spadesuit \in \SC.
\endequat
Moreover, $z_\spadesuit$ is the only singular point of $\SC$. Therefore, 
there are three symplectic leaves:
\begin{itemize}
\item The smooth locus: through the parametrization of Theorem~\ref{theo:leaves}, it corresponds to 
the pair $(1,p)$, where $p$ is the unique point of the Calogero-Moser space 
$\ZC_\kspo(0,1)$.

\item $\SC^\circ=\SC \setminus \{z_\spadesuit\}$: through the parametrization of 
Theorem~\ref{theo:leaves}, it corresponds to 
the pair $(C_3,q)$, where $q$ is the unique cuspidal point of the Calogero-Moser space 
$\ZC_\kspo(V/V^{C_3},C_3)$.

\item $\{z_\spadesuit\}$: it is cuspidal.
\end{itemize}
This parametrization fits perfectly with the partition of $\unip(G_4)$ 
into Harish-Chandra series, so this proves the spetsial analogue of Conjecture~\ref{conj:d-cusp} 
and Conjecture~\ref{conj:d-hc}(a) for $d=1$. 

Concerning Conjecture~\ref{conj:d-hc}(b), the only interesting case is the second one. 
Recall that $\Nrmov_{G_4}(C_3)=\mub_2$. It is proved in~\cite{bonnafe thiel} that 
\equat\label{eq:normalisation-g4}
\SC^\nor \simeq \{(x,y,e) \in \CM^3~|~(e-2)(e+4)=xy\} \simeq \ZC_{k_{C_3,\cus_{C_3}}}(\mub_2)
\endequat
as Poisson varieties. Note that the computation in~\cite{bonnafe thiel} 
is done for the parameter $-3\kspo$, so the equation given here is just 
obtained after a rescaling. This proves that Conjecture~\ref{conj:d-hc}(b) 
holds for $d=1$.

\subsubsection{Symplectic leaves of $\ZC_\kspo^{\mub_4}$} 
The action of $\CM^\times$ being given by~\eqref{eq:action-c-g4}, 
the variety $\ZC_\kspo^{\mub_4}$ is defined, inside $\ZC_\kspo$, by 
the equations $a=b=x_2=y_2=0$. 
This yields
\equat\label{eq:mu4-g4}
\ZC_\kspo^{\mub_4}=\{z_\heartsuit\} ~~\dot{\cup}~~ \SC_4,
\endequat
where
\equat\label{eq:s4-g4}
\SC_4 \simeq \{(x_1,y_1,e) \in \CM^3~|~4/3x_1y_1 = e(e-8)(e+4)^2\}
\endequat
So $\SC_4$ admits only one singular point (namely, $z_\spadesuit$) and so 
$\ZC_\kspo^{\mub_4}$ admits three symplectic leaves:
\begin{itemize}
\item There are two $\t_4$-cuspidal points, namely $z_\heartsuit$ and $z_\spadesuit$.

\item There is one $2$-dimensional symplectic leaf, which is the smooth locus of 
$\SC_4$ (i.e. $\SC_4 \setminus \{z_\spadesuit\}$). Through the parametrization of Theorem~\ref{theo:leaves}, 
it corresponds to the pair $(1,p)$, where $p$ is the unique point of the Calogero-Moser space 
$\ZC_\kspo(0,1)$.
\end{itemize}
This parametrization fits perfectly with the partition of $\unip(G_4)$ 
into $4$-Harish-Chandra series, so this proves the spetsial analogues of Conjectures~\ref{conj:d-cusp} 
and~\ref{conj:d-hc}(a) for $d=4$. 

Concerning Conjecture~\ref{conj:d-hc}(b), the only interesting case is the second one. 
Recall that $W_{\t_4} \simeq \mub_4$. The above description proves that $\SC_4$ (which is the closure 
of the symplectic leaf $\SC_4 \setminus \{z_\spadesuit\}$) is normal and that
\equat\label{eq:normalisation-mu4-g4}
\SC_4 \simeq \ZC_{k_{(1,1)_4}}(\mub_4)
\endequat
as Poisson varieties (for the Poisson bracket, see~\cite{bonnafe thiel}). 
So Conjecture~\ref{conj:d-hc}(b) holds 
for $d=4$.

\subsubsection{Symplectic leaves of $\ZC_\kspo^{\mub_6}$} 
The action of $\CM^\times$ being given by~\eqref{eq:action-c-g4}, 
the variety $\ZC_\kspo^{\mub_6}$ is defined, inside $\ZC_\kspo$, by 
the equations $a=b=x_1=y_1=0$. 
This yields
\equat\label{eq:mu6-g4}
\ZC_\kspo^{\mub_6}=\{z_\diamondsuit\} ~~\dot{\cup}~~ \SC_6,
\endequat
where
\equat\label{eq:s6-g4}
\SC_6 \simeq \{(x_2,y_2,e) \in \CM^3~|~x_2y_2 = (e-8)(e-2)^2(e+4)^3\}
\endequat
So $\SC_6$ admits two singular points (namely, $z_\heartsuit$ and $z_\spadesuit$) and so 
$\ZC_\kspo^{\mub_6}$ admits four symplectic leaves:
\begin{itemize}
\item There are three $\t_6$-cuspidal points, namely $z_\diamondsuit$, $z_\heartsuit$ and $z_\spadesuit$.

\item There is one $2$-dimensional symplectic leaf, which is the smooth locus of 
$\SC_6$ (i.e. $\SC_6 \setminus \{z_\heartsuit,z_\spadesuit\}$). Through the parametrization of 
Theorem~\ref{theo:leaves}, 
it corresponds to the pair $(1,p)$, where $p$ is the unique point of the Calogero-Moser space 
$\ZC_\kspo(0,1)$.
\end{itemize}
This parametrization fits perfectly with the partition of $\unip(G_4)$ 
into $6$-Harish-Chandra series, so this proves the spetsial analogues of Conjectures~\ref{conj:d-cusp} 
and~\ref{conj:d-hc}(a) for $d=6$. 

Concerning Conjecture~\ref{conj:d-hc}(b), the only interesting case is the second one. 
Recall that $W_{\t_6} \simeq \mub_6$. The above description proves that $\SC_6$ (which is the closure 
of the symplectic leaf $\SC_6 \setminus \{z_\heartsuit,z_\spadesuit\}$) is normal and that
\equat\label{eq:normalisation-mu6-g4}
\SC_6 \simeq \ZC_{k_{(1,1)_6}}(\mub_6)
\endequat
as Poisson varieties (for the Poisson bracket, see~\cite{bonnafe thiel}). 
So Conjecture~\ref{conj:d-hc}(b) holds 
for $d=6$.



\vspace{-1cm}

\part*{References}

\def\refname{~}

\DeactivateToc

\end{document}